\documentclass[english,a4paper,11pt]{amsart}
\RequirePackage[T1]{fontenc}
\RequirePackage{amsthm,amsmath}
\RequirePackage[numbers,sort&compress]{natbib}
\usepackage{amsthm,amsmath}
\usepackage{ulem}
\usepackage[a4paper]{geometry}

\def\XXint#1#2#3{{\setbox0=\hbox{$#1{#2#3}{\int}$}    \vcenter{\hbox{$#2#3$}}\kern-.5\wd0}}

\newcounter{gr1}

\newcounter{gr1n}
\newenvironment{numlistn}
{\begin{list} { (\roman{gr1n})}
{\usecounter{gr1n}
\setlength{\leftmargin}{0.9cm}
\setlength{\topsep}{0.1cm}
\setlength{\itemsep}{0.0cm}
\setlength{\parsep}{0.1cm}
\setlength{\itemindent}{-0.7cm}
\setlength{\parskip}{0.0cm}}}
{\end{list}}
\theoremstyle{plain}
\newtheorem{notation}{Notation}[section]
\newtheorem{theorem}{Theorem}[section]
\newtheorem{lemma}[theorem]{Lemma}

\newtheorem{proposition}[theorem]{Proposition}

\newtheorem{claim}[theorem]{Claim}
\newtheorem{convention}{Convention}[section]

\theoremstyle{definition}
\newtheorem{definition}[theorem]{Definition}
\newtheorem{assumption}[theorem]{Assumption}

\theoremstyle{remark}
\newtheorem{remark}[theorem]{Remark}

\numberwithin{equation}{section}

\newcounter{lil1}
\newenvironment{step}
{\begin{list} { \bf Step (\Roman{lil1})}
{ \usecounter{lil1}
\setlength{\leftmargin}{0.0cm}
\setlength{\topsep}{0.2cm}
\setlength{\itemsep}{0.0cm}
\setlength{\parsep}{0.1cm}
\setlength{\itemindent}{0.8cm}
\setlength{\parskip}{0.0cm}}}
{\end{list}}

\newcounter{lil33}

\newcounter{gr11}
\newenvironment{letters}
{\begin{list} { (\alph{gr11})$\, $} {\usecounter{gr11}
\setlength{\labelwidth}{-0.2cm} \setlength{\leftmargin}{0.4cm}
\setlength{\topsep}{0.1cm} \setlength{\itemsep}{0.2cm}
\setlength{\parsep}{0.1cm} \setlength{\itemindent}{0.4cm}
\setlength{\parskip}{0.0cm}}} {\end{list}}

\newcounter{lil1q}

\newcommand{\eps}{\vareps}

\newcommand{\MA}{\mathfrak{A}}

\usepackage{dsfont}
\usepackage{stmaryrd}
\usepackage{color}
\newcommand{\INT}{[0,T]}

\newcommand{\Law}{\mathcal{L}\mbox{aw}}

\newcommand\refy[1]{{}}

\newcommand{\Bspace}{X}

\newcommand{\bcal}{\mathcal{B}}

\newcommand{\shiftpro}{\pro^{\mbox{s}}} 
\newcommand{\Int}{(0,T]}
\newcommand{\AAA}{{U}}

\newcommand{\bNN}{{\bar{ \mathbb{N}}}}

\newcommand{\wi}[1]{\widetilde{#1}}

\newcommand{\zaehler}{\iota}

\newcommand{\baray}{\begin{array}{rcl}}
\newcommand{\earay}{\end{array}}
\newcommand{\barray}{\begin{array}{rcl}}
\newcommand{\earray}{\end{array}}

\newcommand{\levy}{L\'evy }

\definecolor{ballblue}{rgb}{0.13, 0.67, 0.8}
	\definecolor{ao}{rgb}{0.0, 0.0, 1.0}
    \definecolor{amethyst}{rgb}{0.6, 0.4, 0.8}

\newcommand\erika[1]{{\color{ballblue} #1}}

\newcommand\question[1]{{\color{amethyst} #1}}

\newcommand\dela[1]{}

\newcommand{\bcase}{\begin{cases}}
\newcommand{\ecase}{\end{cases}}

\newcommand\Leb{\mbox{\textup{Leb}}}

\newcommand{\WienH}{\mathcal{H}}

\newcommand{\DeltaA}{A}
\newcommand\del[1]{}

\del{

\newcommand\del[1]{}
}

\def\eps{\varepsilon}
   \newcommand{\zaehlerz}{\kappa}
    \newcommand{\vareps}{{\eps_\zaehlerz}}

\newcommand{\lk}{\left}
\newcommand{\lqq}{\lefteqn}
\newcommand{\rk}{\right}
\newcommand{\la}{\langle}

\newcommand{\ra}{\rangle}

\newcommand{\LL}{{\rm I \kern -0.2em L}}

\newcommand{\ep} {\vareps }

\newcommand{\be} {\begin{enumerate} }
\newcommand{\ee} {\end{enumerate} }

\newcommand{\CO}{{{ \mathcal O }}}

\newcommand{\CZ}{{{ \mathcal Z }}}

\newcommand{\CI}{{{ \mathcal I }}}
\newcommand{\CA}{{{ \mathcal A }}}

\newcommand{\CH}{{{ \mathcal H }}}
\newcommand{\CS}{{{ \mathcal S }}}
\newcommand{\CG}{{{ \mathcal G }}}

\newcommand{\CB}{{{ \mathcal B }}}

\newcommand{\CM}{{{ \mathcal M }}}
\newcommand{\CP}{{{ \mathcal P }}}
\newcommand{\BF}{{{ \mathbb{F} }}}
\newcommand{\CF}{{{ \mathcal F }}}

\newcommand{\CN}{{{ \mathcal N }}}
\newcommand{\CL}{{{ \mathcal L }}}

\newcommand{\RR}{{\mathbb{R}}}
\newcommand{\Rb}[1]{{\mathbb{R}_{#1}}}

\newcommand{\WW}{{\mathbb{W}}}

\newcommand{\DD}{\mathbb{D}}

\newcommand{\NN}{\mathbb{N}}

\newcommand{\PP}{{\mathbb{P}}}

\newcommand{\EE}{ \mathbb{E} }

\newcommand{\DEQS}{\begin{eqnarray*} }
\newcommand{\EEQS}{\end{eqnarray*} }
\newcommand{\DEQSZ}{\begin{eqnarray} }
\newcommand{\EEQSZ}{\end{eqnarray} }
\newcommand{\DEQ}{\begin{eqnarray}}
\newcommand{\EEQ}{\end{eqnarray}}

\newcommand{\Dcal} {{\mathcal D}}

\newcommand{\Fcal} {{\mathcal F}}
\newcommand{\Gcal} {{\mathcal G}}

\newcommand{\Kcal} {{\mathcal K}}

\newcommand{\Mcal} {{\mathcal M}}

\newcommand{\Ocal} {{\mathcal O}}

\newcommand{\Vcal} {{\mathcal V}}

\newcommand{\Xcal} {{\mathcal K}}

\newcommand{\Afrak} {{\mathfrak A}}

\newcommand{\R}{\mathbb{R}}
\newcommand{\N}{\mathbb{N}}

\renewcommand{\P}{\mathbb{P}}
\newcommand{\Eb}{\mathbb{E}}

\newcommand{\X}{\mathbb{X}}
\newcommand{\D}{\mathbb{D}}

\newcommand{\Pro}{\mbox{Proj}}
\newcommand{\kl}{\zaehler}

\usepackage{amsmath}
\usepackage{amssymb, color}
\usepackage{latexsym}
\usepackage{mathrsfs}
\usepackage{enumitem}

\usepackage[T1]{fontenc}
\usepackage{lipsum}
\usepackage{amsfonts}
\usepackage{graphicx}
\usepackage{epstopdf}
\usepackage{algorithmic}
\usepackage{bbm}

\usepackage{url}

\definecolor{DarkGreen}{rgb}{0.1,0.7,0.3}   
\definecolor{applegreen}{rgb}{0.55, 0.71, 0.0}

\AtBeginDocument{  }

\begin{document}

\title[A Schauder Tychonoff Theorem]{ 
A Schauder-Tychonoff fixed-point approach 
for nonlinear L\'evy driven reaction-diffusion systems 
}
\author[Hausenblas]{Erika Hausenblas}
\address{Montanuniversit\"{a}t Leoben\\
Department Mathematik und Informationstechnologie\\
Franz Josef Stra{\ss}e 18\\
8700 Leoben\\
Austria}
\email{erika.hausenblas@unileoben.ac.at}
\author[H\"ogele]{Michael A. H\"ogele}
\address{Universidad de los Andes\\
Departamento de Matem\'aticas, Facultad de Ciencias\\
Cra 1 No 18A - 12\\
111711 Bogot\'a\\
Colombia}
\email{ma.hoegele@uniandes.edu.co}
\author[Kistosil]{Fahim Kistosil}
\address{Department of Mathematics, Institut Teknologi Sepuluh Nopember,Surabaya East Java, Indonesia}
\email{kfahim@matematika.its.ac.id}
\date{\today}
\begin{abstract}
{
We show a stochastic version of the Schauder-Tychonoff fixed point theorem 
which yields a solution of the martingale problem for a class of systems of nonlinear reaction-diffusion equations driven by a cylindrical Wiener process and a Poisson-random measure with certain moments. 
By this type of theorem one can solve systems by linearization which have 
a possibly unbounded, non-dissipative and non-coercive nonlinearity.
}
\end{abstract}

\keywords{stochastic Schauder-Tychonoff type theorem,  nonlinear stochastic partial differential equation.}
\subjclass[2010]{Primary 35K57, 60H15; Secondary 37N25, 47H10, 76S05, 92C15}

\thanks{}

\maketitle

Details of the authors:
\begin{itemize}
	\item  {Erika Hausenblas},{Montanuniversit\"{a}t Leoben\\
		Department Mathematik und Informationstechnologie\\
		Franz Josef Stra{\ss}e 18\\
		8700 Leoben\\
		Austria}
\url{erika.hausenblas@unileoben.ac.at}
\item 
{Michael A. H\"ogele}
{Universidad de los Andes\\
	Departamento de Matem\'aticas, Facultad de Ciencias\\
	Cra 1 No 18A - 12\\
	111711 Bogot\'a\\
	Colombia}
\url{ma.hoegele@uniandes.edu.co}
\item Fahim Kistosil,
Department of Mathematics,\\
 Institut Teknologi Sepuluh Nopember,
 \\
  Surabaya East Java, Indonesia
\\
\url{kfahim@matematika.its.ac.id}
\end{itemize}

{\color{black} 
\section{Introduction}

Systems of reaction-diffusion equations describe the temporal and spatial evolution of the vector of concentrations $u = (u_1, \dots, u_n)$, $u_i = u_i(t, x)$, $t\geq 0$, $x\in \mathcal{O}$ of the reactants of a given chemical reaction in a $d$-dimensional domain $\mathcal{O}$ by 
\begin{equation}\label{e:deteq}
\partial_t u = A u + f(u) \qquad \mbox{ with initial data } u(0) = u_0, 
\end{equation}
where $Au$ typically represents the multidimensional (possibly nonlinear generalization of the)
spatial diffusion term given by the Laplacian according to Fick's second law of diffusion and $f$ a local reaction term. 
The nonlinear function $f(u) = (f_1(u), \dots, f_n(u))$ is a $n$-dimensional polynomial in the $n$ components of $u$. In one dimension famous  examples of nonlinear reaction terms include the Kolmogorov-Petrovsky-Piskunov (KPP) equation of population dynamics also known as Fisher's equation in evolution theory with $f(u) = c u (1-u)$ \cite{KPP37, Fi37}, the Newell-Whitehead equations of finite bandwidth Raleigh-B\'enard convection with $f(u) = u(1-u^2)$ \cite{NW69}, 
or the Zeldovich-Frank-Kamenetzkii equation \cite{ZFK38} of combustion, which exhibits $f(u) = u(1-u) e^{-\beta(1-u)}$ an exponential damping term. 
Two-component systems include the linear models derived by A. Turing in its foundational article \cite{Tu52} on pattern formation. Important nonlinear examples are the Fitzhugh-Nagumo system of action potentials in a nerve cell \cite{FH61, NAY62} and the Gierer-Meinhard system \cite{GM72}. Higher-order equations are given by the Huxley-Hodgkin model of the neuronal action potential \cite{HH52}, Patlak-Keller-Segel equations of chemotaxis \cite{KS70} and the Gray-Scott system of an autocatalytic chemical reaction~\cite{GS83}. There is a huge amount of literature on the solutions of related equations, which we cannot review here due to sheer extensions; important references certainly include classic texts as \cite{Ev10, Fr64, JLL69, Kr87, Ro84, SY02, Te97} among many others. 

When a stochastic term $\xi = (\xi_1, \dots, \xi_d)$, $\xi_i = \xi(x, t, \omega)$ is added, the system 
\[
d u = (A u + f(u)) dt + d\xi(x, t) \qquad \mbox{ with initial data }\quad  u(0) = u_0, 
\]
turns into a system of stochastic partial differential equations, which faces all sorts of additional challenges and adds substantial complexity due to the difficulties of the stochastic integration theory in infinite dimensions combined with the nonlinearity~$f$. For a rigorous  derivation of stochastic partial differential equations of this type by mesoscopic limit theorems we refer to \cite{Ko08}. 

In the main applications of interest in this article the algebraic Nemitskii type nonlinearity $f(u)$ exhibits a polynomial coupling. More precisely, in many situations, the reaction term $f_i$ of the $i$-th component on the right-hand side of \eqref{e:deteq} consists of multidimensional (generalized) polynomials, that is, with possibly non-integer exponents. 
If in $f_i$ the pre-factor of the monomial with the highest power of of $u_i$ is non-positive  
the component is called an inhibitor, otherwise an activator. 
It is well-known that such models with only inhibitor terms are non-Lipschitz but dissipative. 
However, if an activator term exists, they do not satisfy any dissipativity or coercivity conditions.  

 Due to the polynomial type dependence of the equations, the nonlinear right-hand side $f$ retains good continuity properties in the correct spaces and the diffusive term $A$ yields good compactness properties, which make it prone (a priori in the deterministic case) for the application of the Schauder-Tychonoff fixed point theorem~\cite[Theorem 2.A]{zeidler1}, see also \cite{granas}. 

\begin{theorem}[Schauder-Tychonoff fixed point theorem]\label{thm:Schauder}
Consider a Banach space $\mathbb{X}$ and let $K\subset \mathbb{X}$ be a nonempty closed convex subset. If $\Upsilon: K \to K$ is continuous with a pre-compact image set, then the operator $\Upsilon$ has a fixed point. 
\end{theorem}

While there are {many} different kinds of stochastically strong solutions, such as classical solutions \cite{DPR04},  mild solutions \cite{Wa86}, Dirichlet solutions \cite{AR91}, variational solutions \cite{DPR04}, or solutions via rough paths and regularity structures \cite{Ha14} constructed on a given probability space, stochastically weak solutions introduced in \cite{Vio76} only reflect the statistics (laws) of the respective processes as solutions of a respective martingale problem, see \cite{Me88b}. 
We refer to the exhaustive and recent review of the literature on reaction diffusion equations 
with unbounded nonlinearity driven by Gaussian and L\'evy noise in \cite[on p.132-133]{reacdiff}. 
The references there are \cite{BG99, zmtypep, BHZ13, Brz+Haus_2009, BZ10, Ce03, CR04, DaPrZa, DPR04, DHI11, DHI13, DX07, DX09, FS13, GG94, Erika05, Erika07, HG13, JPW10, LR10, roecknerwei, MR06, MR10, My02, Peszat95, Peszat_Z_2007, RZ07, TW03, TW06}. A review of the exiting literature yields, that apart from the recent manuscripts \cite{EAkash22} and \cite{jonas} the existence is shown by finite dimensional approximations and compactness arguments with in comparison with the martingale representation theorem by \cite{dettweiler}, and strong solutions by the Banach fixed point theorem, Dirichlet forms, or other means, but not the Schauder-Tychonoff fixed point theorem.

In this article, we establish the existence of a weak solution of a general class of systems of stochastic reaction-diffusion equations with multiplicative noise of the following type 
\begin{equation}\label{e:SPDE}
dX = (A X + f(X)) dt + g(X) dL, \qquad \mbox{ with initial data }\quad  X(0) = X_0
\end{equation}
with respective (random) inital conditions with the help of Theorem~\ref{thm:Schauder}. In such a setting, the finite-dimensional approximation methods for strong or weak solutions become increasingly untractable for multi-component systems and rather inefficient since the respective methods and similar types of reasoning have to be carefully adapted and repeated in each of the settings. A Schauder-Tychonoff fixed point theorem approach for strong solutions requires to identify compact sets in path space, which turns out to be rather involved; we refer to \cite{bally}, where this approach has been carried out for special situations with Gaussian noise, which is one reason why we focus on weak solutions instead.

In \cite{jonas}, a stochastic version of the Schauder-Tychonoff theorem is part of a strategy to obtain  
the existence of a martingale solution to the stochastic Klausmeier system with Wiener noise. Even though the Schauder-Tychonoff argument was only part of the strategy, the theorem was applied to a nonlinear SPDE the first time. In \cite{EAkash22}, a more direct approach has been implemented to show the existence of a weak solution of the stochastic Gierer-Meinhardt system with Wiener noise. Here, the theorem is applied directly with the following parameters 
\begin{align*}
&n=2, \mbox{ all } c_{ij}>0, A = (c_{11} A_1, c_{21} A_2), f = (f_1, f_2), g = (g_1, g_2),
L = (W_1, W_2)\\
&A_1 = A_2 = \Delta \mbox{ the standard Laplacian with Neumann conditions }\\
&f_1(u) = c_{12}\, \frac{u_1^2}{u_2} - c_{13} u_1, \qquad f_2(u) = c_{22} u_1^2 - c_{23} u_2,\qquad g_j(u) = c_{j4} u_j, j=1,2\\
&(W_1, W_2) \mbox{ a pair of independent $Q_j$-Wiener processes with values in $L^2(\mathcal{O})$}\\
&\qquad \mbox{and covariance operator $Q_j$, $j=1,2$, respectively.}
\end{align*}
That is to say, roughly speaking, we define for a particular, well-chosen class of processes $\eta \in \mathbb{X}$ the solution of the controlled equation 
\[
dX^\eta = (A X^\eta + F(\eta)) dt + g(X^\eta) dL, 
\]
and define the fixed-point operator $\Upsilon(\eta) := X^{\eta}$ on the laws of the processes of $\eta$. In a second step, we construct a convex bounded set $K \subset \mathbb{X}$ and show that $\Upsilon$ maps  $K\to K$,  is continuous on $K$ with respect to the topology on $\mathbb{X}$, and that $\Upsilon(K)$ is relative compact. Again, for special models and Wiener perturbations $L = W$, this method is implemented in \cite{EAkash22} and~\cite{jonas}.

The purpose of this article is two-fold: 
\begin{enumerate}
 
 \item We standardise and modularise the proof of the existence of weak solutions 
by the Schauder-Tychonoff fixed-point method for the laws of respective parametrized solutions. 
In particular, we clarify the precise role of the different spaces and their topologies in an abstract setting, which can be used by a large class of systems of type \eqref{e:SPDE}. 

 \item We generalize the stochastic Schauder-Tychonoff theorem (see \cite{jonas}) to drivers of an independent Poisson random measure, which are crucial, for instance, in biophysics applications. 
\end{enumerate}
}
This method is suitable for application in different settings such as bi-domains, cross-diffusions,  or neurology. Our method complements the existing literature and provides an alternative existence proof strategy which looks flexible enough to be adapted to rather complex settings. 

Let us introduce some basic notation used throughout the article. 
\begin{notation}\label{notationref}
$\RR$ denotes the real numbers, $\RR^+:=\{ x\in\RR:x>0\}$ and $\RR^+_0:=\RR^+\cup\{0\}$. By $\mathbb{N}$ we denote the set of natural numbers (including $0$) and by $\bar{\mathbb{N}}$ we denote the set $\mathbb{N}\cup\{\infty\}$.
\end{notation}


{\begin{definition}\label{notationref3}
A measurable space $(Z,\CZ)$ is called Polish if there exists a metric $\varrho$ on $Z$ such that $(Z,\varrho)$ is a complete separable metric space and $\CZ=\mathscr B(Z,\varrho)$.
\end{definition}}
\begin{notation}\label{notationref4}
{
The set of all {finite non-negative} measures on a Polish space $(Z,\CZ)$ will be denoted by $M_+(Z)$ and $\CP_1(Z)$ will stand for probability measures on $\CZ$. If a family of sets $\{Z_n\in\CZ:n\in\NN\}$ satisfy $Z_n\uparrow Z$ then $M_{\bar \NN}(\{Z_n\})$ denotes the family of all $\bar{\mathbb{N}}$-valued measures $\theta$ on $\CZ$ such that $\theta(Z_n)<\infty$ for every $n\in\Bbb N$. By $\CM_{\bar \NN}(\{Z_n\})$ we denote the $\sigma$-field on $M_{\bar \NN}(\{Z_n\})$ generated by the functions $i_B:M_{\bar \NN}(\{Z_n\})\ni\mu \mapsto \mu(B)\in \bNN$, $B\in \CZ$.
}
\end{notation}

\bigskip 

\section{Stochastic Preliminaries}\label{preliminares}

In this section, we introduce the stochastic setting. In particular, we introduce the cylindrical Wiener process, the Poisson random measure, and the L\'evy process, and the necessary definitions. 

Let $\mathfrak{A}=(\Omega,\CF,\mathbb{F},\PP)$ be a filtered probability space satisfying the so-called usual conditions, see \cite{Protter}, i.e.
\begin{trivlist}
	\item[(i)] $\mathbb{P}$ is complete on $(\Omega, \mathcal{F})$,
	\item[(ii)] for each $t\in \mathbb{R}_+$, $\mathcal{F}_t$ contains all $\mathbb{P}$-null sets of $\mathcal{F}$,
	\item[(iii)] the filtration $\mathbb{F}$ is right-continuous.
\end{trivlist}

\subsection{The cylindrical Wiener process}\label{ss:cylindrical}

Let $\WienH$  be a separable Hilbert space over $\RR$.
Let $W$ a cylindrical Wiener process over $\mathfrak{A}$ taking values in $\WienH$.
Due to the spectral decomposition, we know that the Wiener process can be represented as follows
\DEQ\label{eq:wiener}
W(t):=\int_0^ t \sum_{k=1}^\infty \phi_kd\beta_k(t),\quad t\ge 0,
\EEQ
where $\{ \beta_k:k\in\NN\}$ is a family of independent Brownian motions over $\mathfrak{A}$ 
and $\{ \phi_k:k\in\NN\}$ is an orthonormal basis in $\WienH$. 
It is know that this representation is not a loss of generality, see \cite[Proposition 4.7, p. 85]{DaPrZa}.

\bigskip 

\subsection{Time-homogeneous Poisson random measures}\label{ss:PRM}

Since the definition of time-homogeneous Poisson random measure is introduced in many, not always equivalent ways, we give {our} definition here. Note that the Poisson random measure is constructed by a given non--negative measure, i.e.\ a non-negative measure which can be infinite, alled L\'evy measure and will be denoted in the sequel by $\nu$. Let $(Z,\CZ)$ be a Polish space with a fixed metric $d$. We suppose the measure is $\sigma$-finite, hence we know there exists a family of sets {$\{Z_n\in\CZ\}$ such that $Z_n\uparrow Z$ and $\nu(Z_n)<\infty$ for every $n\in\Bbb N$.
The $\sigma$--finiteness is not only essential to show the existence of a Poisson random measure, 
it is also necessary to be able to switch between different probability spaces.

\begin{definition}\label{def-Prm}(see \cite{ikeda}, Def. I.8.1)
Let $(Z,\CZ)$  be a Polish space, 
$\nu$ a $\sigma$--finite measure on $(Z,\CZ)$ and $T>0$. 
A
{\sl time-homogeneous Poisson random measure} $\eta$ 
over a filtered  probability space $(\Omega,\CF,\BF,\PP)$, where $\BF=(\CF_t)_{t\in\INT}$,  is a measurable function
$$\eta: (\Omega,\CF)\to ({ M_{\bar{\Bbb N}}(\{Z\times\Int \}),\mathcal M_{\bar{\Bbb N}}(\{Z\times\Int \})}) 
$$
such
that
{
\begin{trivlist} \item[(i)]
for each $B\in  \CZ \otimes
\mathcal{B}({\mathbb{R}^+}) $ with $\EE \eta(B) < \infty$
 $\eta(B):=i_B\circ \eta : \Omega\to \bar {\mathbb{N}} $ is a Poisson random variable with parameter 
 $\EE \eta(B)$, otherwise  $\eta(B)=\infty$ a.s.
\item[(ii)] $\eta$ is independently scattered, i.e.\ if the sets $
B_j \in   \CZ\otimes \mathcal{B}({\mathbb{R}^+})$, $j=1,\cdots,
n$, are  disjoint,   then the random variables $\eta(B_j)$,
$j=1,\cdots,n $, are mutually independent;
\item[(iii)] for each $U\in \CS$, the $\bar{\mathbb{N}}$-valued
process $(N(t,B))_{t\in \INT }$  defined by
$$N(B, t):= \eta(B \times (0,t]), \;\; t\in \INT $$
is $\BF$-adapted and its increments are stationary and independent of the past,
i.e.\ if $t>s\geq 0$, then $N(B, t)-N(B, t)=\eta(B \times (s,t])$ is
independent of $\mathcal{F}_s$.
\end{trivlist}
}
\end{definition}


\begin{remark}\label{compensator}
In the framework of Definition \ref{def-Prm}  the assignment
$$
\nu: \CZ \ni  B \mapsto \mathbb{E}\big[\eta( B\times(0,T])\big] 
$$
defines a  uniquely determined measure, called in the following {\sl intensity measure}.
\end{remark}

\subsection{L\'evy processes}\label{rel-l-prm}
Given a {pure-jump} L\'evy process, one can construct a corresponding Poisson random measure. Vice versa,
given a Poisson random measure, one easily gets a corresponding L\'evy process. To illustrate this fact, let us recall start with the definition of a L\'evy process.
 
\begin{definition} Let $E$ be a Banach space. 
A stochastic process $L=\{ L(t):t\ge 0\}$ over a probability space $(\Omega,\CF,\PP)$
is an $E$-valued L\'evy process if the following conditions are satisfied.
\begin{numlistn} 
\item $L_0=0$ a.s. 
\item For any choice $n\in\NN$ and $0\le t_0<t_1<\cdots t_n$, the random vectors
$L(t_0)$, $L(t_1)-L(t_0)$, $\ldots$, $L(t_n)-L(t_{n-1})$ are independent.
\item For all $0\le s<t$, the law of $L(t+s)-L(s)$ does not depend on $s$.
\item The trajectories of $L$ are a.s. c\`adl\`ag on $E$. 
\end{numlistn} 
Let $\BF=\{\CF_t\}_{t\ge 0}$ be a filtration on $\CF$. We say that
$L$ is a L\'evy process over $(\Omega,\CF,\BF,\PP)$, if $L$ is an $E$--valued and $\BF$-adapted L\'evy process.
\end{definition}

The characteristic function of a L\'evy process is uniquely defined and is given by the L\'evy-Khinchin formula. In particular, if $E$ is a Banach space  of type $p$ and $L= \{ L(t):t\ge 0\}$ is an $E$-valued L\'evy process, there exist a positive symmetric operator $\mathcal{Q}:E ^ \prime \to E$, a non-negative measure
$\nu$ concentrated on $E\setminus \{0\}$ such that $\int_E (1\wedge |z|^p)\,\nu(dz)<\infty$, and an element $m\in E$ such that (we refer e.g.\ to \cite{apple, Gine, linde})
\DEQS
 \lqq{ \EE e^{i\la L(1),x\ra}=  \exp\Big( i \la m,x \ra -\frac 12 \la \mathcal{Q}x,x\ra }&&
\\&&
-\int_E \lk( 1-e ^ {i\la y,x\ra }+1_{(-1,1)}(|y|)i\la y,x\ra \rk)\nu(dy)  \Big),\quad x\in E^\prime .
\EEQS
The measure $\nu$ is called {characteristic measure} of the L\'evy process $L$. A L\'evy process is of {pure jump type}  iff $\mathcal{Q}=0$. Moreover, the triplet $(\mathcal{Q},m,\nu)$ is uniquely determined and characterizes the law of the L\'evy process.

Now, starting with an $E$--valued L\'evy process over a filtered probability space $(\Omega,\CF,\BF, \PP)$
one can construct an integer valued random measure as follows.
For each $ (B,I)\in \CB(\RR)\times \CB(\RR_+)$ let
$$
 \eta_L(B\times I ) := \#\{s\in I \mid \Delta_s L\in B\}\in\NN_0\cup\{\infty\}.\footnote{The {jump process} $\Delta X = \{\Delta _tX,\,0\le t<\infty\}$
of a process $X$ is defined by
$\Delta_t X := X(t) - X(t-)= X(t) - \lim_{s\uparrow  t} X(s)$,
$t> 0$ and $\Delta _0X:=0$.}
$$
If $E=\RR ^ d$, it can be shown that $\eta_L$ defined above is a time-homogeneous Poisson random measure
(see Theorem~19.2 \cite[Chapter 4]{sato}).

\medskip
Vice versa,
let $\eta$ be a time-homogeneous Poisson random measure on a Banach space $E$ of type $p$, $p\in (0,2]$.
Then the integral (Dettweiler \cite{dettweiler})
$$ 
t\mapsto \int_0 ^ t \int_Z  z\,\tilde \eta(dz,ds)
$$
is well-defined, if the intensity measure is a L\'evy measure, whose definition is given in Appendix~\ref{levy}. 

With respect to Poisson random measures or L\'evy processes, one can define a stochastic integral in different settings. In particular, one can assume that the integrand is predictable or only progressively measurable. Here, we will use the second setting, where the integrand is supposed to be progressively  measurable.

Let us recall that
a process $\xi:[0,T]\times \Omega\times Z\to E$ is $\mathbb{F}\otimes \mathcal{Z}$-progressively measurable, if $\xi$ is $\mathcal{BF}\otimes \mathcal{Z}/\mathcal{B}(E)$-measurable, where, see \cite[Section 6.5]{wentzell}, $\mathcal{BF}$ is the $\sigma$-field consisting of all sets $B\subset  [0,T]\times \Omega$ such that for every $t\in [0,T]$, the set $B\cap \big( [0,t]\times \Omega)$ belongs to the sigma field $\mathcal{B}_{[0,t]}\otimes \mathcal{F}_t$. Note that $\mathcal{BF}\otimes \mathcal{Z}$ is the $\sigma$-field generated by a family of all sets $B\subset  [0,T]\times \Omega \times Z$ such that for every $t\in [0,T]$, the set $B\cap \big( [0,t]\times \Omega \times Z)$ belongs to the sigma field $\mathcal{B}_{[0,t]}\otimes \mathcal{F}_t \otimes \mathcal{Z}$.\\
For $p\in [1,\infty)$, the set of all $p$--integrable $\mathcal{BF}\otimes \mathcal{Z}$-progressively processes $\xi:[0,T]\times \Omega\times Z\to E$ will be denoted by 
\[
\mathcal{M}^{p}([0,T]\times Z; \mathcal{BF}\otimes \mathcal{Z};E)
\] 
and the Banach space of all equivalence classes of $p$--integrable $\mathcal{BF}\otimes \mathcal{Z}$-progressively processes $\xi:[0,T]\times \Omega\times Z\to E$ will be denoted by 
\[
\mathbb{M}^{p}([0,T]\times Z; \mathcal{BF}\otimes \mathcal{Z};E).
\]

In \cite{janmaxreg}  it  is proven that for any Banach space $E$ of $M$--type $p$ there exists a unique
continuous linear operator $I$ which associates   to  each progressively measurable
process $\xi: {\mathbb{R}}_{+} \times \Omega \to L^p(Z,\nu;E) $  with $\PP$-a.s.\
\begin{equation} \label{cond-2.01}
\int_{0}^{T} \int_{Z} \vert \xi (r,z) \vert_E^{p} \, \nu (dz) dr   < \infty,
\end{equation}
for every $T>0$, an adapted $E$-valued c\`{a}dl\`{a}g process
$$
{I}_{\xi , \tilde{\eta }} (t):=\int_{0}^{t} \int_{Z} \xi (r,z) \tilde{\eta } (dz, dr), \quad t\in [0, T], 
$$
such that if a process $\xi$ satisfying  the above condition (\ref{cond-2.01}) is a
random step process with representation
\DEQSZ\label{kkk}
\xi(r,z) = \sum_{j=1} ^n 1_{(t_{j-1}, t_{j}]}(r) \, \xi_j(z),\quad z\in Z,\quad  r\in \INT ,
\EEQSZ
where $\{t_0=0<t_1<\ldots<t_n<\infty\}$ is a finite partition of $[0,\infty)$ and
for all $j\in \{ 1, \ldots ,n \}$, $\xi_j$ is an $E$-valued $\CF_{t_{j-1}}$--measurable $p$-summable simple  random variable, then
\begin{equation} \label{eqn-2.02}
 I_{\xi,\tilde{\eta }} (t) 
=\sum_{j=1}^n  \int_Z  \xi_j (z) \,\tilde \eta \lk(dz, (t_{j-1}\wedge t, t_{j}\wedge t] \rk),\quad   t\in \INT .
\end{equation}
This definition can be extended to all progressively measurable mappings $\xi:\Omega\times \INT \times Z\to E$ with $\PP$--a.s.
\DEQS
\int_0^ T \int_Z\min(1,|\xi(r,z)|^p_E) \nu(dz)\, dr<\infty .
\EEQS
More information on the different settings is given in \cite{ruediger}.

\bigskip 

\section{The stochastic Schauder-Tychonoff type theorem}\label{schauder}

In this section we first present the setting, then we state the main result, and finally we give the proof. Let us fix some notation.


Let $\Afrak=(\Omega,\Fcal,\mathbb{F},\P)$ be a filtered probability space
with filtration $\mathbb{F}=(\Fcal_{t})_{t\in [0,T]}$ satisfying the usual conditions introduced in Section~\ref{preliminares}.
Let $\WienH$ be a separable Hilbert space and $\{W(t):{t\in [0,T]}\}$ be a 
on $\WienH$  Wiener process with covariance $\mathcal{Q}$ over $\MA$ {defined in Subsection~\ref{ss:cylindrical}} and 
$\eta $ be a time-homogeneous Poisson random measure with intensity measure $\nu$ over $\MA$ described before.

\medskip

\begin{assumption}\label{standingassumption}
Let us  assume that
\begin{itemize}
\item Let$U$ and $\Bspace $ be some Besov spaces over a bounded domain $\CO\subset \RR^d$. 
\item If a Wiener process is involved, we suppose that both spaces are of  UMD type $2$, if only a L\'evy process is involved, then we suppose that both spaces are of  UMD  type $p$, where $p\ge 1$ satisfies  \eqref{heregrow}.
\item Let $\Ocal\subset \R^d$ be an open domain, $d\ge 1$. Let $\X\subset\{\xi:[0,T]\to \Bspace \subset\Dcal^{\prime}(\Ocal)\}$
be the Banach function space\footnote{Here, $\Dcal^\prime(\Ocal)$ denotes the space of Schwartz distributions on $\Ocal$, that is, the topological dual space of smooth functions with compact support $\Dcal(\Ocal)=C_0^\infty(\Ocal)$.} given by $L^m(0,T;X)$, where $m\ge 1$, and let $\X^{\prime}\subset\{\xi:[0,T]\to \Bspace \subset\Dcal^{\prime}(\Ocal)\}$
be a reflexive Banach function space embedded compactly into $\X$. 
\item In particular, let $\mathbb{X}':=L^m(0,T;X')\cap \mathbb{W}^{\alpha}_m(0,T;X_0)$\footnote{For the definition of $\mathbb{W}^{\alpha}_m(0,T;X_0)$ see Appendix \ref{dbouley-space}}, where $\alpha>0$, $X'\hookrightarrow X$ compactly, and $X\hookrightarrow X_0$ continuously.
In both cases, the trajectories take values in a Banach function space $\Bspace $ over the spatial domain $\Ocal$.

\item  $\DeltaA:D(\DeltaA)\subset \Bspace \to \Bspace $ is a possibly nonlinear and measurable (single-valued) operator, defined on a Gelfand triple
$V\hookrightarrow X  \hookrightarrow  V^\ast$.  In particular, $X$ is a separable Hilbert, $V$ a reflexive Banach space. We assume that the operator satisfies the setting given in Theorem~5.1.3  \cite[p.\ 125]{roecknerwei} with a given $\alpha>1$.
In case, where e.g.\ $A$ is a unbounded linear generating an analytic semigroup, we have  $V := V_\frac{1}{2}$,  {where $V_\gamma := [X,D(A)]_\gamma$, $\gamma\in (0,1)$ is the standard real interpolation space of exponent $\gamma$ between $X$ and the domain $D(A)$ of $\DeltaA$, see \cite[Section~2.4]{bergh}.}  In addition, let $\mathbb{X}:= L^\alpha(0,T;V)$.

\item   $F:V_\gamma\times [0,T]\to X $, $\gamma<1$, a (strongly) measurable map. 

\item 
{$\Sigma:V\times [0,T]  \to L(\mathcal{H},\Bspace )
 $
 such that there exists a constant $C>0$ such that we have
 $$
 |\Sigma (u,t)|^2_{L(\mathcal{H},X)}\le C
{ |u|^2_{V} }
 ,\quad t\in[0,T],\quad  u\in V.
 $$}
\item {a function $g: V \times Z \to \Bspace $ such that 
  there exists a constant $C>0$ such that we have
 $$
\int_Z |g (u,z)|^2_{X}\nu(dz)\le C { |u|^2_{V}} 
,\quad u\in V.
 $$}
 \item 
and a UMD-Banach space $U$ and $U'$, where $U'\hookrightarrow U$ compactly,  of type $2$, such that for all $u\in X$, we have 
 $|A u|_{ U'}$, $|F(u,t)|_{ U'}\le C$,
 $$
\lk|\Sigma(u,t)\rk|_{\gamma(\mathcal{H},U')}\le C  \quad \mbox{and}\quad \int_Z |g(u,z)|_{U'}^l\nu(dz)\le C\,\mbox{ for all } l=m,2,\,t\in[0,T].
 $$
Set $\mathbb{U}:=\mathbb{D}(0,T;U)$\footnote{Here, $\mathbb{D}(0,T;U)$ denotes the Skorokhod space of all c\'adl\'ag functions over $[0, T]$ with values in $U$.}.
\end{itemize}
\end{assumption}

For $m\ge 1$ we define the space of processes
\begin{equation}\label{eq:MMdef}
\begin{aligned}  \Mcal_{\Afrak}^{m}(\X)
:= & \Big\{ \xi:\Omega\times[0,T]\to \Bspace \;\colon\;\\
&\qquad\text{\ensuremath{\xi} is \ensuremath{\mathbb{F}}-progressively measurable}\;\text{and}\;\Eb|\xi|_{\X}^{m}<\infty\Big\}
\end{aligned}
\end{equation}
equipped with the semi-norm
\[
|\xi|_{\Mcal_{\Afrak}^{m}(\X)}:=(\Eb|\xi|_{\X}^{m})^{1/m},\quad\xi\in\Mcal_{\Afrak}^{m}(\X).
\]

\begin{remark}
Given a $Z$--valued L\'evy process the assumption can be easily rewritten. {Let
$G : V \to L(Z, X)$ be the coefficient of the jump term.
In particular, we get for $g$ 
\DEQSZ\label{coefflevy}
g(x,z) &:=&  G(x)z,\quad z\in Z,\,x\in X.
\EEQSZ}
\end{remark}
\begin{remark}
    The exact space is given by the properties of the operator $A$ and the parameter $\alpha$. In fact, the coercivity of $A$  in the variational setting gives the exponent $\alpha$ and the space $V$. This is satisfied in most applications, however, it is not used directly.
\end{remark}
The aim is to provide a tool for proving the existence of a weak solution in the variational PDE sense 
to the following stochastic system 
\DEQSZ\label{spdesproblem}
\lqq{ dw(t) =\lk(\DeltaA w(t)+ F(w(t),t)\rk)\, dt }
&&
\\ &&{}+\Sigma(w(t),t)\,dW(t)+ \int_Z g(w(t),z)\tilde \eta(dz,dt),\quad w(0)=w_0\in V^\ast .\notag 
\EEQSZ

\medskip
\noindent Since the term strong solution can refer to different notions in the PDE and probabilistic settings, we provide a precise definition here.
    \begin{definition}[Strong solution in the probabilistic sense]\label{def:singleSLN}\label{eq:sol-def}
	Consider a filtered probability space
	\begin{equation}\label{eq:filt-prob-space}
		\mathfrak{A}=(\Omega,\CF,\mathbb{F},\PP),
		\quad
		\mathbb{F}=(\mathcal{F}_t)_{t\in[0,T]},
	\end{equation}
	where $\mathbb{F}$ is a right-continuous filtration
	and $\mathfrak{A}$ satisfies the usual conditions.
	Let $W$ be a   $\WienH$-valued Wiener process with covariance $\mathcal{Q}$, modeled
	on $\mathfrak{A}$.
	Furthermore, let $\eta$ be an independent Poisson
	random measure with its compensator $\tilde \eta$ defined over $\mathfrak{A}$.
  
        Additionally, let $A:V\to V^\ast$, $F:V_\gamma \to X$, $\Sigma: V\to L(\mathcal{H},X)$,   and $g:V\times Z\to X$
        be two mappings satisfying Assumption \ref{standingassumption}.
%
   We fix $T>0$ and consider solutions on
	$[0,T]$. We call the progressively measurable process
	\begin{equation}\label{eq:sol-tuple}
u:\Omega\times [0,T]
		\to X 
	\end{equation}
	a {strong} solution (in the probabilistic sense) of equation \eqref{spdes}, if $u$ solves $\PP$-a.s. {for all $\phi \in V$}
\DEQSZ\label{spdes}
\lqq{ \quad\quad  \la u(t),\phi\ra  =\la u_0,\phi\ra +\int_0^ t \la \DeltaA u(s),\phi\ra \, ds + \int_0^ t {\la F(u(s)),\phi\ra} \, ds }
&&
\\ &&{}+\int_0^ t {\sum_{j=1}^{\infty}\la\Sigma(u(s))\phi_j,\phi\ra} d\beta_j(s) + \int_0^t \int_Z \la g(u(s),z),\phi \ra \,\tilde \eta(dz,ds),\quad u(0)=u_0\in \Bspace .\notag 
\EEQSZ
\end{definition}

\medskip 
{In the next step we define the integral operator for the fixed point theorem.} 
To show the existence of a solution we define a fixed point operator as follows. 
For the given filtered probability space $\Afrak$, a Wiener process $W$, and an independent  Poisson random measure 
$\eta$, and  $m\ge 2$, we define the operator $\Vcal_{\Afrak,W,\eta }:\Mcal_{\Afrak}^{m}(\X)\to\Mcal_{\Afrak}^{m}(\X)$
for $\xi\in\Mcal_{\MA}^m(\X)$ via
\DEQSZ\label{operator_def}
\Vcal_{\Afrak,W,\eta }(\xi) &:=&w,
\EEQSZ 
where $w$ is  the solution of the following It\^o stochastic partial differential equation (SPDE)
\DEQSZ\label{spdes}
\lqq{ dw(t) =\lk(\DeltaA w(t)+ \bar F(\xi(t),w(t),t)\rk)\, dt }
&&
\\ &&{}+\Sigma(w(t))\,dW(t)+ \int_Z g(w(t),z)\,\tilde \eta(dz,ds),\quad w(0)=w_0\in \Bspace .\notag 
\EEQSZ
Here, the mapping $\bar F$ is chosen in such a way that $\bar F(\xi,\xi,t)=F(\xi,t)$ in \eqref{spdesproblem} for all $t\in[0,T]$ and $\xi\in X$.

\begin{remark}
Let us assume that we have given  a $Z$--valued L\'evy process $L$ and consider in the following equation
\DEQSZ\label{spdeslevy}
\lqq{ dw(t) =\lk(\DeltaA w(t)+ \bar F(\xi(t),w(t))\rk)\, dt }
&&
\\ &&{}+\Sigma(w(t))\,dW(t)+  G(w(t))\, d L(t),\quad w(0)=w_0\in \Bspace .\notag 
\EEQSZ
Then,  for $g$  defined in \eqref{coefflevy}, the system  \eqref{spdeslevy} is equivalent to the system \eqref{spdes}.
\end{remark}

Usually, in order to apply a Schauder-Tychonoff type theorem, it is necessary to provide a bounded convex subset of $\Mcal_{\Afrak}^{m}(\X)$, such that the operator $\Vcal_{\Afrak,W,\eta }$ maps this set into itself.
Here, it is essential to characterize the set in such a way that  one can
 transfer the definition to the set of probability measures.  
{Therefore, let us fix two} measurable functions $\Phi:\DD(0,T;U)\to\RR$ and $\Psi:\DD(0,T;U)\to\RR\cup\{\infty\}$ such that for any bounded closed interval $I$ the set $\Psi^{\leftarrow}(I)$\footnote{For a measurable function $f:X\to\RR$, $f^{\leftarrow}$ denotes the preimage 
given by $\{ x\in X: f(x)\in A\}$ for all $A\in\CB(I)$.} is closed in $\mathbb{X}$.
Let us now define for any   $R>0$  
a subset 
 $\mathcal{K}_{R}(\mathfrak{A})$ of $\mathcal{M}_\MA^m(\mathbb{X})$ by
 \DEQSZ\label{characteriseK}
 \mathcal{K}_{R}(\mathfrak{A})&:=&\lk\{ \xi\in \mathcal{M}_\MA^m(\mathbb{X}):\EE\Phi(\xi)\le R^m \mbox{ and } \PP\lk(\Psi(\xi)<\infty\rk)=1\rk\}.
 \EEQSZ
The set $\mathcal{K}_R(\MA)$ has to be chosen in such a way that we can assume that \eqref{spdes} is well-posed and a unique strong solution (in the stochastic sense) $w\in\Mcal_{\MA}^m(\X)$ exists for all $\xi\in\Mcal_{\MA}^m(\X)\cap \mathcal{K}_R(\MA)$.

Since we rely on a compactness argument, the solution will be obtained in the probabilistic weak sense.
For the convenience of the reader and to keep the presentation self-contained, we recall here the notion of a probabilistic weak solution, also referred to as a martingale solution.

\begin{definition}[Weak solution in the probabilistic sense]\label{Def:mart-sol}
A {\sl  weak solution} in the probabilistic sense to the problem
\eqref{spdes} for initial data $u_0 \in X$ is a tuple 
{ \begin{equation}
\left(\Omega ,{{\mathcal{F}}},{\mathbb{F}},\mathbb{P},
W, \eta, u\right)
\label{mart-system}
\end{equation}}
such that
\begin{enumerate}
\item the quadruple $\mathfrak{A}:=(\Omega ,{{\mathcal{F}}},{\mathbb{F}},\mathbb{P})$ is a complete filtered
probability space with a filtration ${\mathbb{F}}=(\mathcal{F}_t)_{t\in [0,T]}$ satisfying the usual conditions,
\item $W$ is a  $\CH$-valued  Wiener processes over the probability space
$\mathfrak{A}$ with covariance operator $\mathcal{Q}$,
\item { $\eta$ is a  time-homogeneous Poisson random measure on $(Z,\CZ)$ with L\'evy measure $\nu$ over the probability space $\mathfrak{A}$.}
\item the process $u:[0,T]\times \Omega \to U$ is 
${\mathbb{F}}$-adapted
such that $u$ is a weak solution of the system \eqref{spdes} over the probability space $\mathfrak{A}$ in the sense of Definition \ref{def:singleSLN}.
 \end{enumerate}
\end{definition}

Now, we give our main theorem, which is a stochastic variant of the (deterministic) Schauder-Tychonoff fixed point theorem(s) from \cite[§ 6--7]{granas}.

\begin{theorem}\label{ther_main}
Let $\WienH$ be a Hilbert space, $Z$ a measurable Polish space, and $\nu$ a L\'evy measure defined on $(Z, \CZ)$.
Let 
\DEQSZ\label{probabilityspace}
\mathfrak{A}=\lk(\Omega,\PP,\CF,(\CF_t)_{t\in[0,T]}\rk)
\EEQSZ
 be a given filtered probability space carrying a \erika{$\WienH$}-valued Wiener process $W$ and a Poisson random measure $\eta\in\CM_{\bar\NN}(Z_n\times [0,T])$.
 Let us assume that $A,F,\Sigma$, and $g$ satisfy Assumption \ref{standingassumption}.

Let $m>1$ and let  two measurable functions $\Phi:\DD(0,T;U)\to\RR$ and $\Psi:\DD(0,T;U)\to\RR\cup\{\infty\}$ be given,
such that for any bounded closed interval $I$ the set $\Psi^{\leftarrow}(I)$ is closed in $\mathbb{X}$.
Let $\mathcal{K}_{R}(\mathfrak{A})$  be defined  for any $R>0$ by \eqref{characteriseK}. Let us assume that the operator $\Vcal_{\Afrak,W,\eta}$, defined by \eqref{operator_def}, restricted to $\Xcal_{R}(\Afrak)$
satisfies the following assumptions:
\begin{enumerate}
\item the operator $\Vcal_{\Afrak,W,\eta}$ is well posed on $ \Xcal_{R}(\Afrak)$ for all $R>0$;
\item there exists $R_0>0$ such that for all $R\ge R_0$, the mapping $\Vcal_{\Afrak,W,\eta}$ is onto, i.e.\ 
$$\Vcal_{\Afrak,W,\eta}(\Xcal_{R}(\Afrak))\subset\Xcal_{R}(\Afrak),
$$
\item  for all  $R>0$ the restriction $\Vcal_{\Afrak,W,\eta}\big|_{\Xcal_{R}(\Afrak)}$ is
continuous w.r.t.\ the strong topology of $\Mcal_{\Afrak}^{m}(\X)$,
\item for all $R>0$, there exist constants $K>0$, $m_0\ge 1$ such that
\[
\Eb|\Vcal_{\Afrak,W,\eta}(\xi)|_{\X^{\prime}}^{m_0}\le K\quad\text{for every}\quad \xi\in\Xcal_R(\Afrak),
\]
\item  there exist constants $K>0$ and  $m_1> m$ such that
\[
\Eb|\Vcal_{\Afrak,W,\eta}(\xi)|_{\X}^{m_1}\le K\quad\text{for every}\quad \xi\in\Xcal_R(\Afrak),
\]
\end{enumerate}

Then, there exists a filtered probability space $\hat{\Afrak}=(\hat{\Omega},\hat{\Fcal},\hat{\mathbb{F}},\hat{\P})$
(that satisfies the usual conditions) together with a  $\WienH$-cylindrical Wiener process $\hat{W}$ with covariance $\mathcal{Q}$ and Poisson random measure $\hat{\eta}\in\CM_{\bar\NN}(\{Z_n\}\times [0,T])$, both defined on  $\hat{\Afrak}$, 
and an element $w^\ast\in\Mcal_{\hat{\Afrak}}^{m}(\X)$ such that
for all $t\in[0,T]$, $\hat{\P}$-a.s.
\[
\Vcal_{\hat{\Afrak},\hat{W},\hat \eta}(w^\ast)(t)=w^\ast(t)
\]
for any initial datum $w^\ast(0):=w_0\in L^m(\Omega,\Fcal_0,\P;\Bspace)$.
Here, the operator $\Vcal_{\hat{\Afrak},\hat{W},\hat \eta}$ is defined for $\xi\in\mathcal{K}_R({\hat{\MA}})$ by 
the solution of \eqref{spdes}, where 
$$
 \mathcal{K}_{R}(\hat{\mathfrak{A}}):=\lk\{ \xi\in \mathcal{M}_{\hat \MA}^m(\mathbb{X}):\hat \EE\Phi(\xi)\le R^m \mbox{ and } \hat \PP\lk(\Psi(\xi)<\infty\rk)=1\rk\},
$$
the Wiener process $W$ and the Poisson random measure $\eta$ are replaced by $\hat W$ and $\hat \eta$, and the underlying probability space $\MA$ by $\hat{\MA}$.
\end{theorem}
\begin{remark}
Clearly,  $\Phi$ represents the order of the moment condition for the laws in $\mathcal{K}_{R}(\hat{\mathfrak{A}})$, while $\Psi$ may allow to represent an additional desired property such as non-negativity of the solutions. 
However, it is not necessary in the proof.  
\end{remark}


Before starting with the proof of Theorem~\ref{ther_main},
let us introduce the following definition.
\begin{definition} 
Let $W$ 
and $\eta\in{\CM_{\bar{\Bbb N}}(\{Z_n\times\Bbb [0, T]\})}$ be the Wiener process and the Poisson random measure introduced before.
\begin{itemize}
\item
Then we define 
\begin{equation}\label{etat1}
\mathcal{W}_t(\phi)=
\sigma\left(\left\{{W}(s):0\le s\le t\right\}\right), \,
\mathcal{W}^t=\sigma\left(\left\{ 
{W}(s)-{W}(t) 
:t< s\le T\right\}\right) ,
\end{equation}
\item
and
\begin{equation}\label{etat2}
\eta_t(B)=\eta(B\cap({Z}\times(0,t])),\,\eta^t(B)=\eta(B\cap({Z}\times(t,T] )),\,B\in{\mathcal Z\otimes\mathscr B(\Int )}.
\end{equation}
\end{itemize}
\end{definition}

\begin{proof}[Proof of Theorem~\ref{ther_main}]
Fix
$\Afrak$, the Wiener process $W$, and the Poisson random measure~$\eta$. 

{In step (I), we approximate  by the shifted Haar projection the processes, such that they are simple, predictable, and $V$-valued and compose this approximation with $\mathcal{V}_{\Afrak,W,\eta}$ and get $\mathcal{V}^\zaehler_{\Afrak,W,\eta}$. These approximations $\mathcal{V}^\zaehler_{\Afrak,W,\eta}$  induce an operator $\mathscr{V}_\zaehler$ on the set of probability measures on $\mathbb{X}$. In Step (II), we show that the discrete operators $\{\mathscr{V}_\zaehler:\zaehler\in\NN\} $ are satisfying the hypothesis of the Schauder-Tychonoff fixed point theorem. Applying the Schauder-Tychonoff fixed point theorem gives a sequence of fixed point, which are, in fact, only probability measures on $\mathbb{X}$. In Step (III) we show that the sequence of probability measures $\{\mathscr{P}_\zaehler :\zaehler \in \NN\}$ on $\mathbb{X}$ are tight and, therefore, there exists a weakly convergence subsequence and a limiting probability measure $\mathscr{P}^\ast$.
In Step (IV)  we construct a filtered probability space, such that for each probability measure $\mathscr{P}^{\ast}_\zaehler $ a process exists being a fixed point of $\mathcal{V}_{\widetilde {\Afrak},\widetilde{W},\widetilde{\eta}}$. In Step (V) we apply the usual Skorohood representation Theorem, to get a stochastic processes defined over a probability space.
In Step (VI) we reconstruct an underlying Poisson random measure for the SPDE.  Step (VII) is dedicated to the verification that, indeed, the so-constructed process solves the original martingale problem. }

\begin{step}
\item 

First, let $\zaehler \in\NN$ and let us introduce a dyadic time grid $\pi_\zaehler =\{t^\zaehler_0=0<t^\zaehler_1<t^\zaehler_2<\cdots <t^\zaehler_{2^\zaehler}=T\}$
by $t^\zaehler_k= T\frac{k}{ 2^\zaehler}$, $k=0,\ldots, 2^\zaehler-1$. A process $\xi\in \mathcal{M}^m_\MA(\mathbb{X})$
is approximated by an averaging operator over the preceding time interval.
In particular, 
let for $k\ge 1$
\begin{eqnarray}\label{hatdefined}
 (\Pro_\kl \xi )(s) &:=&
\displaystyle
 \frac {2^\zaehler} T\int_{t^\zaehler_k}^{t^\zaehler_{k+1}} \xi (r)\: dr, \mbox{ if
 } s\in [t^\zaehler_k,t^\zaehler_{k+1}), \\ &&\qquad k=1,\ldots ,2^\zaehler-1.
\notag 
 \end{eqnarray}
 For $k=0$ let $\{x_0^\zaehler:\zaehler\in\NN\}\subset V$ be a seqeunce such that $x_0^\zaehler\to w_0$ in $X$. Then put  $ (\Pro_\kl \xi )(s) :=x_0^\zaehler$ for $s\in[0,t_1^\kappa)$.
In this way, for any $\zaehler\in\NN$ and for any $\xi\in\mathcal{M}^m(\mathbb{X})$, we get a simple, 
 piecewise  constant,  $X$-valued  process ${\Pro}_\zaehler  \xi$.
 However, to define later on the It\^o integral, we need a progressive measurable process.
 To get a progressive measurable process, we shift the time intervals. First, let us define the starting value for the first interval.
 Let $\xi_0^\zaehler  :\Omega\to V^\ast$, $\CF_0$-measurable, such that  
\DEQSZ\label{initconv}
\lk( \frac {2^\zaehler }T\rk)^m \EE|\xi_0^\zaehler -w_0|^m_X &\to & 0,
\EEQSZ
 as $\zaehler \to\infty$. Next, for any $\xi\in\mathcal{M}_\MA^m(\mathbb{X})$ let
\begin{equation}\label{eq:hatProj}
	\widehat{\Pro}_\kl \xi (s):=\bcase 
	\xi_0^\zaehler \qquad \qquad \qquad \qquad\text{if } \ s\in [0,t^\zaehler_1),\vspace{0.2cm}\\
	(\Pro_\kl\xi ){(s-\frac{T}{2^{\kl}})}~
	\qquad \text{ if } s\in [t^\zaehler_1,T].
	\ecase
\end{equation} 
 In this way, for any $\zaehler\in\NN$, for any $\xi\in\mathcal{M}^m(\mathbb{X})$, we get a simple, 
 piecewise  constant,  $X$-valued, and progressive measurable   process $\widehat {\Pro}_\zaehler  \xi$.
\begin{remark}\label{projection}
Observe, the projection $\widehat{\Pro}_\kl$ satisfies the following properties (compare to Appendix \ref{haarproj}):
%
%
\begin{enumerate}[leftmargin=0cm,itemindent=.5cm,parsep=5pt,labelwidth=\itemindent,labelsep=0.0cm,align=left,topsep=0pt, label={{(\alph{*})  }}]
\item The projection $\widehat{\Pro}_\zaehler$ is a continuous operator from $\mathbb{X}$ into $\mathbb{X}$ (compare \eqref{1.2} and 
$\widehat{\Pro}_\kl$ is linear on $[T 2^{-\zaehler },T]$).
\item The projection $\widehat{\Pro}_\zaehler$ is a continuous operator from $\mathbb{W}^\alpha_m(0,T;X)$ into $\mathbb{W}^\alpha_m(0,T;X)$ (compare Lemma \eqref{boundinw}).
\item If $K$ is a bounded compact subset of $\mathbb{X}$, i.e. $K$ is bounded in $\mathbb{X}'$, 
then we know by Lemma \ref{convergenceproj}-\eqref{1.5} and \eqref{initconv}, that for all $\ep>0$ there exists a $\zaehler_0\in\NN$ such that
$$\| \widehat{\Pro}_\zaehler  \xi - \xi\|_{\mathbb{X}}\le \ep,\quad \xi\in K,\quad \forall \zaehler\ge \zaehler_0.
$$
\end{enumerate}
\end{remark}
\noindent Finally, let us define the operator
\begin{equation}\label{eq:hatV}
\hat{\mathcal{V}}^\zaehler  _{\mathfrak{A},W,\eta}(\xi):=  \lk( \widehat{\Pro}_\zaehler  \mathcal{V}_{\mathfrak{A},W,\eta}(\xi)\rk),\quad \xi\in \Kcal_R (\mathfrak{A}).
\end{equation}
\noindent Claim: For any $\ep>0$ there exists some $\zaehler  _0\in\NN$ such that we have for all $\zaehler  \ge \zaehler  _0$
\DEQSZ\label{convergence}
\lk(\EE \lk\|(\widehat \Pro_{\zaehler })\xi-\xi\rk\|_{\mathbb{X}}^m\rk)^\frac 1m<\ep ,\quad \xi\in \mathcal{V}^\zaehler_{\mathfrak{A},W,\eta }(\Kcal_R (\mathfrak{A})).
\EEQSZ
To see this, let us fix $\ep>0$. First, we know that there exists a $C>0$ such that for all $\xi\in \hat{\mathcal{V}} _{\mathfrak{A},W,\eta}(\mathcal{K}_R(\MA))$,
$$
\EE\| \xi\|_{\X'}\le C.
$$
Hence, by the Chebychev inequality, it follows that there exists a compact set $K_\ep\subset\X$ (which is bounded in $\X'$)
 such that
\DEQSZ\label{vorher}
\PP\lk( \hat{\mathcal{V}} _{\mathfrak{A},W,\eta}^\zaehler  (\xi) \not\in K_\ep\rk)  &\le& \frac C{\tilde R},
\quad
{\xi\in\mathcal{K}_R(\MA).}
\EEQSZ
Choose $\tilde R$ sufficiently large so that there exists 
$\tilde m$ with 
\[
\frac{1}{\tilde m} + \frac{1}{m_1} \leq 1,
\]
and such that 
\[
\left(\frac{C}{\tilde R}\right)^{\tilde m} \leq \frac{\vareps}{4}.
\]
%
Now, let $\zaehler _0\in\NN$ {be sufficiently large such
that}
$$
\|\widehat{\Pro}_
\kl  f- f\|_{\X}^m \le \frac \ep2,\quad f\in K_\ep ,\quad \forall \zaehler\ge\zaehler_0.
$$
Taking expectation, we get
\DEQS
\lqq{ \EE\| \hat{\mathcal{V}}^\zaehler  _{\mathfrak{A},W,\eta} (\xi) - {\mathcal{V}} _{\mathfrak{A},W,\eta}(\xi)\|^m_{\X}  }
\\
&\le&
\EE1_{\xi\in K_\ep}\| \hat{\mathcal{V}}^\kl _{\mathfrak{A},W,\eta} (\xi) - {\mathcal{V}} _{\mathfrak{A},W,\eta}(\xi)\|^m_{\X}
\\ &&{}+2^{m-1}\EE 1_{\xi\not \in K_\ep}\lk(\|
 \hat{\mathcal{V}} ^\zaehler_{\mathfrak{A},W,\eta} (\xi) \|^m_{\X}+\| \hat{\mathcal{V}} _{\mathfrak{A},W,\eta}(\xi)\|^m_{\X}\rk).
\EEQS
To estimate the first term we know that
\DEQS
\EE 1_{\xi\in K_\ep}\| \hat{\mathcal{V}}^\kl_{\mathfrak{A},W,\eta} (\xi) - {\mathcal{V}} _{\mathfrak{A},W,\eta}(\xi)\|^m_{\X}
\le 
\lk(\frac \ep 2\rk)\quad \forall \zaehler \ge \zaehler _0.
\EEQS
To estimate the second term we apply the H\"older inequality for $\tilde m>1$ and
{$\frac 1{\tilde m}+\frac 1{m_1} =  1$}
\DEQS
\lqq{
\EE 1_{\xi\not \in K_\ep}\lk(\|
 \hat{\mathcal{V}} ^\zaehler_{\mathfrak{A},W,\eta} (\xi) \|^m_{\X}+\|{\mathcal{V}} _{\mathfrak{A},W,\eta}(\xi)\|^m_{\X}\rk)
 }
&&
\\ &\le& \lk(\EE 1_{\xi\not \in K_\ep}^{\tilde m}\rk)^\frac 1{\tilde m}\,
\lk(\EE \lk(\|{\mathcal{V}} _{\mathfrak{A},W,\eta}(\xi) \|^{m_1}_{\X}+\| \hat{\mathcal{V}} ^\zaehler_{\mathfrak{A},W,\eta} (\xi)\|_{\X}^{m_1}\rk)\rk)^\frac 1{m_1}.
\EEQS
In addition, we know that  $\EE 1_{\xi\not \in K_\ep}^m= \EE 1_{\xi\not \in K_\ep}=\PP\lk( \|\xi\|_{\X'}\ge \tilde R\rk) $. By the Chebycheff inequality, taking into account \eqref{vorher},
we have verified \eqref{convergence} and shown the claim.

\begin{remark}\label{eulerscheme}
This remark illustrates the functioning of the shifted Haar projection. 
Let us assume that  $\xi\in \widehat\Pro_\zaehler  \,\mathcal{K}_R(\MA)$. Then, $\xi$  is a piecewise constant function on $\pi_k$, i.e. 
$
\xi(t)=\xi(s)$ for all $t,s\in[t^\zaehler_k,t^\zaehler_{k+1})$, $k=0,\ldots , 2^\zaehler -1$.
In addition, for all $t\in[0,T]$,  $\xi(t)$ with $t^\zaehler_k\le t$ is a $\CF_{t^\zaehler_k}$-measurable random variable. Let us fix $t^\zaehler_k$, where $0\le k\le 2^\zaehler -1$. Then, we have for $t\in [t^\zaehler_k,t^\zaehler_{k+1})$ 
\begin{align*}
\lqq{ {\mathcal{V}}_{\mathfrak{A},W,\eta }(\xi)(t)- {\mathcal{V}}_{\mathfrak{A},W,\eta }(\xi)(t^\zaehler_k)=
\int_{t^\zaehler_k}^ {t} A w(s)\, ds }
&
\\
& {}+ 
\int_{t^\zaehler_k}^ {t}\bar  F(w(s),\bar\xi_k )\, ds
 + \int_{t^\zaehler_k}^ {t} \Sigma(w(s) )\, dW(s)
 +  {\int_{t^\zaehler_k}^ {t^\zaehler_{k+1}}\int_Z g(w(s),z )\,\tilde{\eta}(dz,ds)},
\end{align*}
 where $w(s)={\mathcal{V}}_{\mathfrak{A},W,\eta }(\xi)(s)$ and
 $$\bar \xi_k=
 \frac 1{T2^\zaehler} \int_{t^\zaehler_{k-1}}^{t^\zaehler_{k}} \xi (r)\: dr.
 $$
In the next step, the average over $[t^\zaehler_k,t^\zaehler_{k+1})$ is calculated. Here, we get
\begin{align*}
\lqq{ \frac 1\tau \int_{t^\zaehler_k}^{t^\zaehler_{k+1}}{\mathcal{V}}_{\mathfrak{A},W,\eta }(\xi)(t)\, dt = {\mathcal{V}}_{\mathfrak{A},W,\eta }(\xi)(t^\zaehler_k)+
\frac 1\tau  \int_{t^\zaehler_k}^{t^\zaehler_{k+1}} \int_{t^\zaehler_k}^t Aw(s)\, ds\, dt
}&
\\
&{}
+
\frac 1\tau \int_{t^\zaehler_k}^{t^\zaehler_{k+1}}\int_{t^\zaehler_k}^ t \bar  F(w(s)
,\bar\xi_k)\, ds\, dt
+\frac 1 \tau  \int_{t^\zaehler_k}^ {t^\zaehler_{k+1}}\int_{t^\zaehler_k}^t  \Sigma(w(s) )\, dW(s)\, dt
 \\
 &{}+ \frac 1 \tau  \int_{t^\zaehler_k}^ {t^\zaehler_{k+1}}\int_0^ t \int_Z g(w(s),z )\,\tilde{\eta}(dz,ds) \, dt.
\end{align*}
Now, we apply the shift operator. To be more precise, we can write for $t\in (t^\zaehler_{k+1},t^\zaehler_{k+2}]$, 
 \begin{equation}
 \label{recursion}
\begin{aligned}
\lqq{ \hat {\mathcal{V}}^\zaehler_{\mathfrak{A},W,\eta }(\xi)(t) =  {\mathcal{V}}_{\mathfrak{A},W,\eta }(\xi)(t^\zaehler_k)+
\frac 1\tau  \int_{t^\zaehler_k}^{t^\zaehler_{k+1}} \int_{t^\zaehler_k}^t Aw(s)\, ds\, dt
}&
\\
&{}
+
\frac 1\tau \int_{t^\zaehler_k}^{t^\zaehler_{k+1}}\int_{t^\zaehler_k}^ t \bar  F(w(s)
,\bar\xi_k)\, ds\, dt
\\
&{}
 +\frac 1 \tau  \int_{t^\zaehler_k}^ {t^\zaehler_{k+1}}\int_{t^\zaehler_k}^t  \Sigma(w(s) )\, dW(s)\, dt+ \frac 1 \tau  \int_{t^\zaehler_k}^ {t^\zaehler_{k+1}}\int_0^ t \int_Z g(w(s),z )\,\tilde{\eta}(dz,ds) \, dt.
\end{aligned}
 \end{equation}
 In particular, the random variable  $\hat {\mathcal{V}}^\zaehler_{\mathfrak{A},W,\eta }(\xi)(t)$ 
 for $t\in (t^\zaehler_{k+1},t^\zaehler_{k+2}]$ is $\CF_{t^\zaehler_{k+1}}$--measurable and depends on ${\mathcal{V}}_{\mathfrak{A},W,\eta }(\xi)(t^\zaehler_k)$.
 \medskip
 If we suppose that $A$ is linear,  $\bar F$ is  linear in the first variable and independent of the time,
 we get
 $$
 \int_{t^\zaehler_k}^ {t^\zaehler_{k+1}}\bar  F(w(s),\bar\xi_k)\, ds
= \underbrace{ \int_{t^\zaehler_k}^ {t^\zaehler_{k+1}}w(s) \, ds}_{\tau \bar w_{k+1}}\, \bar F(1,\bar \xi_k)
=\tau \bar w_{k+1}\bar F(1,\bar \xi_k)=\tau  \bar F(\bar w_{k+1},\bar\xi_k).
$$
where $\bar w_{k+1}:=\frac 1 \tau \int_{t^\zaehler_k}^ {t^\zaehler_{k+1}}  w(s)\, ds$.
If, in addition,  $\Sigma$ and  $g$ are constant and independent of the time 
we get
\DEQS
\hat {\mathcal{V}}^\zaehler_{\mathfrak{A},W,\eta }(\xi)(t)
=  {\mathcal{V}}_{\mathfrak{A},W,\eta }(\xi)(t^\zaehler_k)
+{\tau}  \lk( A\bar\xi_k + \bar F(\bar w_{k+1},\bar\xi_k)\rk)
+  \Sigma\, (W(t^\zaehler_{k+1}-W(t^\zaehler_k))+
\zeta_{k+1},
\EEQS
where 
$$
\zeta_{k+1}=\int_{t^\zaehler_k}^{t^\zaehler_{k+1}}  \int_Z g(z)\,\tilde \eta(dz,ds)  .
$$
$\square$
\end{remark}

\item
\newcommand{\Pros}{\mathscr{P}\mbox{\sl roj}}
Given the probability space $\mathfrak{A}=\lk(\Omega,\PP,\CF,(\CF_t)_{t\in[0,T]}\rk)$ (which is the probability space given in \eqref{probabilityspace} of 
Theorem~\ref{ther_main}),
for any $\zaehler  \in\NN$ the operator $\mathcal{V}^\zaehler  _{\mathfrak{A},W,\eta}$ is defined on $\mathcal{M}_\MA^m(\X)$. Now, this operator $\mathcal{V}^\zaehler  _{\mathfrak{A},W,\eta}$
induces  an operator $\mathscr{V}_\zaehler $ on the set of Borel probability measures on $\mathbb{X}$, denoted by  $\mathscr{M}_1(\mathbb{X})$. The construction of the operator $\mathscr{V}_\zaehler $ is done in this step. 

Since it is important that we have simple processes,
we define firstly an operator $\Pros_\zaehler $ acting on $\mathscr{M}_1(\X)$. To do this,  
let $\mathscr{Q}\in \mathscr{M}_1(\X)$. By the Skorokhod lemma, it follows that there exists a probability  space $\hat {\MA}=(\hat\Omega,\hat {\CF},\hat \PP)$ and a $\X$-valued random variable $\xi$ over $\hat{\MA}$ such that $\Law(\xi)=\mathscr{Q}$. Let $\Pros_\zaehler  \mathscr{Q}$ be given by
$$
\Pros_\zaehler  \mathscr{Q}:\mathscr{B}(\mathbb{X})\ni {B}\mapsto \hat \PP\lk( \lk\{ \omega\in\hat\Omega :\Pro_\zaehler  \xi(\omega)\in {B}\rk\}\rk)\in[0,1],
$$
and the compact set 
\DEQS
\lqq{ \mathscr{K}_{\zaehler ,R}:=\lk\{ P\in \Pros_\zaehler  \mathscr{M}_1(\mathbb{X}):\int_{\mathbb{X}}\Phi(\xi)dP(\xi)\le R\rk.
}
\\
&& \lk. \phantom{\int_{\mathbb{X}}}
\qquad
\mbox{ and } P\lk(\{\xi\in\mathbb{X}:\Psi(\xi)<\infty\}\rk)=1\rk\}.
\EEQS
Since $X\subset U$, $\Pro_\zaehler \mathbb{X}\subset \mathbb{U}$, and the measurable functions $\Phi$ and $\Psi$ is well-defined.

Now we define the operator $\mathscr{V}_\zaehler $ acting on $ \mathscr{K}_{\zaehler ,R}$. 
 Let $\mathscr{Q}$ be a  probability measure belonging to $ \mathscr{K}_{\zaehler ,R}$. Then, by the Skorokhod theorem, we know that there exists a probability space $\mathfrak{A}_0=(\Omega_0,\CF^0,\PP_0)$ and a random variable $\xi:\Omega_0\to \Pro_\zaehler \mathbb{X}$ such that the law of $\xi$ coincides with the law $\mathscr{Q}$. In particular, the probability measure $\mathscr{P}_\xi:\mathscr{B}(\mathbb{X})\to[0,1]$ induced by $\xi$ and given by
 $$
\mathscr{P}_\xi:\mathscr{B}(\mathbb{X})\ni B\mapsto \PP_0\lk( \lk\{ \omega{\in \Omega_0}:\xi(\omega)\in {B}\rk\}\rk)
$$
coincides with the probability measure $\mathscr{Q}$.

Due to the construction of the set $ \mathscr{K}_{\zaehler ,R}$,  we know that $\xi$ is a simple stochastic process,  i.e.\ for any $s\in[0,T]$, $\xi(s)$ is a $X$--valued random variable, and we can define a filtration.
Let
				$$
			 \CG^0_t:=\sigma \lk( \lk\{\,{\xi (s)}\;\colon\; 0\le s\le t \, \rk\}\cup \CN_0\rk),\quad t\in [0,T],
				$$
				where $\CN_0$ denotes the zero sets of ${\mathfrak{A}}_0$.
				Let us put $\mathfrak{A}_0:=(\Omega_0,\CF^0,(\CG_t^0)_{t\in[0,T]},\PP_0)$.

Next, we construct an extension of $\MA_0$ on which the Wiener process and the Poisson random measure are defined.
				To do so, let $\mathfrak{A}_1=(\Omega_1,\CF^1,(\CF_t)_{t\in[0,T]},\PP_1)$ be a probability space
				on which  a cylindrical Wiener process  $W:\Omega_1\to C([0,T];\WienH)$   
				and  a  time-homogeneous Poisson random measure  $\eta:\Omega_1\to{M_{\bar{\Bbb N}}(\{Z_n\times\Int \})}$  with intensity measure  $\nu$ are  defined.
Let $\mathfrak{A}_{\mathscr{Q}}$ be the product probability space of $\mathfrak{A}_0$ and $\mathfrak{A}_1$.
				In particular, we set
\DEQS
\Omega_\mathscr{Q} &:= &  \Omega_0\times \Omega_1,
\\
\CF_\mathscr{Q} &:=& \CF^0\otimes \CF^1,
\\
\CG^\mathscr{Q}_t &:=& \CG^0_t\otimes \CG^1_t,\, t\in[0,T],
\\
\mbox{and}\quad \PP_\mathscr{Q} &:=& \PP_0\otimes \PP_1.
\EEQS
Here, we know that $\xi$ at time $t$ is independent of $\mathcal{W}^t$  and independent of $\eta^t$. In particular, for a process $\zeta\in\Pro_\zaehler  \CM_\MA^m(\mathbb{X})$,  the integral
 $$
 [0,T]\ni t\mapsto \int_0^ t \Sigma(\zeta (s), {s})\, dW(s)
 \quad \mbox{and}\quad 
  [0,T]\ni t\mapsto \int_0^ t \int_Z g(\zeta (s),z)\, \tilde \eta(dz,ds)
$$
can be defined in the It\^o sense.

Next, we verify that the family of operators
$$
\lk\{ \Vcal^\zaehler _{\mathfrak{A}_\mathscr{Q},{W},\eta }: \zaehler \in\NN\rk\}
$$
is well-posed. This is given by item (i) in Theorem~\ref{ther_main}, 
and since $\mathscr{Q}\in\mathscr{K}_{\zaehler,R}$, that is, $\xi\in\Kcal_R (\mathfrak{A}_0)$, where 
\DEQS
\lqq{
\Kcal_R (\mathfrak{A}_0) :=\{ \xi:\Omega_0\to\mathbb{X}\mid  }
&&
\\
&&\quad 
\xi \in \CM^m(\X) \mid 
\EE_1\Psi(\xi)\le R,\mbox{ and }\PP_1\lk(\Phi(\xi)<\infty\rk)=1\}.
\EEQS
Now, let us assume that  $A\in\mathscr{B}( \mathbb{X})$. Then, the mapping $\mathscr{V}_\zaehler $ maps
the probability measure $\mathscr{Q}:\mathscr{B}(\mathbb{X})\to[0,1]$ via  $\xi$ onto the probability measure $\mathscr{P}_{ {\mathscr{V}_\zaehler (\mathscr{Q})}}:\mathscr{B}(\mathbb{X})\to[0,1]$ given by
$$
\mathscr{P}_{ {\mathscr{V}_\zaehler (\mathscr{Q})}}(B):=\PP_\mathscr{Q}\lk( \lk\{ \omega\in\Omega_\mathscr{Q}: \hat{\Vcal}^\zaehler _{\mathfrak{A}_\mathscr{Q},{W},\eta }(\xi(\omega))\in B\rk\}\rk).
$$
Note, since $\mathbb{X}$ is a complete metric space,
the space of probability measures over $\mathbb{X}$ equipped with the
Prokhorov metric\footnote{Let $\mathscr{M}_1(X)$ be the set of Borel probability measures on the metric space $(X,d)$ equipped with the weak topology. Let $\nu,\mu\in\mathscr{M}_1(X)$. Then
$$d_\alpha(\mu,\nu):=\inf\{\alpha>0: \mu(A)\le \nu(A_\alpha)+\alpha \mbox{ and }
\nu(A)\le \mu(A_\alpha)+\alpha \mbox{ for all } A\in\mathscr{B}(X)\}.
$$
Here, $A_\alpha:=\{ x\in X: d(x,A)<\alpha)$.}   is also complete.

The following points are valid: \label{seetightnss}
\begin{enumerate}
\item[a)]  $\mathscr{K}_{\zaehler ,R}$  is invariant under $\mathscr{V}_\zaehler $.
This follows directly from Assumption (ii) of the statement  
and the properties of the projection $\widehat \Pro_\zaehler $ (see Remark \ref{projection}).
\item[b)] Due to the fact that $\hat{\Vcal}^\zaehler _{\Afrak_{\mathscr{Q}},W,\eta}$ restricted to $\Kcal_R(\MA_\mathscr{Q}) $ is continuous, $\mathscr{V}_\zaehler $ restricted to $\mathscr{K}_{\zaehler ,R}$ is continuous on $\mathscr{M}_1(\mathbb{X}) $ in the Prokhorov metric.
This point follows from Theorem~11.7.1 \cite{dudley2002}.
\item[c)]  \label{seetightnss1}
The operator $\mathscr{V}_\zaehler $ restricted to $\mathscr{K}_{\zaehler ,R}$ is compact on $\mathscr{M}_1(\mathbb{X}) $, i.e.\
$\mathscr{V}_\zaehler $ maps bounded sets into pre-compact sets.
In fact, {it is sufficient} to show that for all $\zaehler \in\NN$ and $\ep>0$ there exists a compact subset $K_\ep\in \mathscr{B}(\mathbb{X})$ such that
$$
\forall \mathscr{Q}\in \mathscr{K}_{\zaehler ,R} \mbox{ we have }\mathscr{P}_ {\mathscr{V}_\zaehler (\mathscr{Q})}\lk (\mathbb{X}\setminus K_\ep\rk)\le \ep \quad \mbox{and}\quad \mathscr{P}_ {{ \mathscr{V}}_\zaehler (\mathscr{Q})}:=\mathscr{V}_\zaehler (\mathscr{Q}).
$$
However, due to (iv) in the statement there exists a $C>0$ such that
$$
\EE\|\mathcal{V}^\zaehler _{\mathfrak{A},W,\eta}(\xi)\|_{\mathbb{X}'}^m\le C,\quad \xi\in\Kcal_R (\mathfrak{A}). 
$$
Let $\tilde R>0$ so large that
$$
\frac {C}{\tilde R^m}\le \ep
$$
and let $ {K_{\ep}}:=\{x\in\mathbb{X}:\|x\|_{\mathbb{X}'}\le \tilde R\}$.
Due to the construction of the operator $\mathscr{V}_\zaehler $ we have
equality in law. In particular,
we know by the Chebyscheff inequality that
{\begin{eqnarray*}
&&\mathscr{V}_\zaehler (\mathscr{Q})\lk (\mathbb{X}\setminus K_\ep\rk)\\
&=&
\mathscr{P}_{ \mathscr{V}_\zaehler (\mathscr{Q})}\lk (\mathbb{X}\setminus K_\ep\rk)
=
\mathscr{V}_\zaehler (\mathscr{Q})\lk ({ \lk\{
x\in\mathbb{X}\mid \|x\|_{\mathbb{X}'} \ge R\rk\}}\rk)\\
&=& \PP_\mathscr{Q}\lk( \lk\{ \omega\in\Omega_\mathscr{Q}: \hat{\Vcal}^\zaehler _{\mathfrak{A}_\mathscr{Q},{W},\eta }(\xi(\omega))\in \mathbb{X}\text{ and }\left\|\hat{\Vcal}^\zaehler _{\mathfrak{A}_\mathscr{Q},{W},\eta }(\xi(\omega))\right\|_{\mathbb{X}'} \ge R\rk\}\rk)\\
&\leq& \PP_\mathscr{Q}\lk( \lk\{ \omega\in\Omega_\mathscr{Q}: \left\|\hat{\Vcal}^\zaehler _{\mathfrak{A}_\mathscr{Q},{W},\eta }(\xi(\omega))\right\|_{\mathbb{X}'} \ge R\rk\}\rk)\\
&\leq&\frac{\EE_\mathscr{Q}\left\|\hat{\Vcal}^\zaehler _{\mathfrak{A}_\mathscr{Q},{W},\eta }(\xi(\omega))\right\|_{\mathbb{X}'}^m}{\tilde R^m}\leq \frac {C}{\tilde R^m}\le \ep.
\end{eqnarray*}}
$$
\mathscr{V}_\zaehler (\mathscr{Q})\lk (\mathbb{X}\setminus K_\ep\rk)=
\mathscr{P}_{ \mathscr{V}_\zaehler (\mathscr{Q})}\lk (\mathbb{X}\setminus K_\ep\rk)
=
\mathscr{V}_\zaehler (\mathscr{Q})\lk ({ \lk\{
x\in\mathbb{X}\mid \|x\|_{\mathbb{X}'} \ge R\rk\}}\rk)
\le \epsilon$$
for all $\mathscr{Q}\in \mathscr{K}_{\zaehler,R}$.
Since $\mathbb{X}'\hookrightarrow \mathbb{X}$ compactly embedded,
we have proven the tightness.
\item[d)] We show, that $\mathscr{K}_{\zaehler,R}$ is a convex subset of $\mathscr{M}_1(\mathbb{X})$.
Let $\mathscr{P},\mathscr{Q}\in\mathscr{K}_{\zaehler,R}$, it is sufficient to show that for any $\alpha\in(0,1)$ we have $\alpha \mathscr{P}+(1-\alpha)\mathscr{Q}\in\mathscr{K}_{\zaehler,R}$.
First, we show that the expectation of $\Psi$ with respect to $\alpha \mathscr{P}+(1-\alpha)\mathscr{Q}$ is smaller than $R$. However, this follows by the linearity of the expectation. 
Secondly, we show that
$$
\lk(\alpha \mathscr{P}+(1-\alpha)\mathscr{Q}\rk)\lk(\{x\in\mathbb{X}:\Psi(x)<\infty\}\rk)=1.
$$
This can be shown by direct calculations. In fact, since $\mathscr{P},\mathscr{Q}\in\mathscr{K}_{\zaehler,R}$
we know that $\mathscr{P}\lk(\{x\in\mathbb{X}:\Psi(x)<\infty\}\rk)=1$ and $\mathscr{Q}\lk(\{x\in\mathbb{X}\mid \Psi(x)<\infty\}\rk)=1$. Let $\alpha\in(0,1)$. Then
\begin{align*}
&\lk(\alpha \mathscr{P}+(1-\alpha)\mathscr{Q}\rk)\lk(\{x\in\mathbb{X}:\Psi (x)<\infty\}\rk)
\\
&=\alpha \underbrace{\mathscr{P}\lk(\{x\in\mathbb{X}:\Psi(x)<\infty\}\rk)}_{=1}
+(1-\alpha)\underbrace{\mathscr{Q}\lk(\{x\in\mathbb{X}:\Psi(x)<\infty\}\rk)}_{=1}=1.
\end{align*}
Finally, we show that $\mathscr{P}\lk( \lk\{ x\in \Pro _\zaehler\mathbb{X}\rk\}\rk)=1$ and $\mathscr{Q}\lk( \lk\{ x\in \Pro_\zaehler \mathbb{X}\rk\}\rk)=1$ imply
\begin{align*}
&\lk(\alpha \mathscr{P}+(1-\alpha)\mathscr{Q}\rk)\lk( \lk\{ x\in \Pro_\zaehler\, \mathbb{X}\rk\}\rk)=1.
\end{align*}
But this follows by by straightforward calculations.

\end{enumerate}
Summarising, we know that the mapping $\mathscr{V}_\zaehler $ restricted to $\mathscr{K}_{\zaehler,R}$
satisfies the assumption of the Schauder-Tychonoff theorem.
Hence, for any $\zaehler \in\NN$ there exists a probability measure $\mathscr{P}^\ast_\zaehler $  such that $\mathscr{V_\zaehler }(\mathscr{P}^\ast_\zaehler )=\mathscr{P}^\ast_\zaehler $.

				\item
Note, since the estimate on $\mathbb{X}'$ is uniform for all $\zaehler\in\NN$,
the set
$$
\lk\{ \mathscr{P}^\ast_\zaehler \mid\zaehler \in\NN\rk\}
$$				
is tight. Therefore there exists a subsequence $\{\zaehler _j\mid j\in\NN\}$ and a Borel probability measure $\mathscr{P}^\ast$ such that
$ \mathscr{P}_{\zaehler _j}^\ast\to 	 \mathscr{P}^\ast$, as $j\to\infty$.	
In this step, we construct from the family of probability measures $\{\mathscr{P}^\ast_{\zaehler _j}\mid j\in\NN\}$ and $\mathscr{P}^\ast$, a filtered probability space $\mathfrak{A}^\ast$, a Wiener process
${W}^\ast$, a Poisson random measure $
\eta^\ast$, a
progressively measurable process $w^\ast$, and
a family of progressively measurable processes  $\{ w^\ast_{\zaehler _j}\mid j\in\NN\}$
				belonging a.s. to $\mathbb{X}$
	  over $\mathfrak{A}^\ast$ such that these objects have probability measures $\{\mathscr{P}^\ast_{\zaehler _j}\mid j\in\NN\}$ and $\mathscr{P}^\ast_{\zaehler _j}\in\mathscr{K}_{\zaehler_j,R}$.
	
Let us start. By the Skorokhod lemma \cite[Theorem~4.30]{kallenberg},
there exists a probability space $\mathfrak{A}_0^{\ast}=(\Omega^\ast_0,\CF^{\ast}_0,\PP^\ast_0)$ and a  sequence of $\mathbb{X}$-valued random variables $\{ {w}^\ast_{\zaehler _j}:j\in\NN\}$
and ${w}^\ast_{\zaehler _j}$ where
the random variable $w^\ast_{\zaehler _j}: \Omega^\ast_0\to \X$ such that
%
\DEQSZ\label{equallaw1}
\text{Law}({w}_{\zaehler _j}^\ast)=
\mathscr{P}_{\zaehler _j}^\ast,\quad j\in\NN.
\EEQSZ
In addition, we have 
\[
{w}^\ast_{\zaehler _j}\to{w}^\ast\quad \mbox{ as $j\to \infty$ }\quad  {\P}^\ast_0\text{-a.s.}
\]
on $\X$. 
Let us introduce the filtration $\mathbb{G}^\ast_0=(\mathcal{F}_t^{\ast,0})_{t\in[0,T]}$  given by
				$$
			 \CF^{\ast,0}_t:=\sigma \lk( \lk\{\,w _{\zaehler _j}^\ast(s),w^\ast(s),\;\colon\; 0\le s\le t, \,j\in\NN \rk\}\cup \CN_0^{\ast}\rk),\quad t\in [0,T],
				$$
				where $\CN_0^{\ast}$ denotes the zero sets of ${\mathfrak{A}}_0^\ast$. 
				
Next, we have  to construct  the Wiener process and the time-homogeneous Poisson random measure.
 Since we have only the processes $w^\ast$ and $w_{\zaehler _j}$ defined on our probability space $\mathfrak{A}^\ast_0$.
Let 
$$
{\mathfrak{A}_1=\lk(\Omega_1,\PP_1,\CF_1,(\CG^1_t)_{t\in[0,T]}\rk).}
$$
 be a filtered probability space
on which a  Wiener process   $W^\ast:\Omega_1\to C\lk([0,T];\WienH\rk)$ (with covariance operator $\mathcal{Q}$) and  a time-homogeneous Poisson random measure $\eta^\ast:\Omega_1\to{M_{\bar{\Bbb N}} \lk(\{Z_n\times[0,T] \}\rk)}$ with intensity measure $\nu$ are defined.
Let $\mathfrak{A}^\ast:=\mathfrak{A}_0^\ast\times \mathfrak{A}_1$. In particular, we put
\DEQS
\Omega ^\ast&= &  \Omega_0^\ast\times \Omega^\ast_1,
\\
\CF^\ast &=& \CF_0^\ast\otimes \CF_1^\ast,
\\
\CG^\ast_t &=& \CG^{0,\ast}_t\otimes \CG^{1,\ast}_t,\,t\in[0,T],
\\
\mbox{and}\quad \PP^\ast &=& \PP^\ast_0\otimes \PP_1.
\EEQS
In addition, let
$$\Kcal_R (\mathfrak{A}^\ast):=\lk\{ \xi\in \mathcal{M}_{\MA^\ast}^m(\mathbb{X}):\EE^\ast \Phi(\xi)\le R^m \mbox{ and } \PP^\ast \lk(\Psi(\xi)<\infty\rk)=1\rk\}.
$$
Let us remind that the  operators 
 ${\mathcal{V}}_{\mathfrak{A}^\ast,W^\ast,\eta^\ast}$ and $\hat{\mathcal{V}}^\zaehler  _{\mathfrak{A}^\ast,W^\ast,\eta^\ast}$ for $\xi\in \Kcal_R (\mathfrak{A}^\ast)$ are defined  by 
\DEQS
{\mathcal{V}}_{\mathfrak{A}^\ast,W^\ast,\eta^\ast}(\xi):=w,\, \mbox{ where $\xi$ solves \eqref{spdes}},
\EEQS
and
\DEQS
\widehat{\mathcal{V}}^\zaehler_{\mathfrak{A}^\ast,W^\ast,\eta^\ast} (\xi):=\widehat{\Pro}_\zaehler \,
{\mathcal{V}} _{\mathfrak{A}^\ast,W^\ast,\eta^\ast}(\xi),\quad \zaehler\in\NN.
\EEQS
Observe, the equation 	\eqref{spdes} is given, where the Wiener process and Poisson random measure $W$ and $\eta$ are replaced by $W^\ast$ and $\eta^\ast$, the underlying probability space is also replaced by $\MA^\ast$.			
				\item
			
Since $\mathscr{P}^\ast_{\zaehler _j}\in  {\mathscr{K}_{\zaehler ,R}}$, the process $w^\ast_{\zaehler _j}$ belongs to $ {\Kcal_R} (\mathfrak{A}^\ast)$ and, hence,
 $\hat{ \mathcal{V}}_{\mathfrak{A}^\ast,W^\ast,\eta^\ast}w^\ast_{\zaehler _j}$ is well defined.
Observe, that  $\hat{\mathscr{V}}_{\zaehler _j}(\mathscr{P}^\ast_{\zaehler _j})=\mathscr{P}^\ast_{\zaehler _j}$, 
does not imply {that} 
the process $w^\ast_{\zaehler _j}$ satisfies 
\DEQS
\PP\lk(
	\hat{\mathcal{V}}^{\zaehler _j}_{\MA^\ast,W^\ast,\eta^\ast}\lk(w^\ast_{\zaehler _j}\rk) (s) =w^\ast_{\zaehler _j}(s)\rk)&=&1 \quad \mbox{for} \quad 0\le s\le T
.
\EEQS

In this step, we construct here a fixed point, denoted by  $ w ^\ast_{\zaehler_j,\infty}$, to the operator $\hat{\mathcal{V}}^{\zaehler _j}_{\MA^\ast,W^\ast,\eta^\ast}$ and that the probability measure will not change.
Let us define a new process by induction. To start with let (here, $t^\zaehler_k=\frac {Tk}{2^\zaehler}$)
\DEQSZ\label{nummer11}
\tilde w^\ast_{\zaehler _j,1}(s) &:=&  \begin{cases}
 {\xi_0^{\zaehler _j}}& \mbox{ if } 0\le s<t^{\zaehler _j}_1,
\\
\lk(\hat{\mathcal{V}}^{\zaehler _j}_{\MA^\ast,W^\ast,\eta^\ast}(w^\ast_{\zaehler _j})\rk)(s)& \mbox{ if } t^\zaehler _1\le s\le T,
\end{cases}
\EEQSZ
 {where $\xi_0^{\zaehler _j} :\Omega\to V^\ast$, $\CF_0$-measurable, such that  
$$
\lk( \frac {2^{\zaehler _j}}T\rk)^m \EE||\xi_0^{\zaehler _j}-w_0||^m_X \to  0,
$$
as ${\zaehler _j}\to\infty$, 
see \eqref{initconv}.}
Clearly, in the time interval $[0, t^{\zaehler _j}_1)$ the law is the same.
In Remark~\ref{eulerscheme}, we have seen that 
the operator $\hat{\mathcal{V}}^{\zaehler _j}_{\MA^\ast,W^\ast,\eta^\ast}$ at time $t_{k}$ is $\CF_{t_k}$-measurable. 
Next, since the operator $\hat{\mathcal{V}}^{\zaehler _j}_{\MA^\ast,W^\ast,\eta^\ast}$ is invariant with respect to the measure $\PP^\ast$, we not change the law on the time interval $[t^\zaehler _1, T]$. 
Next, let us put
\DEQSZ\label{nummer21}
\tilde w^\ast_{\zaehler _j,2}(s) &:=&\begin{cases}
\tilde w^\ast_{\zaehler _j,1}(s)& \mbox{ if } 0\le s<t^{\zaehler_j}_2,
\\
\lk(\hat{\mathcal{V}}^{\zaehler _j}_{\MA^\ast,W^\ast,\eta^\ast}(\tilde w^\ast_{\zaehler _j,1})\rk)(s)& \mbox{ if } t^{\zaehler_j}_2\le s\le T.
\end{cases}
\EEQSZ
Again, since in the time interval $[0, t^{\zaehler _j}_2)$ the law is the same and, since the operator $\hat{\mathcal{V}}^{\zaehler _j}_{\MA^\ast,W^\ast,\eta^\ast}$ is invariant with respect to the measure $\PP^\ast$, we not change the law on the time interval $[t^\zaehler _2, T]$. 
Let us define the remaining part by induction.
Now, having defined $\tilde w^\ast_{\zaehler _j,k}$, let
\DEQSZ\label{nummerk1}
\tilde w^\ast_{\zaehler _j,k+1}(s) &:= & \begin{cases}
\tilde w^\ast_{\zaehler _j,k}(s)& \mbox{ if } 0\le s<t^{\zaehler_j}_{k+1},
\\
\lk(\hat{\mathcal{V}}^{\zaehler _j}_{\MA^\ast,W^\ast,\eta^\ast}(\tilde w^\ast_{\zaehler _j,k})\rk)(s)& \mbox{ if } t^{\zaehler_j}_{k+1}\le s\le T.
\end{cases}
\EEQSZ
Let us put $  w^\ast_{\zaehler_j,\infty}(s)= {\xi_0^{\zaehler _j}}$ for $t_0^{\zaehler _j}=0\le s< t_1^{ {\zaehler _j}}$, and
\DEQSZ\label{definfty1}
w_{\zaehler _j,\infty}^\ast(s) &:=&  \tilde  w^\ast_{\zaehler _j,k}(s),\quad \mbox{if}\quad t^{\zaehler _j}_{k}\le s< t^{\zaehler _j}_{k+1},\,\,k=1,\ldots, 2^{\zaehler_j}.
\EEQSZ
Our claim is now, that the process $\tilde  w^\ast_{\zaehler_j  ,\infty}$ satisfies 
\DEQSZ\label{isasolution1}
\PP^\ast \lk(\lk\{\omega\in\Omega^\ast\mid 
{\hat {\mathcal{V}}}^{\zaehler _j }_{\MA^\ast,W^\ast,\eta^\ast}\lk( w^\ast_{\zaehler_j ,\infty}\rk) (s) = w^\ast_{\zaehler _j,\infty}(s)\rk\}\rk)&=&1 \,\mbox{ for } \, 0\le s\le T
.
\EEQSZ
{In fact, on the one hand we have $\PP^\ast$-a.s. by definition \eqref{eq:hatV} 
\begin{eqnarray}
    { \hat{ \mathcal{V}}}^{\zaehler _j }_{\MA^\ast,W^\ast,\eta^\ast}\lk(  w^\ast_{{\zaehler _j } ,\infty}\rk)(s)= \lk( \widehat{\Pro}_{\zaehler_j}  \mathcal{V}_{\MA^\ast,W^\ast,\eta^\ast}(w^\ast_{{\zaehler _j } ,\infty})\rk)(s),\quad \xi\in \Kcal_R (\mathfrak{A}).
\end{eqnarray}}

By the definition of $\widehat{\Pro_{\zaehler _j}}$ on $[0,t_1^{\zaehler _j})$
for any $\xi\in\CM_{\MA^\ast}^m(\X)$ the process $\widehat{\Pro_{\zaehler _j}}\xi$ on $[0,t_{1}^{\zaehler _j})$ is defined by  {$\xi_0^{\zaehler _j}$ }.
In particular, we have  $\PP^\ast$-a.s.
$$
{ \hat{\mathcal{V}}}^{{\zaehler_j} }_{\MA^\ast,W^\ast,\eta^\ast}\lk(  w^\ast_{{\zaehler _j },\infty}\rk)(s)
=
 {\xi_0^{\zaehler _j}}, \quad \mbox{for} \quad 0\le s<t^{\zaehler _j}_1.
$$
On the other side, we have $ w^\ast_{\zaehler _j ,\infty}(s)= w^\ast_{\zaehler _j }(s)= {\xi_0^{\zaehler _j}}$ for $ 0\le s<t^{\zaehler _j}_1$.
Hence, \eqref{isasolution1} holds for $0\le s<t^{\zaehler _j}_1$,
$\PP^\ast$-a.s.
Let us consider the next time interval. 
Remember, the value at time $t_1^{\zaehler _j}$ depends on the value on the time interval $[0,t_1^{\zaehler _j})$. 
 {At time $s\in[0,t^{\zaehler _j}_1]$, we have by \eqref{definfty1}, $w^\ast_{{\zaehler _j } ,\infty}(s) =\tilde{ w}^\ast_{{\zaehler _j ,1 } }(s)$ and, therefore,}
$$
\lk({ \hat{\mathcal{V}}}^{{\zaehler_j} }_{\MA^\ast,W^\ast,\eta^\ast}\lk( w^\ast_{{\zaehler _j } ,\infty}\rk)\rk)(t_1^{\zaehler _j})
%
=\lk({\hat{\mathcal{V}}}^{{\zaehler_j} }_{\MA^\ast,W^\ast,\eta^\ast}  \lk(\tilde w^\ast_{{\zaehler_j},1 }\rk)\rk)(t^{\zaehler _j}_1)
. 
$$
Due to before, the process $ {w^\ast_{{\zaehler _j } ,\infty}}$ and $\tilde w^\ast_{\zaehler_j,1}$ are indistinguishable, therefore 
$\hat{\mathcal{V}}^{{\zaehler_j} }_{\MA^\ast,W^\ast,\eta^\ast}( {w^\ast_{{\zaehler _j } ,\infty}})$ and 
$\hat{\mathcal{V}}^{{\zaehler_j} }_{\MA^\ast,W^\ast,\eta^\ast} (\tilde w^\ast_{\zaehler_j,1})$ 
 are indistinguishable on $[0,t_1^{\zaehler _j})$, and, by \eqref{nummer11}
 we have $\PP^\ast$ a.s.
 $$\hat{\mathcal{V}}^{{\zaehler_j} }_{\MA^\ast,W^\ast,\eta^\ast} (\tilde w^\ast_{\zaehler_j,1})(t_1^{\zaehler_j})=\tilde w _{\zaehler_j,1}^\ast
 (t_1^{\zaehler_j}).
 $$
Due to the fact that the processes are simple, we know the equality holds on $[t_1^{\zaehler_j},t_2^{\zaehler_j})$.
However, by definition \eqref{definfty1}, we have $\tilde w^\ast_{{\zaehler_j} ,1}(s)= w^\ast_{{\zaehler _j } ,\infty}(s)$
for $s\in [t_1^{\zaehler_j},t_2^{\zaehler_j})$. It follows that   $\PP^\ast $-a.s. we have $ {\tilde w^\ast_{{\zaehler_j} ,1}(s)}=w^\ast_{{\zaehler_j} ,\infty}(s)$ for $ t^{\zaehler _j}_1\le s<t^{\zaehler _j}_2$, and hence
$$
\PP\lk( w^\ast_{{\zaehler_j} ,\infty }(s)=  {\hat{\mathcal{V}}^{{\zaehler_j} }_{\MA^\ast,W^\ast,\eta^\ast}} \lk(w^\ast_{{\zaehler_j} ,\infty }\rk)(s)\rk) =1,  \quad \mbox{for} \quad t^{\zaehler _j}_1\le s<t^{\zaehler _j}_2.
$$
Let us analyse what happens in $t^{\zaehler _j}_2$. By the definition \eqref{definfty1} we have $w^\ast_{{\zaehler_j} ,\infty}(t^{\zaehler _j}_2)= {\tilde w^\ast_{{\zaehler_j} ,2}(t^{\zaehler _j}_2)}$. Next, for $s\in[t_1^{\zaehler_j},t_2^{\zaehler_j})$ we know from before that 
$\PP^\ast$-a.s.
$$ 
\hat{\mathcal{V}} ^{{\zaehler_j} }_{\MA^\ast,W^\ast,\eta^\ast}
\lk( w^\ast_{{\zaehler _j } ,\infty}\rk)(s)=w^\ast_{{\zaehler _j } ,\infty}(s).
$$
Due to definition \eqref{nummer21}
we have 
$$
 {\hat{\mathcal{V}}^{{\zaehler_j} }_{\MA^\ast,W^\ast,\eta^\ast}}\lk(  w^\ast_{{\zaehler _j } ,\infty}\rk)(t_2^{\zaehler _j})
= {\hat{\mathcal{V}}^{{\zaehler_j} }_{\MA^\ast,W^\ast,\eta^\ast}}\lk( {\tilde w^\ast_{{\zaehler_j} ,2}}\rk)(t^{\zaehler _j}_2)=  {\tilde w^\ast_{{\zaehler_j} ,2}}(t^{\zaehler _j}_2).
$$
Again, by the definition \eqref{definfty1} we have  {$\tilde w^\ast_{{\zaehler_j} ,2}(t^{\zaehler _j}_2)= w^\ast_{{\zaehler _j } ,\infty}(t^{\zaehler _j}_2)$.}

Now, we can proceed by induction. Let us assume that in $[0,t_k)$ we have shown that
\DEQSZ\label{inductionstart}
\PP\lk(
 {\hat{\mathcal{V}}^{{\zaehler_j} }_{\MA^\ast,W^\ast,\eta^\ast}}\lk( w^\ast_{{\zaehler _j } ,\infty}\rk)
 (s) =w^\ast_{{\zaehler_j} ,\infty}(s)\rk)&=&1 \quad \mbox{for} \quad 0\le s\le t^{\zaehler _j}_k
.
\EEQSZ
Then,  we have by definition \eqref{nummerk1} on $t_k^{\zaehler _j}\le s<t_{k+1}^{\zaehler _j}$
$$
  {\hat{\mathcal{V}}^{{\zaehler_j} }_{\MA^\ast,W^\ast,\eta^\ast}}\lk(w^\ast_{{\zaehler_j} ,\infty }\rk) (s) = {\hat{\mathcal{V}}^{{\zaehler_j} }_{\MA^\ast,W^\ast,\eta^\ast}}\lk( {\tilde w^\ast_{{\zaehler_j} ,k}}\rk) (s)=
\tilde w^\ast_{{\zaehler_j} ,k}( {s})
$$
Since we have by definition \eqref{definfty1} it follows 
 $w^\ast_{{\zaehler_j} ,\infty}=\tilde w^\ast_{{\zaehler_j} ,k}$.

\item
\newcommand{\weg}[1]{{\color{yellow} #1}}
\newcommand{\Xw}{\tilde w}
\newcommand{\Xwt}{\tilde w}
\newcommand{\wii}{\widetilde{aa}}
Next, we verify the following statements with the goal to pass on to the limit. 
Let us recall that due to the construction of the operator, the laws of the set 	$$
	\left\{ \mathcal{V}_{\mathfrak{A}^{\ast},W^{\ast},\eta^\ast }^{\zaehler _j}(\xi) : \xi \in \mathcal{X}_R(\mathfrak{A}^{\ast}) , j\in\NN\right\}   
	$$
	are  tight in $\mathbb{X}$, and hence the set $\Law\{ w^\ast_{{\zaehler_j},\infty}: j\in\NN\}$  
    is also tight in $\mathbb{X}$. That means, there exists a subsequence 
    $\{ {\zaehler_{j_m}}:m\in\NN\}$ and a measure $\rho$ on $\mathbb{X}$ such that  $\Law( w^\ast_{\zaehler_{j_m}})\to \rho$ weakly as $m\to\infty$.
    
    {For the convenience of the reader, we now relabel the sequence $\{ w^\ast_{{\zaehler_{j_m}},\infty} : j \in \mathbb{N} \}$ by introducing a new sequence $\{ w_{m} : m \in \mathbb{N} \}$ defined via $w_m := w^\ast_{{\zaehler_{j_m}},\infty}$.} 
    In the first step, we apply the Skorokhod embedding theorem to get a sequence 
$\{ {\Xw} _m\}_{m\in\mathbb{N}}$  of $\mathbb{X}\cap \mathbb{D}(0,T;U)$-valued random variables 
and a $\mathbb{X}\cap \mathbb{D}(0,T;U)$-valued random variable $ \Xwt $ over a filtered probability space 
$\widetilde{\mathfrak{A}}=(\widetilde{\Omega}, \widetilde{\mathcal{F}}, \widetilde{\mathbb{P}})$ such that 
$\{ \Xw_m\}_{m\in\mathbb{N}}$ converges
$\widetilde{\PP}$-a.s.\ to the process $\Xwt^\infty$. 
In addition, we have 
$$
\Law(w_m)=\Law(\tilde w_m)\quad \forall m\in\NN, 
$$
and  $\Law(\tilde w^\infty)=\rho$.  
\item 
In this next step, we reconstruct the martingale parts of the processes 
$\{ \Xw _m\}_{m\in\mathbb{N}}$ and $\Xwt$ and the corresponding filtrations, such that the stochastic integral is represented. 
Then, we construct 
\begin{itemize}
    \item 
an extension\footnote{For the definition of extension, we refer to Appendix \ref{extension}.}
${\widetilde{\widetilde{\mathfrak{A}}}=(\widetilde{\widetilde{\Omega}}, \widetilde{\widetilde{\mathcal{F}}}, \widetilde{\widetilde{P}})}$ of the probability space $\widetilde{\mathfrak{A}}=(\widetilde{\Omega}, \widetilde{\mathcal{F}}, \widetilde{\mathbb{P}})$, 
\item a filtration {$(\widetilde{\widetilde{\mathcal{F}}}_ t)_{t\in[0,T]}$}, 
\item a Wiener process
{$\widetilde{\widetilde{W}}(\cdot)$}
 \item and a Poisson random measure {$\wi{\wi{\eta}}$} on $(Z_n,\CZ)$ over {$\widetilde{\widetilde{\mathfrak{A}}}$} ,
 \end{itemize}
 such that  $\{\Xw_m\}_{m\in\mathbb{N}}$ and  
 $\Xw$ solve the equations \eqref{spdesproblem} 
 on the extended probability space 
 $\widetilde{\widetilde{\mathfrak{A}}}$.
 \begin{remark}
     Let us note, that by the defintion of an extension, all random variable existing on $\widetilde{\mathfrak{A}}$
     will be carried over to the probability space $\widetilde{\widetilde{\mathfrak{A}}}$. To avoid confusion, we will denote the elements of $\widetilde{\mathfrak{A}}$ by $\wi{\wi{{\cdot}}}$.
     
 \end{remark}
 To start with, let us define 
 for each $m\in\N$, the following process $\widetilde{M}_m(\cdot)$ 
	\begin{align}\label{MT1}
		\wi{M}_m(t):= {\wi{w}_m}(t)-x -\int_0^t A(\Pro_m  {\wi{w}_m}(s))d s-\int_0^tF_m(\Pro_m  {\wi{w}_m}(s))d s, \quad t\in[0,T],
	\end{align}
Observe, {$w^\ast$ is defined on $\mathfrak{A}^\ast$ and $\wi{w}_{\zaehler_m,\infty}$ is defind on $\tilde{\mathfrak{A}}$.
On  $\tilde{\mathfrak{A}}$, we can only see the process and we have to find an extension such that we can define a Poisson random measure.}

By the L\'evy-It\^o decomposition 
we can identify the continuous part $\widetilde M^{c}_n$, the finite variational part $\widetilde M^{fv}_m$, 
and the discontinuous part $\widetilde M^{d}_m$ i.e.
$$ 
\widetilde M_m=\widetilde M^{c}_m+\widetilde M^{d}_m+\widetilde M^{fv}_m.
$$
\renewcommand{\Xw}{w}
\renewcommand{\Xwt}{ w}
%
Note, the continuous martingale part
has a quadratic variation given by 
\begin{align}\label{MT3}	
	\left\langle\left\langle 
    \widetilde M^{c}_m(t)\right\rangle\right\rangle
    =\int_{0}^{t}\left(\Sigma\left(\Pro_m \widetilde\Xwt _{m}(s)\right)\, {\mathcal{Q}^{1 / 2}}\right)\left(\Sigma\left(\Pro_m  \widetilde \Xwt _{m}(s)\right)\, {\mathcal{Q}^{1 / 2}}\right)^{*} d s, \quad t \in[0, T] . 
\end{align} 
 \erika{It should be in the formulation of the theorem, it is the covariance matrix}\question{If you want to keep it, please insert $\mathcal{Q}$ on the definition of wiener process $W$ in \eqref{eq:wiener}.}
\newcommand{\up}{{\bf p}}
\newcommand{\uq}{{\bf q}}
For the continuous part we can reconstruct the Wiener process along the lines of \cite{ankit}. 

Let us focus on the discontinuous part. 
Here, we want to introduce the following notation. 
\begin{remark}\label{pandq}
    {If we have in the index $\up$, we mean we investigate the random measure related to the point process in $U$. If we write $\uq$, we investigate that random measure, corresponding  to the point process in $Z$.}
        {If we are one the probability space $\mathfrak{A}^\ast$, we add some $\ast$ and use  $\up^\ast$, respective  $\uq^\ast$.}

    \end{remark}
\noindent For a parameter $\epsilon$ we define $Z_\epsilon:=\{z\in Z: |z|\le \epsilon\}$ and $Z_\epsilon^c:=Z\setminus Z_\epsilon$.
Fix a sequence \(\vareps>0\) such that $Z^c_{\ep}\subset Z^c_{\epsilon_{\zaehlerz-1}}$ and $\nu(Z^c_{\ep}\setminus Z^c_{\epsilon_{\zaehlerz-1}}) =1$, split \(Z\) into \(Z_\vareps=\{z\in Z~|~|z|_Z\le\vareps\}\) and \(Z^c_\ep=Z\setminus Z_\vareps\). Observe, due to the construction of the operator $\Pro_m$, we know that 
$\xi_{k,m}:=\Pro_m\tilde w_m(t^m_k)$ is $V$-valued for all $k=0,1,2,\cdots,2^m$ 
Due to the assumption on $g$, we know that  
$${\wi{\wi{\mathbb{E}}}}\int_Z  \Big(1\wedge |g(\xi_{k,m},z)|_X^2\Big) \nu(dz)<\infty, $$
and
$$\sup_{m\in\NN} {\wi{\wi{\mathbb{E}}}} \int_0^ T \int_Z  \Big(1\wedge |g(\Pro_m\tilde w_m(s),z)|_X^2\Big)\nu(dz)\, ds<\infty.
$$
 {Consider the random set of admissible jump increments of $\xi_{k,m}$ defined by } 
 $$V_\ep^{k,m}:=\{ x\in V:\exists z\in Z_\ep^c \mbox{ such that } g(\xi_{k,m} ,z)=x\}$$ and let 
 \DEQSZ\label{def_ell}
 {\ell}^{\vareps}_{k,m}:=\bcase  \widetilde {M}^d_m(t_{k+1}^m)- \widetilde {M}^d_m(t_{k}^m), & \mbox{if}\, \widetilde {M}^d_m(t_{k+1}^m)- \widetilde {M}^d_m(t_{k}^m)\in V_\ep^{k,m},
 \\0 & \mbox{else},
 \ecase
 \EEQSZ 
 and $\bar{\ell}^{\vareps}_{k,m}:=2^{-m}\!\int_{Z_\vareps^c} g(\xi_{k,m},z)\,\nu(dz)$.
 \del{ 
\[
{\ell}^{\vareps}_{k,m}:=\int_{t_k^m}^{t_{k+1}^m}
\int_{Z_\vareps^c} g(\xi_{k,m}^\ep,z)\,\eta(dz,ds) \quad \mbox{and}\quad 
\bar{\ell}^{\vareps}_{k,m}:=2^{-m}\!\int_{Z_\vareps^c} g(\xi_{k,m}^\ep,z)\,\nu(dz).
\]}
Finally, for each dyadic interval define the increment 
	\[
	\widetilde{\ell}^{\vareps}_{k,m}:= {\ell}^{\vareps}_{k,m}-
	\bar{\ell}^{\vareps}_{k,m}.
	\]
%
Let us recall that, since $u \in \mathbb{D}(0,T;U)$
we can observe {and extract} the jump times within each interval $I_k^m$
that belong to the set $V_{\vareps}^{k,m}$. 
This means that for each time interval $I_{k}^m$
there exists a random number of jumps associated to the set $V_{\vareps}^{k,m}$, 
denoted by $n^{\vareps}_{k,m}$, 
and, if $n^{\vareps}_{k,m} > 0$, the corresponding jump times are given by
\[
\{ \sigma^{j}_{k,m,\vareps} : j = 1, \dots, n^{\vareps}_{k,m} \}.
\]
 Observe that
$ \{ \sigma_{k,m,\ep_1}^j: j=1,\cdots, n^{\ep_1} _{k,m}\}\subset \{ \sigma_{k,m,{\ep_2}}^j: j=1,\cdots, n^{\ep_2} _{k,m}\}$, if $\ep_1>\ep_2$.
%

\noindent Next, let us define the random measure $\hat{N}_\up^{m,k,\vareps}$ 
by setting for $B\in\CB(V_\ep^{k,m})$ and $I\in \CB(I_{k,m})$ 
\DEQS
\hat{N}_\up^{m,k,\vareps}(B\times I)
	:= \begin{cases}
    \sum_{j=1}^
  {n^\ep_{k,m}} \mathds{1}_{I}(\sigma_{k,m,\ep}^j) 
	,& \text{if } n_{k,m}^{\ep}\ge 1\quad    \text{ and }  \frac {l_{k,m}^\ep}{{n_{k,m}^{\ep}}}\in B, 
 \\
		0,& \text{otherwise}.
	\end{cases}
\EEQS
The random measure on $[0,T]$ is given by
\DEQSZ\label{defrandom}
\\
\notag 
\hat{N}_\up^{m,\vareps}(B\times I)
	:= \sum_{k=0}^{2^m-1} \hat{N}_\up^{m,k,\vareps}(B\times (I\cap I_{k,m})), \quad \forall I\in\CB([0,T]) \, \mbox{and}\, 
    B\in\CB(V_\ep^{k,m}).
\EEQSZ
	Define for each \(m\in\mathbb{N}$ and $\vareps>0\) the predictable map
	\[
	\Theta^{m}_{\vareps}(s,z):=g(\Pro_m \wi{w}_m(s),z), \quad z\in Z. 
	\]
Note that, by the definition of $\Pro_m$, we have $(\Pro_m \wi{w}_m)(s) \in V$ for all $s\in[0,T]$ and is constant on the intervals $I_k^m$. 
Since $g : V \to L(Z, X)$, it follows that $\Theta^{m}_{\vareps}$ 
is an $X$-valued predictable process.
In the next step, we define for any $m\in\NN$ a random kernel 
$$
Q^{m,\ep} :[0,T]\times X\times \CZ
\ni  ( s,x, C)\mapsto  {Q^{m,\ep}(s,x,C)}\in {[0,1]}
$$
such that we have 
for any function $\phi:[0,T]\times X\times Z\to\RR$ measurable and $I\in \CB([t_k^m,t_{k+1}^m))$  $\PP$-a.s.
\begin{align}\notag
\int_I\int_{Z} & \left\{\mathds{1}_{\Theta_\ep ^m(s, z)\neq \{0^\ast\} } \phi(s,\Theta_\ep ^m(s, z), z)\right\} \nu(d z) \, ds \\
	& = \EE\lk[\int_I \int_{X}\left\{\int_{Z} \phi(s,x, z) \, Q^{m,\ep}(s, x, d z)\right\}  \hat{N}_\up^{m,\vareps} ( d x,ds)\mid \CF_{t_{k}^m}\rk]   \label{7.28}.
\end{align}
Here, it is essential  that for $s\in [t_k^m,t_{k+1}^m)$, $\Theta_\ep ^m(s, z) $ is $\CF_{t_{k}^m}$ measurable. 
We note that $\Theta_\ep ^m(s, z)\neq 0^\ast \in X$ implies that we reconstruct the Poisson random measure on $Z$ which 
is somehow minimal in the sense that it represents only non-zero jumps in $X$. 

	By Kallenberg’s disintegration theorem \cite[Corollary~3.6]{kallenberg2}
    there is a regular kernel \(Q^{m,\vareps}\) and it can be chosen to be predictable in \(s\). 
    Intuitively, \(Q^{m,\vareps}_s(x,\cdot)\) plays the role of a {conditional inverse} of 
    \(\Theta^{m}_{\vareps}(s,\cdot)\).

In the next step, we construct an approximation of a Poisson random measure on $(Z,\CZ)$ having L\'evy measure $\nu$.
First, note that the stochastic process 
$[0,T] \ni t \mapsto Q^{m,\vareps}_t$, 
where $Q^{m,\vareps}_t : X \times \mathcal{Z} \to [0,\nu_{\vareps}(Z)]$ 
is a  kernel of a finite measure, satisfies (after normalisation) the assumptions of 
Lemma~3.22 in~\cite[p.~56]{kallenberg}. 
Hence, there exists a random process, 
interpreted as the density process 
 $$ f^{m,\ep}:[0, T] \times X
 \times[0,1] \times \Omega 
 \ni (t, x, \alpha, \omega)\mapsto
 f^{m,\ep}(t, x, \alpha, \omega)\in Z
 $$
   such that  $\mbox{Leb}( \{\alpha : f^{m,\ep}(t, x, \alpha, \omega) \in C\})=Q^{m,\ep}_t( x, C)$ for every $C \in \CZ$.


Let $\mathfrak{A}^\prime := \left(\Omega^{\prime}, \mathscr{F}^{\prime}, \mathbb{P}^{\prime}\right)$ 
be a probability space on which a sequence of mutually independent and identically distributed 
random variables $\xi^\zaehlerz_{m,k}$, $m \in \mathbb{N}$, $k \in\NN$, $\zaehlerz\in\NN$,
is defined, each uniformly distributed on $[0,1]$.
Set 
\DEQSZ\label{sigmatilde}
\wi{\wi{{\Omega}}}= \wi{\Omega} \times \Omega^{\prime}, 
\wi{\wi{\mathscr{F}}}=\wi{\mathscr{F} }\otimes \mathscr{F}^{\prime}, \wi{\wi{\mathbb{P}}}=\wi{\mathbb{P}} \otimes \mathbb{P}^{\prime}
.
\EEQSZ
Finally, let $\wi{\wi{\mathfrak{A}}}=(\wi{\wi{\Omega}}, \wi{\wi{\CF}},\wi{\wi{\PP}})$ be the corresponding probability space.
Let us define the following $\NN_0$-valued random measure\footnote{Let us remind, $\hat  N_{ \up}^{m,\ep}$ is defined in \eqref{defrandom}.} 
\DEQS 
&&\hat N_{\bar \up}^{m,\ep}:~\mathscr{B}([0,T])\times \mathscr{B}(V_\ep^{k,m})\times \mathscr{B}([0,\nu_{\vareps}(Z)]) \to \NN_0\\
&& (I_1, B, I_2) \mapsto \hat N_{\bar \up}^{m,\ep}(I_1 \times B\times I_2)  := \hat N_\up^{m,\ep}(B\times I_1) \, \#\{ k: \xi^\ell_{m,k}\in I_2,\ell\le \zaehlerz\}
\EEQS 
and define $\hat N^{m,\ep}_\uq:\CZ\times \mathscr{B}([0,T])\to \NN_0$ 
 by 
\DEQSZ\label{zuvor}
\\
\notag 
\hat N^{m,\ep}_\uq( C\times I )&:=&\int_I \int_{X\times[0,1]}  \mathds{1}_{C}(f^{m,\ep}(s, x, \alpha))\hat  N^{m,\ep}_{\bar\up}(d s, d x, d\alpha) .
\EEQSZ 
The aim is to investigate the properties of  $\hat N^{m,\vareps}_{\uq}$  and to show that the {limit in law} of $\hat N^{m,\vareps}_{\uq}$ for $m\to \infty$ is a time-homogeneous Poisson random measure with L\'evy measure $\nu_\ep$.  
By direct calculation, using the fact that the processes have the same law, we can show that the random measure is indeed a Poisson random measure on
 $(Z,\CZ)$ with L\'evy measure $\nu_\ep$. 
 Here it is important that the entities like $\wi{\PP}(\{\hat N^{m,\vareps}_{\uq}=k\})$ or properties such as  being independently scattered and predictability can be characterized by identities at the level of the probability distributions.
Since these properties are distributional in nature, equality in law suffices to transfer them to the constructed random measure, which justifies its identification as a Poisson random measure.

Recall that the family of random measures 
$$\{ \hat {N}^{m,\vareps}_{\uq} : m \in \mathbb{N}\}$$
is constructed on the probability space 
{$\widetilde{\widetilde{\mathfrak{A}}}$}. 
For the construction of the random measure,  we have used the processes 
$$\{{\widetilde{\widetilde{ w}}_m} : m \in \mathbb{N}\},$$ 
which also live on {$\widetilde{\widetilde{\mathfrak{A}}}$} and are equal in distribution 
to the relabeled processes 
$\{{ w^\ast_{\zaehler _m,\infty}} : m \in \mathbb{N}\}$ 
defined on {$\mathfrak{A}^\ast$}. 
From these processes we know that each ${w^\ast_{\zaehler _m,\infty}}$ is a fixed point of the operator 
$\hat{\mathcal{V}}^{\zaehler_m}_{\mathfrak{A}^\ast, W^\ast, \eta^\ast}$ 
(see Step (IV)). 
Hence, these processes are solutions to the SPDEs given by~
\begin{align*}
\lqq{w^\ast_m(t)= x_0+ 
\int_{0}^ {t} A \Pro_m w^\ast_m(s)\, ds+ 
\int_{0}^ {t}  F(\Pro_mw^\ast_m(s) )\, ds }
&
\\
& {}
 +\int_{0}^ {t} \Sigma(\Pro_m w^\ast_m (s) )\, dW(s)
 +  {\int_{0}^ {t} \int_Z g(\Pro _mw^\ast_m(s),z)\,\tilde{\eta}(dz,ds)}.
\end{align*}
\del{Let us start with noting that $ {w}^\ast_{\zaehler _m} $ solves over $\mathfrak{A}^\ast$ the following SPDE
\begin{align*}
\lqq{w^\ast_m(t)= x_0+ 
\int_{0}^ {t} A \Pro_m w^\ast_m(s)\, ds+ 
\int_{0}^ {t}  F(w(s),\Pro_mw^\ast_m(s) )\, ds }
&
\\
& {}
 +\int_{0}^ {t} \Sigma(\Pro_m w^\ast_m (s) )\, dW(s)
 +  {\int_{0}^ {t} \int_Z g(\Pro _mw^\ast_m(s),z )\,\tilde{\hat N}^{\infty}(dz,ds)}.
\end{align*}}

\del{Also, note that, due to the construction, for all $\ep_1>\ep_2>0$, there exists some $m_0\in\NN$ such that for all $C\in\CZ$ and $I\in\mathscr{B}([0,T])$, $N^{m,\vareps_2}_{\uq}(B\times I)\ge N^{m,\vareps_2}_{\uq}(B\times I)$ for all $m\ge m_0$.
}
We mimick the construction on $\mathfrak{A}^\ast$, 
starting with the relabeled sequence $\{  {w^\ast_{\zaehler _m,\infty}} : m \in \mathbb{N} \}$, 
defining 	\[
	\Theta^{\ast,\zaehler _m}_{\vareps}(s,z):=g(\Pro_{\zaehler _m}( {w^\ast_{\zaehler _m,\infty}})(s),z), \quad z\in Z, 
	\]
 defined over $\mathfrak{A}^\ast$,
results in a family of random measures 
$\{ {N}^{\ast,\zaehler _m,\vareps}_{\up} : m \in \mathbb{N} \}$,
defined by
%
\DEQSZ\label{zuvorast}\,\,
\\
\nonumber
N^{\ast,\zaehler _m,\ep}_\up( B\times I )&:=&\int_I \int_{ {Z}}  \mathds{1}_{B}( \Theta^{\ast,\zaehler _m}_{\vareps}(s,z)) \mu(d s,  {d z}), \,B\in \CB(V_{k,m,\ep}^c) , \, I\in \CB([0,T]) .
\EEQSZ 
On the other side, we can define a random measure similarly to \eqref{defrandom}, only living on $\mathfrak{A}^\ast$.
Note, that we add a $\ast$ at all random variables and entities defined on the probability space $\mathfrak{A}^\ast$.

Let us remind, that we know that the laws of  {$w_m$} and $ {w^\ast_{\zaehler _m}}$ are identical 
Moreover, the number of jumps of $\up$ denoted by $n^{\vareps}_{k,m}$ and the corresponding set of jump times 
$\{ \sigma^{j}_{k,m,\vareps} : j = 1, \dots, n^{\vareps}_{k,m} \}$ 
have the same distribution as the number of jumps of  {$p^{\ast}$}\footnote{For the definition of $\up^\ast$ and $\uq^\ast$ see Remark \ref{pandq}.}
denoted 
 {$n^{\ast,\vareps}_{k,m}$}, and its jump times 
$\{  {\sigma^{\ast,j}_{k,\zaehler _m,\vareps}} : j = 1, \dots, n^{\ast,\vareps}_{k,m} \}$.
In addition, the distribution of ${\ell}^{\vareps}_{k,m}$, defined in~\eqref{def_ell}, is known. 
To show this, let us denote by $\nu_{\up}^{\ast,k,\vareps,m}$ the random  measure on $X$ induced by the mapping $g$, i.e.,
\[
 {\nu_{\up}^{\ast,k,\vareps,m}}(B)
:= \int_{Z_\vareps^c} 
\mathds{1}_B\big( g( {\xi_m(t_k^m)}, z) \big) \, \nu(dz),
\qquad B \in \mathscr(V_{\vareps,m,k}).
\]
Let us note that for fixed $m$ and $k=0,\ldots, 2^m-1$, this is a L\'evy measure. 
We know that the distribution of  {$\ell_{k}^{m,\ast,\vareps}$} 
is {given by \eqref{def_ell}, but all underlying objects are defined on the probability space $\mathfrak{A}^\ast$}
is given by the exponential of measures $e(\nu_{\up}^{\ast,k,\vareps,m})$, 
as defined in~\cite[Chapter~5.3, p.~63]{linde}. 
Therefore, the law of all components from which 
$\hat {N}^{m,\vareps}_{\up}$ over $\wi{\wi{\mathfrak{A}}}$ is constructed is fully determined. 
This, in turn, allows us to analyze the limit $m \to \infty$ of 
$\hat{N}^{m,\vareps}_{\up}$.

This again allows us to analyze the limit $m \to \infty$ of 
$\hat {N}^{m,\vareps}_{\up}$.
In particular,  since  {$w^\ast_{\zaehler _m}$} and  {$\tilde{\tilde{w}}_m$} has the same distribution, we know for  $B\in\mathscr{B}(V_\ep^{k,m})$, we have 
\DEQS
\lqq{ \Law\lk(\int_{V^\ep_{m,k}}z ({N}_{\up,\tiny \mbox{comp}}^{m,\vareps}-\gamma_ {\up,\tiny \mbox{comp}}^{m,\vareps})(dz \times I_k^m) {\mid \tilde{\CF}_{t_k^m}}\rk)}
\\
&=&
\Law\lk( \int_{V^\ep_{m,k}}z ( -\gamma_{\up,\tiny \mbox{comp}}^{\ast,\zaehler _m,\vareps})(dz\times I_k^m) {\mid {\CF}^\ast _{t_k^m}}\rk).
\EEQS 
Here, $\gamma_ {\up,\tiny \mbox{comp}}^{m,\vareps}$ is {\sl compensator} of $N_ {\up,\tiny \mbox{comp}}^{m,\vareps}$, defined by 
$$
\gamma_ {\up,\tiny \mbox{comp}}^{m,\vareps}(B\times I):=\EE \lk[N_ {\up,\tiny \mbox{comp}}^{m,\vareps}(A\times I)\mid \CF_{t_k^m}\rk],\quad B\in\mathscr{B}(V_\ep^{k,m}),\, I\in\mathscr{B}([0,T]).
$$
Observe,  ${N}_\up^{\ast,\zaehler _m,\vareps}$ is define in the same way as ${N}_\up^{\zaehler _m,\vareps}$ , but all objects defined over $\mathfrak{A}^\ast$.
In addition, we have 
$$
\Law\lk( \int_{V^\ep_{m,k}}z ({{N}_{\up,\tiny \mbox{comp}}^{\ast,\zaehler _m,\vareps}}-{{\gamma}_{\up,\tiny \mbox{comp}}^{\ast,\zaehler _m,\vareps}})(dz\times I_k^m) {\mid {\CF}^\ast _{t_k^m}}\rk)= \EE \lk[ {
e\lk(\nu_\up^{\ast,k,\ep,\zaehler _m}\rk)\mid {\CF}^\ast _{t_k^m}}\rk].
$$
%
 Hence, we know  
{$
\nu_\up^{\ast,\ep,\zaehler _m,k}\stackrel{d}{=}\wi{\nu}_\up^{\ep,m,k}$.}
However, we are interested in the limit for $m\to\infty$. Therefore, we have to have a closer look.
%
 In particular,  let us define for any $m\in\NN$, the discrete filtration 
    $(\CA^m_k)_{{k\in\NN}}$ with $\CA_k^m:= \CF_{t_k^m}$ for $k=0,\cdots, 2^m$.
Let $Z\in\CB(Z_\ep^c)$ and $I\in  \CB(I_{k,m})$. 
\del{Then, we have due to the construction 
\DEQSZ\label{mmmm}\notag 
\lqq{ 
\widetilde{\widetilde{\EE}} \lk[ N^{m,\ep}_\uq( C\times I )\mid \CF_{t_k}\rk] \stackrel{d}{=}\EE^\ast \lk[  N^{\ast,m,\ep}_\uq( C\times I )\mid \CF_{t_k}\rk]}
\\
&=&\EE^\ast\lk[  \int_I \int_{X\times[0,1]}  \mathds{1}_{C}(f^{\ast,m,\ep}(s, x, \alpha)) N^{\ast,\zaehler _m,\ep} _{\up}(d s, d x)\, d\alpha \mid \CF_{t_k}\rk], \quad C\in \mathscr{B}(Z) .
\EEQSZ}
We now that $\hat{N}_\up^{m_1,\ep,\ast} (B\times [0, t) ){=}\hat{N}_\up^{m_2,\ep,\ast} (B\times [0, t) ] )$ for all $B\in \mathscr{B}(V_\ep^{m,k})$, if $n_{k,m_1}^\ep, n_{k,m_2}^\ep \le 1$ for all $k=0,\dots, 2^m$.
Now we have to calculate the probability of 
$$\PP^\ast \lk( \{ n_{k,m}^{\ep,\ast}\le 1: k=0,\cdots, 2^m\}\rk)
=\lk( \PP^\ast \lk( n_{k,m}^{\ep,\ast}= 1\rk)+\PP^\ast \lk(n_{k,m}^{\ep,\ast}= 0\rk)\rk)^{2^m}
.
$$
Let 
$$
\Omega^\ast _m:=\{ \omega\in \Omega^\ast :  n_{k,m}^{\ep,\ast}\le 1: k=0,\cdots, 2^m\},\quad \Omega^\ast_\infty:=\lim_{m\to\infty} \Omega_m^\ast .
$$
Then $\Omega^\ast _{m+1}\supset \Omega^\ast _m$, hence, $$
{\PP}^\ast (\Omega^\ast _\infty)=\lim_{m\to\infty}\PP^\ast (\Omega^\ast _m). $$
Also note that we have 
\DEQS
 \PP^\ast  \lk(n_{k,m}^{\ep}= 1\rk)
&=&e^{-\nu(Z_\ep^c )\Leb( I_k^m)}\cdot \nu(Z_\ep^c)\, \Leb( I_k^m)
\EEQS 
and
\DEQS 
%
\PP\lk(n_{k,m}^{\ep,\ast}= 0\rk)
&= &e^{-\nu(Z_\ep^c )\Leb( I_k^m)} .
\EEQS 
Taking the limit for $m\to\infty$ we obtain 
\DEQS
 \lk( \PP^\ast \lk(n_{k,m}^{\ep,\ast }= 1\rk)+\PP^\ast \lk(n_{k,m}^{\ep,\ast }= 0\rk)\rk)^{2^m}
&=&\lk[ e^{-\nu_\up^\ep/{2^m}}+ e^{-\nu_\up^\ep/{2^m}} \frac {\nu_\up^\ep}{{2^m}}\rk]^{2^m}
\\
\lqq{\hspace{-3cm}=\lk( e^{-\nu_\up^\ep/{2^m}}\rk)^{2^m}\lk( 1+\frac {\nu_\up^\ep}{{2^m}}\rk)^n
\rightarrow e^{-\nu_\up^\ep}e^{\nu_\up^\ep}=1.\hspace{4cm}\phantom{mmmmmm}}&&
\EEQS 
Hence, ${\PP}^\ast (\Omega^\ast _\infty)=1$. 
The set  $\Omega^\ast_m$ can be defined in the same way on ${\wi{\wi{\mathfrak{A}}}}$
by setting 
$$
\wi{\wi{\Omega}} _m:=\{ \omega\in \wi{\wi{\Omega}} :  n_{k,m}\le 1: k=0,\cdots, 2^m\},\quad \wi{\wi{\Omega}}_\infty:=\lim_{m\to\infty} \wi{\wi{\Omega}}_m.
$$
 we have $\widetilde{\widetilde{\PP}}(\wi{\wi{\Omega}}_\infty)=1$.
That means the limit of the random measures $\hat{N}_\uq^{m,\ep}$ exists with probability one.
Let us define a random measure by 
\DEQS 
\hat{N}_\uq^{\infty,\ep}(B\times I) & := & \lim_{m\to\infty} \hat{N}_\uq^{m,\ep}(B\times I).
\EEQS 
It remains to show that $\hat{N}_\uq^{\infty,\ep}$ is indeed a Poisson random measure on $(Z_\ep,\CZ_\ep)$ with L\'evy measure $\nu_\ep$.

To show that it is a Poisson random measure, we will first show that for all $C\in \CZ_\ep$ and  time interval $I=[t,s)$, $t,s\in[0,T]$, $t<s$, we have
$$
\wi{\wi{\EE}} \lk[ \hat N^{m,\ep}_\uq( C\times I )\mid \CF_t\rk] = \nu_\ep(C)\,(s-t).
$$
%
This is done in the next steps.
By the construction of the density process we know that
$$
\Leb\lk(\lk\{\alpha: \mathds{1}_{C}(f^{m,\ep}(s, x, \alpha))\rk\}\rk)=
Q_s^{m,\ep}(x,C).
$$
That means, we can write
%
%
\DEQS 
\widetilde{\widetilde{\EE}} \lk[\hat  N^{m,\ep}_\uq( C\times I )\mid\wi{\wi{\CF}} _{t_k^m}\rk] 
&\stackrel{d}{=}&
\widetilde{\widetilde{\EE}} \lk[ \int_I\int_{X} 
\left\{\int_{Z} \mathds{1}_C( z) \, Q^{m,\ep}_s( x, d z)\right\}  \hat N^{\zaehler _m,\ep}_{\up} ( d x,ds)\mid \wi{\wi{\CF}} _{t_k^m}\rk].
\EEQS
However, $\int_{Z} \mathds{1}_C( z) \, Q^{m,\ep}_s( x, d z)=Q^{m,\ep}_s(x,C)$. This gives
\DEQS 
\widetilde{\widetilde{\EE}} \lk[\hat  N^{m,\ep}_\uq( C\times I )\mid \wi{\wi{\CF}} _{t_k^m}\rk] 
&\stackrel{d}{=}& \widetilde{\widetilde{\EE}}  \lk[ \int_I\int_{X} 
 Q^{\ast,\zaehler _m,\ep}_s( x, C )  \hat N^{\zaehler _m,\ep}_{\up} ( d x,ds)\mid \wi{\wi{\CF}} _{t_k^m}\rk].
 \EEQS
 Taking into account that all on the right hand side is $\wi{\CF}_{t}$-measurable (see Definition \eqref{sigmatilde}), we get 
 \DEQS 
\widetilde{\widetilde{\EE}} \lk[\hat  N^{m,\ep}_\uq( C\times I )\mid \wi{\wi{\CF}} _{t_k^m}\rk] 
&\stackrel{d}{=}& \widetilde{\widetilde{\EE}}  \lk[ \int_I\int_{X} 
 Q^{\ast,\zaehler _m,\ep}_s( x, C )  \hat N^{\zaehler _m,\ep}_{\up} (d x,ds)\mid {\wi{\CF}} _{t_k^m}\rk].
 \EEQS
 Now, we know that the laws of $(\hat N^{m,\ep}_\up, Q^{m,\ep})$ on $\wi{\mathfrak{A}}$ and $(\hat N_{\up^\ast}^{\ast,\zaehler _m,\ep}, Q^{\ast,\zaehler _m,\ep})$ are equal on $\mathfrak{A}^\ast$
 are equal. This gives
\DEQS 
\widetilde{\widetilde{\EE}} \lk[\hat  N^{m,\ep}_\uq( C\times I )\mid \wi{\wi{\CF}} _{t_k^m}\rk] 
&\stackrel{d}{=}& \EE^\ast \lk[ \int_I\int_{X} 
 Q^{\ast,\zaehler _m,\ep}_s( x, C )  \hat N^{\ast,\zaehler _m,\ep}_{\up^\ast} ( d x,ds)\mid {\CF}^\ast _{t_k^m}\rk].
 \EEQS
\question{sigma algebra}%
\question{is ${\CF}^\ast _t$}
Now, let us recall how the random measure 
$\hat{N}^{\ast,\zaehler _m,\vareps}_{\up^\ast}$ 
on $\mathfrak{A}^\ast$ is constructed. 
According to the law of total probability, 
the overall probability of an event can be obtained 
by summing over all conditional probabilities 
with respect to a partition of the sample space. 
In our setting, the partition is determined by the number of jumps 
occurring within a single interval $I_k^m$.
\DEQS 
\lqq{ \widetilde{\widetilde{\EE}} \lk[\hat  N^{m,\ep}_\uq( C\times I )\mid\wi{\wi{\CF}} _{t_k^m}\rk] \stackrel{d}{=}\sum_{{n}=0}^{2^m}   \PP^\ast \lk( \mu(Z_\ep\times I)=n\rk)  }
\\
&&\times 
\, \EE^\ast\lk[  \int_{I_{k}^m} \EE^\ast\lk[  Q^{\ast,\zaehler _m,\ep}_s\lk( \frac {l_{k,m}^{\ep,\ast}}n, C\rk ) 
 \mid \CF^\ast _{{t_k^m}}\rk]\, ds \mid    \mu(Z_\ep\times I)=n\rk].
%
\EEQS
Note that, if there is only one jump then we have for $I\in  \mathscr{B}(I_{k,m})$ 
$$
\nu(Z^c_\ep)\EE^\ast \lk[ \int_{I_k^m}   Q^{\ast,\zaehler _m,\ep}_s\lk( {l_{k,m}^{\ep,\ast}}, C\rk ) \mid  \mu(Z^c_\ep\times I)=1\rk]= \frac
{ \nu(C\cap Z_\ep^c)}{ \nu(Z^c_\ep)}\Leb(I)\nu(Z^c_\ep)
.  $$
That means, the first and second term of the sum are fine, only the higher order terms diverge and we can write 
\DEQS
\lqq{  \wi{\wi{\EE}}   \lk[ \hat N^{m,\ep}_\uq( C\times I )\mid \wi{\wi{{\CF}}} _t\rk]-\nu_\ep(C) \Leb(I) =}
\\
&=& \sum_{{n}=2}^{\infty } \,{ 
\PP^\ast \lk(  \eta^\ast(Z_\ep\times I)=n\rk)} 
\\
&&\cdot 
\EE^\ast\lk[  \int_{I_{k}^m}   Q^{\ast,\zaehler _m,\ep}_s\lk( \frac {l_{k,m}^{\ep,\ast} }n, C\rk ) \,ds -\nu_\ep (C)2^{-m }
 \mid {\CF}^\ast  _t \cap \{ \omega\in \Omega^\ast:  \eta^\ast (Z_\ep\times I)=n\} \rk] .
 \EEQS
 Note, that
 \DEQS
 \PP^\ast  \lk(n_{k,m}^{\ep}= \ell \rk)
&=&e^{-\nu(Z_\ep^c )\Leb( I_k^m)}\frac {(\Leb( I_k^m)\cdot \nu(Z_\ep^c))^\ell} {\ell!}   .
\EEQS 
This and the fact that $\Leb( I_k^m)=2^{-m} $, gives,
\DEQS
\lqq{  \wi{\wi{\EE}}   \lk[ \hat N^{m,\ep}_\uq( C\times I )\mid \wi{\wi{{\CF}}} _t\rk]-\nu_\ep(C) \Leb(I) =}
\\
& \le &  \sum_{{n}=2}^{\infty } \frac {( {\nu(Z_\ep^c )}\, 2^{-m} )^n}{n!}\,\exp(-\nu_\ep(C)2^{-m}) \cdot 2\,\nu_\ep(C) 2^{-m} 
\\
& \le & \exp(-\nu(Z_\ep^c )2^{-m}) (\nu(Z_\ep^c )2^{-m})^3
\sum_{{n}=0}^{2^m} \frac{ ({\nu_\ep}\,2^{-m} )^{n}}{n!}\,  2 \nu_\ep(C)2^{-m})\, 2^{-m} 
\\
& \le & \exp(-\nu_\ep2^{-m} )({\nu_\ep}\, \Leb{I_{k}^m})
\sum_{{n}=1}^{2^m} ({\nu_\ep}\, \Leb{I_{k}^m})^n/(n!)) 2\cdot 2^{-m} 
\\
& \le & \exp(-\nu_\ep2^{-m} )({\nu_\ep}\, \Leb{I_{k}^m})
\lk( 1- \exp(\nu_\ep2^{-m} )\rk)  2\cdot 2^{-m} 
\\[5mm]
& \le & \exp(-\nu_\ep2^{-m} )({\nu_\ep}\, \Leb{I_{k}^m})
\nu_\ep 2\cdot 2^{-2m}.
\EEQS
Next, let us calculate the variance of the difference, i.e., 
\DEQS
\lqq{ \wi{\wi{\EE}}   \lk[\lk(  \hat N^{m,\ep}_\uq( C\times I )-\eta(C\times I)\rk)^2 \mid \wi{\wi{{\CF}}}_t\rk]}
\\
&=& \sum_{{n}=2}^{2^m} \, \PP^\ast \lk( \eta^\ast (Z_\ep\times I)=n\rk) 
\\
&& \qquad \cdot \EE^\ast \lk[ \lk(  \int_{I_{k}^m} 
Q^{\ast,\zaehler _m,\ep}_s \lk( \frac {l_{k,m}^{\ep,\ast} }n, C\rk) \,ds -\mu(C\times I)\rk)^2 
\mid \CF_{t_k}\cap \{  \eta^\ast (Z_\ep\times I)=n\}\rk] 
\\
&\le & \exp(-\nu_\ep2^{-m} ) \sum_{{n}=2}^{2^m} \, ({\nu_\ep}\, \Leb{I_{k}^m})^n/(n!))
n^2
\\
&\le & {2} \exp(-\nu_\ep2^{-m} )({\nu_\ep}\, \Leb{I_{k}^m})^2  \sum_{{n}=0}^{2^m-2} \, (\nu_\ep\, \Leb{I_{k}^m})^n/(n!))
\\
&\le & {2} ({\nu_\ep}\, \Leb{I_{k}^m})^2 .
\EEQS
\noindent We have shown that for any $\zaehlerz\in\NN$, $\hat N_\uq^{\vareps}$  is a time homogeneous Poisson random measure with L\'evy measure $ \nu_\vareps$ over $\wi{\wi{\mathfrak{A}}}$.
It remains to show that the integration is well defined. We have to show that 
the following sequence
\DEQSZ\label{setdefine}
\lk\{\int_0^T  \int_{ Z_{\epsilon_{\zaehlerz}}^c}
 g(\wi{\wi{w}}_m) d\hat N^{\vareps}_\uq(dz,ds):\zaehlerz\in\NN\rk\}
\EEQSZ 
is a Cauchy sequence in $$\mathcal{M}^2_{\wi{\wi{\mathfrak{A}}}}(L^2(0,T;V)). $$
That means 
\DEQS
\wi{\wi{\EE}}  \lk| \int_0^T  \int_{ Z_{\epsilon_{\zaehlerz_2}}^c}
 g(\wi{\wi{w}}_m,z) d\hat N^{\epsilon_{\zaehlerz_2}}_\uq(dz,ds)-\int_0^T  \int_{ Z_{\epsilon_{\zaehlerz_1}}^c}
 g(\wi{\wi{w}}_m,z) d\hat N^{\epsilon_{\zaehlerz_1}}_\uq(dz,ds)\rk|^2.
 \EEQS 
 Let us note that we can write by linearity the difference above as follows
 \DEQS
\int_0^T  \int_{ Z_{\epsilon_{\zaehlerz_2}}^c}
 g(\wi{\wi{w}}_m,z) d\hat N_\uq^{\epsilon_{\zaehlerz_2}}(dz,ds)-\int_0^T  \int_{ Z_{\epsilon_{\zaehlerz_1}}^c}
 g(\wi{\wi{w}}_m,z) 
 d\hat N^{\epsilon_{\zaehlerz_1}}_\uq(dz,ds)
 \\
 =  \sum_{N=\zaehlerz_1}^{\zaehlerz_2} 
 \int_{ Z_{\epsilon_{N}}^c \setminus Z_{\epsilon_{N-1} }^c }
  g(\wi{\wi{w}}_m,z) 
  d\hat N^{\epsilon_{\max(\zaehlerz_1,\zaehlerz_2)}}_\uq(dz,ds).
 \EEQS 
Hence, due to the It\^o isometry and assuming that $\zaehlerz_2>\zaehlerz_1$, we obtain for the entity 
\DEQS 
\lqq{ \wi{\wi{\EE}}  \sum_{N=\zaehlerz_1}^{\zaehlerz_2} \int_{ Z_{\epsilon_{N}}^c \setminus Z_{\epsilon_{N-1}}^c}
 |g(\wi{\wi{w}}_m,z)|_V^2 \nu_{ \epsilon_{\max(\zaehlerz_1,\zaehlerz_2)}}(dz)}
 \\
&\le& \wi{\wi{\EE}} \sum_{N=\zaehlerz_1}^{\infty} \int_{ Z_{\epsilon_{N}}^c \setminus Z_{\epsilon_{N-1}}^c}  |g(\wi{\wi{w}}_m, z)|_V^2 \nu(dz) 
\\
&=& 
\wi{\wi{\EE}} \int_{Z_{\epsilon_{\min(\zaehlerz_1,\zaehlerz_2)}}}   |g(\wi{\wi{w}}_m, z)|_V^2 \nu(dz)
.
\EEQS 
Since we know that
$$
\wi{\wi{\EE}} \int_{Z}   |g(w^\ast _m)|_V^2 \nu(dz)<\infty,
$$
we can take $\min(\zaehlerz_1,\zaehlerz_2)$ large enough, such that the entity
$$\wi{\wi{\EE}} \int_{Z_{\epsilon_{\min(\zaehlerz_1,\zaehlerz_2)}}}   |g(\wi{\wi{w}}_m, z)|_V^2 \nu(dz)$$
is smaller than $\epsilon$. Since we are on a bounded time interval, 
it follows that the set in \eqref{setdefine} is a Cauchy sequence 
in $\mathcal{M}^2_{\wi{\wi{\mathfrak{A}}}}(L^2(0,T;V))$.


It remains to analyse the filtration generated by the processes 
and to ensure that the Poisson random measure possesses independent increments. 
Such properties, including independence, can be verified through the corresponding distributions.
This concludes the proof.

\medskip


Next, we verify several statements with the goal of passing to the limit. We point out that the same construction used for $w^\ast_{\zaehler _j,\infty}(\cdot)$ can be carried out on the original probability space $\mathfrak{A}$.
The resulting process will be denoted by $w_{\zaehler _j,\infty}(\cdot)$.
Due to the construction and properties of the projection, it is straightforward to verify that the laws are preserved. In particular, we have $\Law(\wi{w}_{\zaehler _j,\infty}) = \Law(w^\ast_{\zaehler _j,\infty})$. 


Let us recall the details we have shown up to now.
We have constructed 
\begin{itemize}
    \item an extension  $\wi{\wi{\mathfrak{A}}}$ over the probability spaces ${\wi{\mathfrak{A}}}$,
    \item  together with a a cylindrical Wiener process $\wi{\wi{W}}$ on $\mathcal{H}$ 
    \item and a Poisson random measure $\hat N_\uq$ on $Z$ with L\'evy measure $\nu$, both over  $\wi{\wi{\mathfrak{A}}}$,
\end{itemize} 
such that the processes $\wi{\wi{w}}_{\zaehlerz_j}$ over $\wi{\wi{\mathfrak{A}}}$
are fixed points of the operator $ \hat{\mathcal{V}}^{\zaehler _j}_{\wi{\wi{\MA}},\wi{\wi{W}},\hat N_\uq} $.
To be more precise, we have for any $j\in\NN$,
$$
\wi{\wi{\PP}}\lk( \wi{\wi{w}}_{\zaehlerz_j}= \hat{\mathcal{V}}^{\zaehler _j}_{\wi{\wi{\MA}},\wi{\wi{W}},\hat N_\uq}\rk)=1.
 $$
 In addition, we know that there exists an element  $\wi{\wi{w}}^\infty$ over $\wi{\wi{\mathfrak{A}}}$, such that 
 $\wi{ \widetilde{\PP}}$-a.s.\ the sequence of processes $\{\wi{\wi{w}}_{\zaehlerz_j}:j\in\NN\}$
 converges to $\wi{\wi{w}}^\infty$ in the topology of $\mathbb{X}$. 
In addition, we have for all $j\in\NN$, 
$$
\Law(\wi{\wi{w}}_{\zaehlerz_j})=\Law( {w} ^\ast _{\zaehlerz_j}).
$$
{\color{blue} check, we twice relabellled}
In addition, the following claim can be shown.
%
for $r\in(1,m_1)$
 \begin{claim}\label{claim222}
\begin{itemize}
\item There exists a constant $C>0$ such that $ \sup_{k\in{\mathbb{N}}}  \wi{\wi{\EE}} \left[ \| \wi{\wi{w}}_{\zaehler_j,\infty}
\|^{m}_{\mathbb{X}}\right]\le C$ and
\item For any $r\in (1,m]$ we have
$$\lim_{k\to \infty}{{\wi{\wi{\mathbb{E}}} }}\left[   \| \wi{\wi{w}}_{\zaehler_j,\infty}- \wi{\wi{w}} \|_{\mathbb{X}}^r\right] = 0. $$
\end{itemize}
\end{claim}
\begin{proof}[Proof of Claim \ref{claim222}:]
Clearly, since $ w^\ast_{\zaehler_j,\infty}$ is a fixed point of the operator
$\hat{\mathcal{V}}^\kl_{\MA^\ast,W^\ast,\eta^\ast}$, we know that 
$w^\ast_{\zaehler_j,\infty}\in \hat{\mathcal{V}}^\kl_{\MA^\ast,W^\ast,\eta^\ast}( \Kcal_R(\MA^\ast))$.
Since
 $\{ w^\ast_{\zaehler_j,\infty}\}_{j\in\NN}\subset \hat{\mathcal{V}}^{\zaehler_j}_{\MA^\ast,W^\ast,\eta^\ast }\lk(\Kcal ( \MA^\ast )\rk)$, we know from (v), that for any $r\le m_1$ (where $m_1>m$) there exists some constant $K>0$ (see iv) such that 
$$ 
\wi{\wi{\EE}} \lk\|\wi{\wi{w}}_{\zaehler_j,\infty} \rk\| _{\mathbb{X}}^r\le K,
\quad \forall j\in\NN ,
$$
%
%
%
Hence, we know that $\{\| w^\ast_{\zaehler_j,\infty}\|_{\X}^{r}\}$
is uniformly integrable for any $r\in [1,m]$ w.r.t. the probability measure ${\P^\ast}$.
From before, ${w}^\ast_{\zaehler_j,\infty}\to{w}^\ast$ ${\P}^\ast $-a.s., so
we get by the Vitali convergence theorem that for $j\to\infty$ 
\begin{equation}
\lim_{j\to\infty}\wi{\wi{{\Eb}}} \left\| \wi{\wi{w}} _{\zaehler_j,\infty}-\wi{\wi{w}}\right\|_{\X}^{r}=0\label{eq:strong-v}
\end{equation}
for any $r\in[1,m]$.
\end{proof}

\begin{remark}
    Here, the assumption (v) is essential. To be more precise, due to assumption (v), we have uniformly integrability, such that we get convergence for the $m$-moment.
\end{remark}

\item
In this step we show that $\wi{\wi{w}}^\infty$ over $\wi{\wi{\MA}}$ together with the Wiener process $\wi{\wi{W}}$ 
and Poisson random measure is indeed a fixed point to the operator  $ {\mathcal{V}}_{\wi{\wi{\MA}},\wi{\wi{W}},\hat N_\uq} $.
In this step, we will show that for all $\epsilon>0$ we have
$$\EE \lk| \wi{\wi{w}}^\infty- {\mathcal{V}}_{\wi{\wi{\MA}},\wi{\wi{W}},\hat N_\uq} (\wi{\wi{w}}^\infty)\rk|_{\mathbb{X}}\le \epsilon .
$$
We first decompose the difference in to the following sum 
\DEQS
\lqq{\wi{\wi{w}}^\infty-  {\mathcal{V}}_{\wi{\wi{\MA}},\wi{\wi{W}},\hat N_\uq} (\wi{\wi{w}}^\infty)
}&&
\\
&=&\underbrace{\wi{\wi{w}}^\infty -\wi{\wi{w}}_{\zaehler _j,\infty}}_{:=I}
+
\underbrace{\wi{\wi{w}}_{\zaehler _j,\infty} -\hat{\mathcal{V}}^{\zaehler _j}_{\wi{\wi{\MA}},\wi{\wi{W}},\hat N_\uq}  ( \wi{\wi{w}}_{\zaehler _j,\infty})}_{=:II}
\\
&&{}+
\underbrace{\hat{\mathcal{V}}^{\zaehler _j}_{\wi{\wi{\MA}},\wi{\wi{W}},\hat N_\uq}  ( \wi{\wi{w}}_{\zaehler _j,\infty})-\hat{\mathcal{V}}^{\zaehler _j}_{\wi{\wi{\MA}},\wi{\wi{W}},\hat N_\uq}  ( \wi{\wi{w}}^\infty )}_{=:III}
+
\underbrace{\hat{\mathcal{V}}^{\zaehler _j}_{\wi{\wi{\MA}},\wi{\wi{W}},\hat N_\uq}  ( \wi{\wi{w}}^\infty ) -\mathcal{V}_{\wi{\wi{\MA}},\wi{\wi{W}},\hat N_\uq} ( \wi{\wi{w}}^\infty }_{=:IV}.
\EEQS
Next, we will use  the triangle inequality and investigate each component separately.
In the following lines, we analyse the preceding terms $I$, $II$, $III$, and $IV$.

Fix $r\in [1,m]$.
Note, due to the claim \ref{claim222} we know
$$\wi{\wi{\EE}} \| \wi{\wi{w}}^\infty -\wi{\wi{w}}_{\kl_j}\|_{\mathbb{X}}^m\le \frac \ep 3.
$$

Next, to tackle II, we first know that we have equality in the laws of $\wi{\wi{w}}_{\zaehlerz_j}$ and $w^\ast _{\zaehlerz_j}$ for all $ j\in\NN$. Secondly, we know that
due to the well posedness (see Theorem \ref{ther_main}-(i)) and we know  by the step before that $\hat{\mathcal{V}}^{\zaehler _j}_{\MA^\ast,W^\ast,\eta^\ast}( w^\ast_{\zaehler _j,\infty})$ and $w^\ast_{\zaehler _j,\infty}$ indistinguishable sind, in particular, 
\DEQSZ\label{istwichtig}
\EE^\ast \lk\| \hat{\mathcal{V}}^{\zaehler _j}_{\MA^\ast,W^\ast,\eta^\ast}( w^\ast_{\zaehler _j,\infty})-w^\ast_{\zaehler _j,\infty}\rk\|^m_{\mathbb{X}}=0.
\EEQSZ 
To tackle III, we again use the equality in the laws of $\wi{\wi{w}}_{\zaehlerz_j}$ and $w^\ast _{\zaehlerz_j}$ for all $ j\in\NN$. Next, due to the continuity of the operator 
$\hat{\mathcal{V}}^{\zaehler _j}_{\MA^\ast,W^\ast,\eta^\ast}$  (see Theorem \ref{ther_main}-(iii)), we know that there exists  a function $\Phi$ with $\lim_{x\to 0}\phi(x)=0$, such that
$$
\EE^\ast  \lk\|\hat{ \mathcal{V}}^{\zaehler _j}_
{\MA^\ast,W^\ast,\eta^\ast}
( w^\ast_{\kl_j,\infty})-\hat{\mathcal{V}}^{\zaehler _j}_
{\MA^\ast,W^\ast,\eta^\ast}(w^\ast)\rk\|_{\mathbb{X}}^m
\le C\phi\lk\{ \lk(\EE\lk\|\tilde w^\ast_{\zaehler _j,\infty}-w^\ast\rk\|_{\mathbb{X}}^m\rk)^{\frac 1m}\rk\}
$$
Finally, since $\hat{\mathcal{V}}^{\zaehler _j}_{\MA^\ast,W^\ast,\eta^\ast}=\hat \Pro_{\zaehler _j}{\mathcal{V}}_{\MA^\ast,W^\ast,\eta^\ast}$ and the image of $\mathcal{V}_{\MA^\ast,W^\ast,\eta^\ast}$ is compact (see Theorem \ref{ther_main}-(v)), the difference
$$
\EE^\ast   \lk\|\hat{\mathcal{V}}^{\zaehler _j}_{\MA^\ast,W^\ast,\eta^\ast}(w^\ast)-\mathcal{V}_{\MA^\ast,W^\ast,\eta^\ast}(w^\ast)\rk\|_{\mathbb{X}}^m
$$
tends to zero (see Remark~\ref{projection}) and there exists some $j_0\in\NN$ such that for all $k\ge j_0 $ we have
$$
\EE^\ast   \lk\|\hat{\mathcal{V}}^{\zaehler _k}_{\MA^\ast,W^\ast,\eta^\ast}(w^\ast)-\mathcal{V}_{\MA^\ast,W^\ast,\eta^\ast}(w^\ast)\rk\|_{\mathbb{X}}^m\le \frac \epsilon 6.
$$
Finally, IV tends to zero, due to the continuity of the operator $\mathcal{V}_{\wi{\wi{\MA}},\wi{\wi{W}},\hat N_\uq}$. 
Here, we have to point out that the assumptions (i), (ii), (iii), (v), and (vi) have to be verified for any probability space where the Wiener process and the Poisson random measure are given.  

As a consequence, we have
\[
\mathcal{V}_{\wi{\wi{\MA}},\wi{\wi{W}},\hat N_\uq} (\wi{\wi{w}}^\infty)
     =\wi{\wi{w}}^\infty,
     \quad \wi{\wi{\PP}} \mbox{-a.s.}
     \]
As seen above, $\wi{\wi{w}}^\infty \in\Xcal(\wi{\wi{\Afrak}}^\ast)$,
so that by (v), $\mathcal{V}_{\wi{\wi{\MA}},\wi{\wi{W}},\hat N_\uq} (\wi{\wi{w}}^\infty) \in\D([0,T];U)$, and
therefore $\wi{\wi{w}}^\infty \in\D([0,T];U)$ ${\wi{\wi{\PP}}}$-a.s. Hence for
all $t\in[0,T]$, $\wi{\wi{\PP}} $-a.s.
\[
\mathcal{V}_{\wi{\wi{\MA}},\wi{\wi{W}},\hat N_\uq} (\wi{\wi{w}}^\infty)(t)=\wi{\wi{w}}^\infty(t)
\]
 and the proof is complete. Let us define the compensator of $ N_\uq$ by $\Gamma=\nu\times\Leb$.
 By the definition 
of  $\mathcal{V}_{\wi{\wi{\MA}},\wi{\wi{W}},\hat N_\uq} $, we see that $\wi{\wi{w}}^\infty$ solves
 \DEQSZ\label{spdes_infact}
 \begin{aligned}
\lqq{ d \wi{\wi{w}}^\infty(t) =\lk(\DeltaA \wi{\wi{w}}^\infty(t)+ F(\wi{\wi{w}}^\infty,t)\rk)\, dt 
} &
\\&{}+\Sigma( \wi{\wi{w}}^\infty(t))\,d\wi{\wi{W}}(t)+\int_ZG(\wi{\wi{w}}^\infty(t),z)\, (N_\uq-\Gamma)(dz,dt),\quad  \wi{\wi{w}}^\infty(0)= w_0,
 \end{aligned}
\EEQSZ
on $\wi{\wi{\Afrak}}$,
where $ w_0$ is a 
{${\wi{\wi{\Gcal}}}_0\otimes{{{\wi{\wi{\Fcal}}}}}_0$-}measurable version of $w_0$.
\end{step}
\end{proof}

\appendix

\section{Preliminaries on Poisson random measures and L\'evy processes}
\label{levy}
In addition, the term of Poisson random measure is sometimes defined in another way, starting with the intensity measure and defining the Poisson random measure with given intensity measure. However, we prove in the next lemma the equivalence between both definitions.

\begin{lemma}\label{eqdefprm}(see \cite{martin}) \refy{Let $S$ be a separable metric space.}
A measurable mapping $\eta:\Omega\to{M_{\bar{\Bbb N}}(\{Z_n\times\Int \})}$ is a time-homogeneous Poisson random measure with intensity  $\nu$ iff
\begin{letters}\label{prmpoints}
\item for any $\AAA \in\CZ$ with $\nu(U)<\infty$, the random variable $N (t,\AAA ) $ is
Poisson distributed with parameter $t\,\nu(\AAA )$, otherwise $\PP\lk( N(t,U)=\infty\rk)=1$;
\item for any $n$ and
disjoint sets $\AAA _1,\AAA _2,\ldots,\AAA _n\in\CZ$, and any $t\in[0,T]$, the random
variables $N (t,{\AAA _1}) $, $N (t,{\AAA _2}) $, \ldots, $N (t,{\AAA _n}) $  are  mutually independent;
\item the ${M_{\bar{\Bbb N}}(\{S_n\})}$-valued process  $(N (t,\cdot))_{t\in \Int }$ is adapted to $\BF$; 
%
\item for any
$t\in[0,T]$, $\AAA \in\CS$, $\nu(U)<\infty$, and any $r,s\ge t$, the random variables
$N (r,\AAA )-N (s,\AAA )$ are independent of $\CF_t$.
\end{letters}
\end{lemma}

\begin{definition}\label{def:levy}(see Linde \cite[Chapter 5.4]{linde})
Let $E$ be a separable Banach space and let $E^\prime$ be its
dual.  A symmetric\footnote{i.e. $\lambda(A)=\lambda(-A)$ for all
$A\in\CB(E)$,} $\sigma$-finite Borel measure $\lambda$ on $E$ is
called a {\sl L\'evy measure} if and only if
\begin{trivlist}
\item[(i)]
 $\lambda(\{0\} )=0$, and
\item[(ii)]
 the
function\footnote{As remarked in Linde \cite[Chapter
5.4]{linde} we do not need to suppose that the integral $\int_E
(\cos\langle x,a\rangle -1) \; \lambda(dx)$ is finite. However, see
Corollary 5.4.2 in ibidem, if $\lambda$ is a symmetric L\'evy
measure, then, for each $a \in E^\prime$, the integral in
question is finite. }
$$
E^\prime \ni a\mapsto  \exp \lk( \int_E (\cos\langle x,a\rangle -1)
\; \lambda(dx)\rk)
$$
is the characteristic function of a Radon measure on $E$.
\end{trivlist}
An arbitrary $\sigma$-finite Borel measure $\lambda$ on $E$ is
called a L\'evy measure provided its symmetric part
$\frac12(\lambda+\lambda^-)$, where $\lambda^-(A):=\lambda(-A)$,
$A\in\CB(E)$, is a
L\'evy measure.
The class  of all L\'evy measures on $(E,\CB(E))$ will be denoted by
$\CL (E)$.
\end{definition}
\begin{remark}(see e.g.\ \cite{linde})
If $E$ is a Banach space of type $p$, a measure $\nu\in\CM_+(E)$ is a L\'evy measure iff
$\nu(\{0\})=0$ and $\int_E|z|^p\nu(dz)<\infty$.
\end{remark}
%
Suppose $E$ is a Banach space of martingale type $p$.
Therefore, for a   time-homogeneous Poisson random measure $\eta$
on  $E$ with intensity measure $\nu\in\CL(E)$, the process
$$
L(t) :=
\int_0 ^ t \int_E z \, \tilde \eta(dz,ds),\quad t\ge 0,
$$
is an $E$--valued L\'evy process with triplet $(0,m,\hat \nu)$, where
$$\hat \nu=\nu\big|_{\{x\in E, |x|\le 1\}} \,\, \mbox{ and }\,\,
m=\int_{\{x\in E,|x|\ge 1\}} z\, \nu(dz). $$
For more details about the connection of Banach spaces of type $p$ and
stochastic integration we refer to Dettweiler \cite{dettweiler} or {Peszat and Zabczyk \cite{Peszat_Z_2007}. 

Usually, one starts with specifying the measurable space $(S,\CS)$ and the intensity measure $\nu$ on $(S,\CS)$. Given this, then there exists a
Poisson random measure {on $(S,\CS)$ having} the intensity measure $
\nu$.

\medskip

In order to define a stochastic integral with respect to the Poisson random measure,
$S$ has to be related to a topological vector space and the measure $\nu$ has either to be finite or has to be a L\'evy measure, for the definition see \cite[Chapter 5.4]{linde}). 
\begin{remark}
Let $Z$ be a separable Banach space, and $\CZ$ its Borel $\sigma$--algebra.
If the intensity measure
$
\nu:\CZ\to\RR
$
satisfies the integrability condition
$$
\sup_{\substack{a\in Z^\ast\\ |a|\le 1}}  \int_{Z} (1\wedge |\la z,a\ra |^2)\,  \nu(dz)<\infty.
$$
then $\nu$ is a \levy measure (see \cite[Proposition 5.4.1, p.\ 70]{linde}). 
\end{remark}

For some Banach spaces, one can characterize the \levy measures in a more precise way. Therefore, let us introduce the following definition. Let $\{\ep_k:k\in\NN\}$ be a sequence of independent, identically distributed  random variables with $\PP\lk( \ep_1=1\rk)=\PP\lk( \ep_1=-1\rk)=\frac12$. Then a Banach space with norm $|\cdot|$ is of $R$--type $p$,
(Rademacher type $p$), where $1\le p\le 2$, if for any sequence $\{ x_j:j\in \NN\}$ belonging to $l_p(E)$, we have (compare \cite[p.\ 40]{linde})
$$\PP\lk(\Big| \sum_{j=1}^\infty \ep_j x_j\Big|<\infty\rk)=1
.
$$
The Minkowski inequality implies that each Banach space is of $R$--type $1$.

\begin{remark}\label{ourass1}
Let $Z$ be a Polish space, $\CZ$ the Borel $\sigma$--algebra (in the sequel we call $(Z,\CZ)$ just a  Polish space). The family $\{Z_n\in\CZ\}$ satisfy $Z_n\uparrow Z$, and $\nu$ be a $\sigma$--finite measure with $\nu(Z_n)<\infty$ for any $n\in\NN$.
Fix $p\in[1,2]$. We assume that $E$ {is} a separable Banach space of $R$--type $p$, and that
$\xi:(Z,\CZ)\to (E,\CB(E))$ {is} a  measurable mapping.
In addition, we assume that the intensity measure
$
\nu:\CZ\to\RR_0^+
$
{satisfies} the integrability condition
\DEQSZ\label{heregrow}
  \int_{S} 1\wedge | \xi(z) |_E^p\,  \nu(dz)<\infty,\quad  \mbox{and} \quad  \nu(\{0\})=0.
\EEQSZ
Then, the measure $\nu_E$ induced by $\xi$ on $E$ is a \levy measure (and $\nu_E(\{0\}):=0$) (compare \cite[p.\ 75]{linde}). In addition, if $\eta$ is a Poisson random measure with intensity $\nu$
over a filtered probability space $(\Omega,\CF,\BF,\PP)$,  the process
$$
L:\INT 
\ni t \mapsto \int_0^ t \int_Z \xi(z\,)\, ( \eta-\nu\otimes \lambda)(dz,ds)
$$
is a \levy process over  $(\Omega,\CF,\BF,\PP)$.
\end{remark}

Hence, from now on we assume during the whole paper that the following convention is valid.

\begin{convention}\label{ourass}
We stipulate that $(Z,\CZ)$ is a Polish space, 
$\nu$ a $\sigma$--finite measure on $(Z,\CZ)$ and $Z_n\in\CZ$ such that $Z_n\uparrow Z$ and $\nu(Z_n)<\infty$ for every $n\in\Bbb N$.
\end{convention}

Let us consider a filtered  probability space $(\Omega,\CF,\BF,\PP)$, 
where $\BF=\{\mathcal{F}_t\}_{t\in \INT }$ denotes a filtration. 
A process $\xi:[0,T]\times \Omega\to X$ is progressively measurable, or simply, progressive, if its restriction to $\Omega\times [0,t]$ is $\CF_t\otimes\CB([0,t])$--measurable for any $t\ge 0$.
The predictable random field $\CP $ on $\Omega\times
\RR _ +$ is the $\sigma$--field generated by all continuous
$\BF$--adapted processes (see e.g.\ Kallenberg \cite[Chapter~25, p.\ 491]{kallenberg}).

A real valued stochastic process $\{x(t):t\in\INT \}$, defined on a
filtered probability space $(\Omega,\CF,\BF,\PP)$ is called
predictable, if the mapping  $x:\Omega\times\Int \to\Rb{}$ is
$\CP/\mathcal{B}(\RR)$-measurable. A random measure $\gamma$ on {$\mathcal Z\otimes\mathcal B(\Int )$}
over $(\Omega;\CF,\BF,\PP)$ is called
predictable, iff for each $U\in \CS$, the $\RR$--valued process
$\Int \ni t\mapsto \gamma( U\times(0,t])$ is predictable.

\begin{definition}\label{def-imPrm}
Assume that  $(Z,\CZ )$ is a measurable space and $\nu$ is a
non-negative $\sigma$--finite measure on $(Z,\CZ )$. 
Assume that $\eta$ is a time-homogeneous Poisson
random measure  with  intensity measure $\nu$ on $(Z,\CZ)$
over $(\Omega,\CF,\BF,\PP)$.  

The {\sl compensator} of $\eta$ is the unique predictable random measure, 
denoted by $\gamma$, on {$\mathcal Z\otimes\mathcal B(\Int )$} over {$(\Omega,\CF,\BF,\PP)$}, such that
for each fixed $T<\infty$ and {$B\in \CZ $} satisfying 
$\EE\eta( B\times(0,T])<\infty$, the $\mathbb{R}$-valued processes
$\{\tilde{N}(t,B)\}_{t\in (0,T]}$  defined by
$$
t\mapsto \tilde{N}(t,B):= \eta(  B\times (0,t] )-\gamma( B\times (0,t] ),
\quad   0< t\le T,
$$
is a martingale on $(\Omega,\CF,\BF,\PP)$.
\end{definition}
\begin{remark}
Assume that $\eta$ is a time-homogeneous Poisson random measure
with intensity $\nu$ on  $(S,\CS)$ over $(\Omega,\CF,\BF,\PP)$. It
turns out that the compensator $\gamma$ of $\eta$ is uniquely determined and moreover
$$
\gamma: \CZ \times \CB[0, T] \ni (B,I)\mapsto  \nu(B)\times
\Leb(I).
$$
Here $\lambda$ denotes the Lebesgue measure on $\RR$.
The  difference between a time-homogeneous  Poisson random measure
$\eta$  and its compensator $\gamma$, i.e.  $\tilde
\eta=\eta-\gamma$, is called a  {\sl compensated Poisson random
measure}.
\end{remark}

Let  $(S,\CS)$ be a measurable space
and
let $\eta$ be a time-homogeneous Poisson random measure on $S$ with intensity  measure $\nu$ being a positive $\sigma$--finite measure 
 over
$\mathfrak{A}$ satisfying Convention \ref{ourass}.
We will denote by $\tilde{\eta }$ the \it  compensated Poisson random measure \rm  defined by
$\tilde{\eta }: = \eta - \gamma $, where the compensator
 $\gamma : \bcal (\Int )\times \CZ  \to {\Int }$ satisfies in our case the following equality
$$
     \gamma (I \times B) = \Leb (I) \nu (B) ,
\qquad I \in \bcal ({\Int }) , \quad  B \in \CZ .
$$

Let us remind that the number of jumps of a Poisson random measure $\mu$ with Levy measure $\nu$ on $E$ is given by the formula
\DEQSZ\label{PRMformula}
	\PP(\mu(A\times I)=l) &=& e^{-\mu(A) \Leb(I)} \frac {(\nu(A)\Leb(I))^l}{l!},\quad A\in \mathscr{B}(E),\, I\in \mathscr{B}(I_k^m).
\EEQSZ

\begin{lemma}\label{poiss_uni_distr}
Let \refy{$(S,\CS)$ be a measurable space and} $\nu$ be a non--negative $\sigma$--finite  measure on $S$  satisfying Convention \ref{ourass}. 
Then the following holds
\begin{enumerate}
  \item there exists a probability space $\mathfrak{A}=(\Omega,\CF,\PP)$ and a time-homogeneous Poisson random measure $\eta:\Omega\to {M_{\bar{\Bbb N}}(\{Z_n\times
  \Int \})}$ with the intensity measure $\nu$;
  \item Denote by $\Theta_\nu$ the law of $\eta$ on ${\mathcal M_{\bar{\Bbb N}}(\{Z_n\times\Int \})}$.  If $\eta^\sharp$ is a time-homogeneous Poisson random measure defined possibly on different stochastic base denoted by $\mathfrak{A}^\sharp=(\Omega^\sharp,\CF^\sharp,\PP^\sharp)$ and $\nu$ is the intensity measure for $\eta^\sharp$ then $\Theta_\nu$ is the law of $\eta^\sharp$ on ${\mathcal M_{\bar{\Bbb N}}(\{Z_n\times\Int \})}$.
\end{enumerate}
\end{lemma}

\begin{proof}
Part i.) is given by Theorem~8.1 \cite[p.\ 42]{ikeda}. It remains to show ii.).
Since $\nu$ is $\sigma$--finite, there exists a increasing family $\{Z_n:n\in\NN\}$ with $Z_{n+1} \supseteq Z_n$, $Z_n\uparrow Z$, and $\nu(Z_n)<\infty$.
To show that $\eta$ and $\eta^ \sharp$ have the same law on $\CM_{\bar \NN}(Z\times\Int )$, it is sufficient to show that for all $f:Z\times \Int \to\RR$ bounded and continuous, 
the random variable $\eta(f):= \int_{Z_n}\int_0^ T  f(z,t)\,\eta(dz,dt)$ and $\eta^ \sharp(f):= \int_{Z_n}\int_0^T f(z,s)\, \eta^ \sharp(dz,ds)$ have the same law, see \cite[Theorem~5.8, p.\ 38]{parth}.
Since $Z\times \RR^ +$ is a Polish space,  the $\sigma$ algebra generated by the family of bounded continuous functions coincides with the Borel--$\sigma$--algebra, see \cite[Proposition 1.4, p.5]{vakhania}. Therefore, it is sufficient to show for all $n\in\NN$, $U\in\CB(Z_n)$ and $I\in\CB(\Int )$, that the random variables
$\eta(U\times I)$ and $\eta^ \sharp(U\times I)$ have the same law.
Let $\Theta_\nu^\sharp$ be the law of $\eta_\sharp$ and let us assume $\nu(U),\Leb(I)<\infty$.
Let $k\in\NN_0$. Then, by the definition of the Poisson random measure and its intensity measure $\nu$ we know that
\DEQS
\Theta_\nu( \eta(U\times I])=k) &= & e^{-\Leb(I)\nu(U)}\, {\frac{(\nu(U)\,\Leb(I))^k}{ k!}}
\\
&=&\Theta_\nu^\sharp( \eta_\sharp(U\times I)=k).
\EEQS
If $\nu(U)=\infty$ or $\Leb(I)=\infty$, then $\Theta_\nu\lk( \eta(U\times I)=\infty\rk)=1=\Theta_\nu^ \sharp\lk( \eta(U\times I)=\infty\rk)$.
\end{proof}

Now, one can define the stochastic integral with respect to the Poisson random measure for   progressively measurable integrands,
introduced e.g.\  in
  \cite{Brz+Haus_2009} in $M$--type $p$ Banach spaces.

\begin{definition}\label{def-mart-intext}
Let $0<p\le 2$. A Banach space $E$ is of  martingale type $p$  iff
there exists a constant $C>0$ 
such that for all
$E$-valued  finite martingale $\{M_{n}\}_{n=0}^N$  the
following inequality holds
\DEQSZ\label{nnn} \sup_{0\le n\le N } \mathbb{E} | M_{n} |_E ^{p} \le  C\, 
\mathbb{E}  \sum_{n=0}^N  | M_{n}-M_{n-1} |_E ^{p},
\EEQSZ 
 where as  usually, we set  $M_{-1}=0$.
\end{definition}

Examples of $M$--type $p$ Banach spaces are, e.g.\ $L^q(\CO)$ spaces, where $\CO$ is a bounded domain. $L^q(\CO)$ is of $M$--type $p$ for any $p\le q$ (see e.g.\ \cite[Chapter 2, Example 2.2]{mtypep}).
If a Banach space $E$ is of $M$--type $p$ and $A$ is the generator of an analytic semigroup on $E$, then the complex interpolation spaces between $D(A)$ and $E$ are of $M$--type $p$.
Similar facts hold also for real interpolation spaces, but not in this generality, for more details we refer to Appendix A of \cite{zmtypep}.

In addition, we would like to point out in the following Proposition,  that we do not need to suppose that the filtration of the given probability space is right continuous.
In particular,  given a Poisson random measure $\eta$ over a filtered probability space $(\Omega,\PP,\CF,\BF)$, $\BF=(\CF_t)_{t\in \INT}$, with an arbitrary filtration,
 a progressively measurable  $L^2(S,\nu)$--valued process $\xi$, one can pass to the right continuous augmentation of the filtration without loosing the necessary  properties.
In fact, following remark can be easily shown.

\begin{remark} Let us assume that  $\eta$ is a Poisson random measure also for the  augmented right continuous filtration $\bar \CF_t:= \wedge _{h>0}\mathcal F^{\Bbb P}_{t+h}$. Then, we can construct the stochastic integral on $(\Omega,\mathcal F,(\mathcal F_t),\Bbb P)$, and the stochastic integral on $(\Omega,\mathcal F,(\bar\CF_t),\Bbb P)$. In particular, let
$$
I_1=\int_0^t\int_S\xi(r,x)\tilde\eta(dr,dx),
$$
be the integral defined by by the progressible approximation on $(\CF_t)_{t\ge 0}$ and
let 
$$ I_2=\int_0^t\int_S\xi(r,x)\tilde\eta(dr,dx).
$$
be the integral defined by by the progressible approximation on  $(\bar \CF_t)_{t\ge 0}$.
To be more precise, the difference between these two integral is, that in the integral on the left hand, we took for the predictable sequence of simple functions $(\xi_n)$ converging to $\xi$ the underlying filtration $(\CF_t)_{t\ge 0}$ and
in the integral on the right hand we took for  predictable  sequence of simple functions $(\xi_n)$ converging to $\xi$ the underlying filtration $(\bar \CF_t)_{t\ge 0}$. By the definition of the integral both are identical, i.e., $I_1=I_2$.
\end{remark}

\subsection{Extension of a filtered probability space}\label{extension}

To construct a Poisson random measure it is necessary to introduce additional random variables on an auxiliary probability space and then extend the original probability space by means of this auxiliary space.
In what follows, we make precise the notion of an extension of a probability space (compare with \cite{ikeda}).

\begin{definition}
	\label{Definition7.1.} We say a probability space $(\tilde{\Omega}, \tilde{\mathscr{F}}, \tilde{\mathbb{P}} )$ with a filteration $(\tilde{\mathscr{F}}_{t})_{t\ge 0}$ is an extension of a probability space $(\Omega, \mathscr{F}, \mathbb{P})$ with a  filteration $\left(\mathscr{F}_{t}\right)$,  if there exists a mapping $\pi: \tilde{\Omega} \longrightarrow \Omega$ which is $\tilde{\mathscr{F}} / \mathscr{F}$-measurable such that\\
	(i) $\tilde{\mathscr{F}}_{t} \supset \pi^{-1}\left(\mathscr{F}_{t}\right)$,\\
	(ii) $\mathbb{P}=\pi(\tilde{\mathbb{P}})(:=\tilde{\mathbb{P}} \circ \pi^{-1}) \quad$ and\\
	(iii) for every $X(\omega) \in \mathscr{L}_{\infty}(\Omega, \mathscr{F}, \mathbb{P})$
	$$
	\tilde{E}\left(\tilde{X}(\tilde{\omega}) \mid \tilde{\mathscr{F}}_{t}\right)=E\left(X \mid \mathscr{F}_{t}\right)(\pi \tilde{\omega}), \quad \tilde{P} \text {-a.s., }
	$$
	where we set $\tilde{X}(\tilde{\omega})=X(\pi \tilde{\omega})^{\prime} \quad$ for $\tilde{\omega} \in \tilde{\Omega}$.
\end{definition}
\begin{definition}
	\label{Definition 7.2.}  Let $(\Omega, \mathscr{F}, \mathbb{P})$ be a probability space with  filteration $\left(\mathscr{F}_{t}\right)_{t\ge 0}$. Let $\left(\Omega^{\prime}, \mathscr{F}^{\prime}, P^{\prime}\right)$ be another probability space and set
	$$
	\tilde{\Omega}=\Omega \times \Omega^{\prime}, \quad \tilde{\mathscr{F}}=\mathscr{F} \times \mathscr{F}^{\prime}, \quad \tilde{\mathbb{P}}=\mathbb{P} \times \mathbb{P}^{\prime}
	$$
	and
	$$
	\pi \tilde{\omega}=\omega \quad \text { for } \quad \tilde{\omega}=\left(\omega, \omega^{\prime}\right) \in \tilde{\Omega}
	$$
	If $(\tilde{\mathscr{F}}_{t})_{t\ge 0}$ is a filtration  on $(\tilde{\Omega}, \mathscr{F}, \tilde{\mathbb{P}})$ such that $\mathscr{F}_{t} \times \mathscr{F}^{\prime} \supset \tilde{\mathscr{F}}_{t} \supset \mathscr{F}_{t} \times\left\{\Omega^{\prime}, \phi\right\}$, then $(\tilde{\Omega}, \tilde{\mathscr{F}}, \tilde{\mathbb{P}})$ with $(\tilde{\mathscr{F}}_{t})_{t\ge 0}$ is called a standard extension of $(\Omega, \mathscr{F}, \mathbb{P})$ with filtration $(\mathscr{F}_t)_{t\ge 0}$.
\end{definition}

\bigskip

\section{The shifted projection on the Haar basis}
\label{haarproj}
\newcommand{\pro}{P}
Let $\phi:[0,1]\to\RR$ be the Haar wavelet given by 
$$
\phi(x):=\bcase 1 &\mbox{ if } 0\le x<\frac{1}{2},
\\-1 &\mbox{ if } <\frac{1}{2}\le x <1,
\ecase
$$
$\phi_{n,j}(x):=2^n \phi(2^n x-j)$, $x\in[0,1]$, the corresponding multiresolution analysis. Let $X$ be a Banach space. For $f\in L^m(0,T;X)$, $\pro_n f=\sum_{j=0}^{2^n-1}\la f,\phi_{n,j}\ra \, \phi_{n,j}$ its orthogonal projection on
the Haar basis given by  $\{\phi_{n,j}:j=0,\ldots, 2^n-1\}$.

In a second step, we shift the time intervals. In particular, we define for all
$\kappa\in \NN$ and 
for $f\in L^m(0,T;X)$    the following shifted  function  $\shiftpro _\zaehler f$ of $f$
$$
\shiftpro _\zaehler f(s)=f(s-\frac{T}{2^n}),\ \forall s\in [0,T]
$$ 
Let
\DEQSZ\label{projectiondef}
\Pro_\zaehler := \shiftpro _\zaehler\, \pro_\zaehler.
\EEQSZ

\begin{remark}
It is straightforward to verify that for any $f\in L^m(0,T;X)$, where $X$  is a Banach space, $\Pro_\zaehler f$ can be also written as follows
\begin{equation*}
\Pro_\zaehler f(s)=
	\frac{2^n}{T}\sum_{k=0}^{2^n-1}1_{[t_k,t_{k+1}]}(s)\int_{t_k}^{t_{k+1}}f(r)\, dr,
\end{equation*}
where $t^\zaehler_k=T\frac{k}{2^n}$. 
\end{remark}

We have the following properties 
\begin{lemma}
For any $m\geq 1$ and  $n\in \mathbb{N}$, the projection  is well defined from  $ L^m(0,T;X)$ to $L^m(0,T;X)$, is linear and satisfies the following inequality. 
\begin{equation}
\| \pro _\zaehler f \|_{L^m(0,T;X)}\leq \|f\|_{L^m(0,T;X)}.\label{1.1}
\end{equation}
Moreover, if $f\in\mathbb{ W}^{\alpha}_m(0,T;X)$ with $0<\alpha<\frac{1}{m}$  then the following inequality holds
\begin{equation}
\| \pro _\zaehler f-f\|_{L^m(0,T;X)}\leq\frac{T^{\alpha}}{2^{\alpha n}}\|f\|_{ \mathbb{W}^{\alpha}_m(0,T;X)}.\label{1.2}
\end{equation}
\end{lemma}
\begin{proof}
The linearity is clear. We will firstly focus on the proof of  inequality (\ref{1.1}). Let us fix $m\geq 1$, $n\in \mathbb{N}$ and $f\in L^m(0,T;X)$. 
By the
definition of the projection we get
\begin{align}
\lqq{
\| \pro_\zaehler f\|^m_{L^m(0,T;X)}=\sum_{k=0}^{2^\zaehler-1}\int _{t^\zaehler _k}^{t^\zaehler _{k+1}}|\pro_\zaehler f(s)|^m_X ds}
&
\notag\\
&=\left(\frac{2^\zaehler}{T}\right)^m\sum_{k=0}^{2^\zaehler-1}\int_{t^\zaehler _k}^{t^\zaehler _{k+1}}\left|\int _{t^\zaehler _k}^{t^\zaehler _{k+1}}	f(r)dr\right|^m_X ds\notag
\end{align}
By the H\"older inequality we derive that 
\begin{align}
&\leq\left(\frac{2^\zaehler}{T}\right)^m\left(\frac{T}{2^{n}}\right)^{m-1}\sum_{k=0}^{2^\zaehler-1}\int_{t^\zaehler _k}^{t^\zaehler _{k+1}}\int_{t^\zaehler _k}^{t^\zaehler _{k+1}}	\left|f(r)\right|^m_Xdr ds\notag\\
&=\sum_{k=0}^{2^\zaehler-1}\int_{t^\zaehler _k}^{t^\zaehler _{k+1}}	\left|f(r)\right|^m_Xdr=\int_{0}^{T}	\left|f(s)\right|^m_Xds,\notag
\end{align}
and the inequality (\ref{1.1}) follows.
Now we show the inequality (\ref{1.2}). For this aim, let us fix $0<\alpha<\frac{1}{m}$ and $f\in \mathbb{W}^{\alpha}_m(0,T;X)$. Then we have
\begin{align}
\lqq{
\| \pro_\zaehler f-f\|^m_{L^m(0,T;X)}=\sum_{k=0}^{2^\zaehler-1}\int_{t^\zaehler _k}^{t^\zaehler _{k+1}}|	\pro_\zaehler f(s)-f(s)|^m_X ds} &
\notag
\\
&=\sum_{k=0}^{2^\zaehler-1}\int_{t^\zaehler _k}^{t^\zaehler _{k+1}}\left|\frac{2^\zaehler}{T}\int_{t^\zaehler _k}^{t^\zaehler _{k+1}}	f(r)dr-f(s)\right|^m_X ds\notag
\\
&
=\sum_{k=0}^{2^\zaehler-1}\int_{t^\zaehler _k}^{t^\zaehler _{k+1}}\left|\frac{2^\zaehler}{T}\int_{t^\zaehler _k}^{t^\zaehler _{k+1}}	f(r)dr-\frac{2^\zaehler}{T}\int_{t^\zaehler _k}^{t^\zaehler _{k+1}}f(s)dr\right|^m_X ds
\notag
\\
&=\left(\frac{2^\zaehler}{T}\right)^m\sum_{k=0}^{2^\zaehler-1}\int_{t^\zaehler _k}^{t^\zaehler _{k+1}}\left|\int_{t^\zaehler _k}^{t^\zaehler _{k+1}}	\left(f(r)-f(s)\right)dr\right|^m_X ds
\notag
\\
&=\left(\frac{2^\zaehler}{T}\right)^m\sum_{k=0}^{2^\zaehler-1}\int_{t^\zaehler _k}^{t^\zaehler _{k+1}}\left|\int_{t^\zaehler _{k}}^{t^\zaehler _{k+1}}	\left(f(r)-f(s)\right)dr\right|^m_X ds
\notag
\\
&=\left(\frac{2^\zaehler}{T}\right)^m\sum_{k=0}^{2^\zaehler-1}\int_{t^\zaehler _k}^{t^\zaehler _{k+1}}\left(\int_{t^\zaehler _{k}}^{t^\zaehler _{k+1}}	\frac{|f(r)-f(s)|_X}{|r-s|^\frac{1+\alpha m}{m}}\times |r-s|^\frac{1+\alpha m}{m}dr\right)^m ds
\notag
\\
&\leq\left(\frac{2^\zaehler}{T}\right)^m\sum_{k=0}^{2^\zaehler-1}\int_{t^\zaehler _k}^{t^\zaehler _{k+1}}\int_{t^\zaehler _{k}}^{t^\zaehler _{k+1}}	\frac{|f(r)-f(s)|_X^m}{|r-s|^{1+\alpha m}} dr \left(\int_{t^\zaehler _{k}}^{t^\zaehler _{k+1}} |r-s|^\frac{1+\alpha m}{m-1}dr\right)^{m-1} ds.
\notag
\end{align}
Let us recall that we have used a H\"older inequality to obtain the previous inequality. Now, let us remark that for $k=0,1,2,3...,2^\zaehler-1$, the following identity holds
\begin{align*}
\lqq{ \left(\int_{t^\zaehler _{k}}^{t^\zaehler _{k+1}} |r-s|^\frac{1+\alpha m}{m-1}dr\right)^{m-1} \leq \left(\frac{T}{2^\zaehler}\right)^{1+\alpha m}\left(\int_{t^\zaehler _{k}}^{t^\zaehler _{k+1}} dr\right)^{m-1}}
&
 \\
 &=\left(\frac{T}{2^\zaehler}\right)^{1+\alpha m} \left(\frac{T}{2^\zaehler}\right)^{m-1}=\left(\frac{T}{2^\zaehler}\right)^{m+\alpha m}.
\end{align*}
This implies that 
\begin{align}
\lqq{ \| \pro_\zaehler f-f\|^m_{L^m(0,T;X)}
}
\\
&
\leq\left(\frac{2^\zaehler}{T}\right)^m\left(\frac{T}{2^\zaehler}\right)^{m+\alpha m}\sum_{k=0}^{2^\zaehler-1}\int_{t^\zaehler _k}^{t^\zaehler _{k+1}}\int_{t^\zaehler _{k}}^{t^\zaehler _{k+1}}	\frac{|f(r)-f(s)|_X^m}{|r-s|^{1+\alpha m}} \, dr\,ds
\notag
\\
&=\left(\frac{T}{2^\zaehler}\right)^{\alpha m}\sum_{k=0}^{2^\zaehler-1}\int_{t^\zaehler _k}^{t^\zaehler _{k+1}}\int_{t^\zaehler _{k}}^{t^\zaehler _{k+1}}	\frac{|f(r)-f(s)|_X^m}{|r-s|^{1+\alpha m}} drds\notag\\
&\leq \left(\frac{T}{2^\zaehler}\right)^{\alpha m}\|f\|^m_{ \mathbb{W}_m^{\alpha}(0,T;X)},\notag
\end{align}
and the  inequality (\ref{1.2}) follows. 
\end{proof}
Now, for the shift operator,    we can prove the following properties.
\begin{lemma}
	For any $m\geq 1$,   $n\in \mathbb{N}$ and $f\in  L^m(0,T;X)$, $\shiftpro _\zaehler f$  satisfies the following inequalities. 
\begin{equation}
	\| \shiftpro _\zaehler f\|^m_{L^m(0,T;X)}\leq \|f\|^m_{L^m(0,T;X)}+\frac{T}{2^\zaehler}|f(0)|^m_X,\label{1.1*}
\end{equation}
\begin{equation}
	\| \shiftpro _\zaehler f-f\|^m_{L^m(0,T;X)}\leq \int_0^\frac{T}{2^\zaehler}|f(0)-f(s)|^m_Xds+2^m\|f\|^m_{ L^{m}(0,T;X)}.\label{1.2*}
\end{equation}
\end{lemma}
\begin{proof}
We note that 
\begin{align*}
\lqq{
\|\shiftpro _\zaehler f\|_{L^m(0,T;X)}^m=\int_0^\frac{T}{2^\zaehler}|\shiftpro _\zaehler f(s)|_X^mds+\int_\frac{T}{2^\zaehler}^T|\shiftpro _\zaehler f(s)|_X^mds
} &
\\
&=\int_0^\frac{T}{2^\zaehler}|f(0)|_X^mds+\int_\frac{T}{2^\zaehler}^T|f(s-\frac{T}{2^\zaehler})|_X^mds\\
&=\frac{T}{2^\zaehler}|f(0)|_X^m+\int^{T-\frac{T}{2^\zaehler}}_0|f(s)|_X^mds\\
&=\frac{T}{2^\zaehler}|f(0)|_X^m+\int^{T}_0|f(s)|_X^mds,
\end{align*}
and the inequality (\ref{1.1*}) follows.

For the proof of the inequality (\ref{1.2*}), we use the inequality $(a+b)^m\leq 2^{m-1}(a^m+b^m)$, $a$, $b> 0$ to derive that
\begin{align*}
\lqq{	\|\shiftpro _\zaehler f-f\|_{L^m(0,T;X)}^m}
&
\\ &=\int_0^\frac{T}{2^\zaehler}|\shiftpro _\zaehler f(s)-f(s)|_X^mds+\int_\frac{T}{2^\zaehler}^T|\shiftpro _\zaehler f(s)-f(s)|_X^mds\\
	&= \int_0^\frac{T}{2^\zaehler}|f(0)-f(s)|_X^mds+\int_\frac{T}{2^\zaehler}^T\left|f(s-\frac{T}{2^\zaehler})-f(s)\right|_X^mds\\
	&\leq  \int_0^\frac{T}{2^\zaehler}|f(0)-f(s)|_X^mds+2^{m-1}\int_\frac{T}{2^\zaehler}^T\left|f(s-\frac{T}{2^\zaehler})\right|_X^mds+2^{m-1}\int_\frac{T}{2^\zaehler}^T\left|f(s)\right|_X^mds\\
	&= \int_0^\frac{T}{2^\zaehler}|f(0)-f(s)|_X^mds+2^{m-1}\int_\frac{T}{2^\zaehler}^T\left|f(s)\right|_X^mds+2^{m-1}\int_\frac{T}{2^\zaehler}^T\left|f(s)\right|_X^mds\\
	&= \int_0^\frac{T}{2^\zaehler}|f(0)-f(s)|_X^mds+2^{m}\int_\frac{T}{2^\zaehler}^T\left|f(s)\right|_X^mds\\
	&= \int_0^\frac{T}{2^\zaehler}|f(0)-f(s)|_X^mds+2^{m}\int_0^T\left|f(s)\right|_X^mds,
\end{align*}
and  we then derive the inequality (\ref{1.2*}).
\end{proof}
Next, for $f\in L^m(0,T;X)$    we define the following shifted Haar projection of $f$
\begin{equation*}
	\shiftpro_\zaehler f(s)=\bcase 
	f(0)\qquad \text{ if } \ s\in [0,t_1)\vspace{0.2cm}\\
	\frac{2^\zaehler}{T}\int_{t_k}^{t_{k+1}}f(r-\frac{T}{2^{n}})dr\qquad \text{ if } s\in [t_k,t_{k+1}), \ k=1,2,3,...,2^\zaehler-1.
	\ecase
\end{equation*}
We start by remark that this projection can be rewritten as follows:
\begin{equation*}
	\Pro_\zaehler(f)(s)=\bcase 
	f(0)\qquad \text{ if } \ s\in [0,t_1)\vspace{0.2cm}\\
	\frac{2^\zaehler}{T}\int_{t_{k-1}}^{t_{k}}f(r)dr\qquad \text{ if } s\in [t_k,t_{k+1}), \ k=1,2,3,...,2^\zaehler-1.
	\ecase
\end{equation*}
We have the following properties 
\begin{lemma}\label{convergenceproj}
	For any $m\geq 1$ and  $n\in \mathbb{N}$, the projection  $$\Pro _\zaehler: L^m(0,T;X)\longrightarrow L^m(0,T;X)$$ is well defined, is linear and satisfies the following inequality. 
	\begin{equation}
		\| \widehat \Pro_\zaehler(f)\|^m_{L^m(0,T;X)}\leq \|f\|^m_{L^m(0,T;X)}+\frac{T}{2^\zaehler}|f(0)|^m_X\qquad \forall f\in L^m(0,T;X).\label{1.4}
	\end{equation}
	Moreover, if $f\in W^{\alpha,m}(0,T;X)$ with $0<\alpha<\frac{1}{m}$  then the following inequality holds
	\begin{equation}
		\| \widehat\Pro_\zaehler(f)-f\|^m_{L^m(0,T;X)}\leq \int_0^\frac{T}{2^\zaehler}|f(0)-f(s)|^m_Xds+\left(\frac{T}{2^{ \kappa}}\right)^{\alpha m}\|f\|^m_{ W^{\alpha}_m(0,T;X)}.\label{1.5}
	\end{equation}
\end{lemma}
\begin{proof}
	The linearity is clear. We will firstly focus on the proof of  inequality (\ref{1.4}). Then let us fix $m\geq 1$, $\kappa\in \mathbb{N}$ and $f\in L^m(0,T;X)$. By the H\"older inequality we derive that 
	\begin{align}
\lqq{		\| \widehat\Pro_\zaehler(f)\|^m_{L^m(0,T;X)}=\sum_{k=0}^{2^\zaehler-1}\int_{t_k}^{t_{k+1}}|	\Pro_\zaehler(f)(s)|^m_X ds}
&
\notag\\
		&=\frac{T}{2^\zaehler}\|f(0)\|^m_X+\left(\frac{2^\zaehler}{T}\right)^m\sum_{k=1}^{2^\zaehler-1}\int_{t_k}^{t_{k+1}}\left|\int_{t_{k-1}}^{t_{k}}	f(r)dr\right|^m_X ds\notag\\
		&\leq\frac{T}{2^\zaehler}\|f(0)\|^m_X+\left(\frac{2^\zaehler}{T}\right)^m\left(\frac{T}{2^{\zaehler}}\right)^{m-1}\sum_{k=1}^{2^\zaehler-1}\int_{t_k}^{t_{k+1}}\int_{t_{k-1}}^{t_{k}}	\left|f(r)\right|^m_Xdr ds\notag\\
		&=\frac{T}{2^\zaehler}\|f(0)\|^m_X+\int_{0}^{T-\frac{T}{2^\zaehler}}	\left|f(s)\right|^m_Xds \notag\\
		&\leq\frac{T}{2^\zaehler}\|f(0)\|^m_X+\int_{0}^{T}	\left|f(s)\right|^m_Xds,\notag
	\end{align}
	and the inequality (\ref{1.4}) follows.
	Now we are going to prove the inequality (\ref{1.5}). For this aim, let us fix $0<\alpha<\frac{1}{m}$ and $f\in W^{\alpha,m}(0,T;X)$. Then we have
	\begin{align}
\lqq{		\| \Pro _\zaehler(f)-f\|^m_{L^m(0,T;X)}=\sum_{k=0}^{2^\zaehler-1}\int_{t_k}^{t_{k+1}}|	\Pro _\zaehler(f)(s)-f(s)|^m_X ds}
&
\notag\\
		&=\int_0^\frac{T}{2^\zaehler}|f(0)-f(s)|^m_Xds\notag\\
		&\qquad+\sum_{k=1}^{2^\zaehler-1}\int_{t_k}^{t_{k+1}}\left|\frac{2^\zaehler}{T}\int_{t_{k-1}}^{t_{k}}	f(r)dr-f(s)\right|^m_X ds\label{1.3}\\
		&=\int_0^\frac{T}{2^\zaehler}|f(0)-f(s)|^m_Xds+I.\notag
	\end{align}
We note that 
\begin{align}
	I&=\left(\frac{2^\zaehler}{T}\right)^m\sum_{k=1}^{2^\zaehler-1}\int_{t_k}^{t_{k+1}}\left|\int_{t_{k-1}}^{t_{k}}	\left(f(r)-f(s)\right)dr\right|^m_X ds\notag\\
	&\leq \left(\frac{T}{2^\zaehler}\right)^{\alpha m}\sum_{k=1}^{2^\zaehler-1}\int_{t_k}^{t_{k+1}}\int_{t_{k-1}}^{t_{k}}	\frac{|f(r)-f(s)|_X^m}{|r-s|^{1+\alpha m}} drds\notag\\
	&= \left(\frac{T}{2^\zaehler}\right)^{\alpha m}\int_{\frac{T}{2^\zaehler}}^{T}\int_{0}^{T-\frac{T}{2^\zaehler}}	\frac{|f(r)-f(s)|_X^m}{|r-s|^{1+\alpha m}} drds\notag\\
	&\leq \left(\frac{T}{2^\zaehler}\right)^{\alpha m}\|f\|^m_{ W^{\alpha,m}(0,T;X)}.\notag
\end{align}
Using this last inequality, we infer from the inequality (\ref{1.3}) that 
\begin{align}
	\| \Pro _\zaehler(f)-f\|^m_{L^m(0,T;X)}\leq \int_0^\frac{T}{2^\zaehler}|f(0)-f(s)|^m_Xds+\left(\frac{T}{2^\zaehler}\right)^{\alpha m}\|f\|^m_{ W^{\alpha}_m(0,T;X)},\notag
\end{align}
and the  inequality (\ref{1.5}) follows. 
\end{proof}

\begin{lemma}\label{boundinw}
	For any $m\geq 1$ and  $n\in \mathbb{N}$, the projection  $$\Pro _\zaehler:\mathbb{W}^\alpha_m(0,T;X)\longrightarrow  \mathbb{W}^\alpha_m(0,T;X)$$ is well defined, is linear and satisfies the following inequality. 
	\begin{equation}
		\| \widehat \Pro_\zaehler(f)\|^m_{\mathbb{W}^\alpha_m(0,T;X)}\leq \|f\|^m_{\mathbb{W}^\alpha_m(0,T;X)}+\frac{T}{2^\zaehler}|f(0)|^m_X\qquad \forall f\in L^m(0,T;X).\label{1.4a}
	\end{equation}
\end{lemma}
\begin{proof}
	Let us fix $m\geq 1$, $\kappa\in \mathbb{N}$ and $f\in \mathbb{W}^\alpha_m(0,T;X)$. By the definition of the space $\mathbb{W}^\alpha_m(0,T;X)$ we can write
	\begin{align}\label{bevor}
&	\| \widehat\Pro_\zaehler(f)\|^m_{\mathbb{W}^\alpha_m(0,T;X)}\le \frac 1\tau \sum_{k,l=0}^{2^\zaehler-1}\int_{t^\kappa_k}^{t^\kappa_{k+1}}\int_{t^\kappa_l}^{t^\kappa_{l+1}}\frac {|	\Pro_\zaehler(f)(s)-\Pro_\zaehler(f)(t)|^m_X }{|t-s|^{\alpha m+1}} ds\, dt.
\end{align}
Let us consider the inner part ($\tau=T2^{-\kappa}$ and $m'$ conjugate to $m$)
\DEQS
\lqq{ 
\frac { \lk|	\Pro_\zaehler(f)(s)-\Pro_\zaehler(f)(t)\rk|^m_X }
{|t-s|^{\alpha m+1}}
} &&
\\
&\le &\frac 1{\tau^m}
\frac 
{  \lk|	\int_{t^\kappa_k}^{t^\kappa_{k+1}} f(s)-f(s-\tau(l-k))\, ds\rk|^m_X } {|t-s|^{\alpha m+1}}
\\
&\le &\frac 1{\tau^m}\frac 
{  \int_{t^\kappa_k}^{t^\kappa_{k+1}}\lk|	 f(s)-f(s-\tau(l-k))\rk|^m_X\, ds \, \tau^\frac m{m'} } {(\tau|l-k|)^{\alpha m+1}}
\\
&\le &\frac 1{\tau^m}\frac { \tau^\frac m{m'} } {(\tau|l-k|)^{\alpha m+1}}
 \int_{t^\kappa_k}^{t^\kappa_{k+1}} {\lk|	 f(s)-f(s-\tau(l-k))\rk|^m_X}
 \frac {(\tau|l-k|)^{\alpha m+1} } {(\tau|l-k|)^{\alpha m +1}}\, ds \,
\\
&\le &\frac { \tau^\frac m{m'} }{\tau^m}
 \int_{t^\kappa_k}^{t^\kappa_{k+1}} \frac {\lk|	 f(s)-f(s-\tau(l-k))\rk|^m_X}
 {(\tau|l-k|)^{\alpha m+1} } \, ds \,
 \\
&\le &\frac 1\tau 
 \int_{t^\kappa_k}^{t^\kappa_{k+1}} \frac {\lk|	 f(s)-f(s-\tau(l-k))\rk|^m_X}
 {(\tau|l-k|)^{\alpha m+1} } \, ds. \,
\EEQS
Substituting above in \eqref{bevor}, we get
	\begin{align*}
&	\| \widehat\Pro_\zaehler(f)\|^m_{\mathbb{W}^\alpha_m(0,T;X)}\le \frac 1\tau \sum_{k,l=0}^{2^\zaehler-1}\int_{t^\kappa_k}^{t^\kappa_{k+1}}\int_{t^\kappa_l}^{t^\kappa_{l+1}}
 \int_{t^\kappa_k}^{t^\kappa_{k+1}} \frac {\lk|	 f(s)-f(s-\tau|l-k|)\rk|^m_X}
 {(\tau|l-k|)^{\alpha m+1} } \, ds \,dr\, dt
 \\
 &\le \sum_{k,l=0}^{2^\zaehler-1}\int_{t^\kappa_l}^{t^\kappa_{l+1}}
 \int_{t^\kappa_k}^{t^\kappa_{k+1}} \frac {\lk|	 f(s)-f(s-\tau|l-k|)\rk|^m_X}
 {(\tau|l-k|)^{\alpha m+1} } \, ds \, dt \le C\| \widehat\Pro_\zaehler(f)\|^m_{\mathbb{W}^\alpha_m(0,T;X)}.
 \end{align*}
\end{proof}

\section{Function spaces and the Aubin-Lions-Simon compactness theorem}\label{dbouley-space}

Let $B$ be a separable Banach space, $0\le c<d<\infty$.
Let $C^{(\beta)}_b(c,d;B)$ denote a set of all continuous and bounded functions $u:[c,d]\to B$ such that
$$
\Vert u\Vert_{C_b^{\beta}(c,d;B)} 
:=\sup_{c\le t\le d} \vert u(t)\vert_B +\sup_{\substack{c\le s,t\le d\\ t\not= s}} \frac{ \vert u(t)-u(s)\vert_{B}}{\vert t-s\vert^\beta},
$$
is finite. The space $C_b^{(\beta)}(c,d;E)$  endowed with the norm 
$\Vert \cdot\Vert_{C_b^{\beta}(c,d;B)}$ is a Banach space.
Let
 $$ L^p(c,d; ;B)=\left\{ u:[c,d)\to B: u~\mbox{ measurable and } \int_{[c,d)} \vert u(t)\vert_{B}^p\, dt
 <\infty\right\}.$$
In addition, for $1< p < \infty$  let $W^1_p(\CO)$ be the standard Sobolev space defined by (compare \cite[p.\ 263]{Brezis})
\begin{align*} 
\lqq{W^{1}_p(\CO)
:=\left\{ u\in L^p(\CO)\mid \exists g_1,\cdots,g_d\in L^p(\CO)\mbox{ such that }\phantom{\bigg\vert}\right.}
\\
&{}\left. \int_\CO u(x)\frac{\partial \phi(x)}{\partial x_i} \, dx =-\int_\CO g_i(x)\phi(x)\, dx \quad \forall \phi\in C^\infty_0(\Ocal), \forall i=1,\ldots ,d \right\}
\end{align*} 
equipped with norm
$$
\vert u\vert_{W^1_p}:=\vert u\vert_{L^p}+\sum_{j=1}^d\left\vert\frac{\partial u}{\partial x_j}\right\vert_{L^p},\quad u\in W^1_p(\CO).
$$
Given an integer $m\ge2$ and a real number $1\le p<\infty$, we define by induction the space
\begin{align*} 
W^{m}_p(\CO) :=\left\{ u\in W^{m-1}_p(\CO)\mid D u\in W^{m-1}_p(\CO) \right\}
\end{align*} 
equipped with norm
$$
\vert u\vert_{W^m_p}:=\vert u\vert_{L^p}+\sum_{\alpha =1}^m\left\vert D_\alpha u\right\vert_{L^p},\quad u\in W^m_p(\CO).
$$
Let $H_2^m(\CO):=W^m_2(\CO)$, and for $\rho\in(0,1)$ let $H_2^\rho (\CO)$ be the real interpolation  space given by $H^\rho _2(\CO):=(L^2(\CO),H^1_2(\CO))_{\rho ,2}$.
In addition, let $H^{-1}_2(\CO)$ be the dual space of $H^1_2(\CO)$ and for $\rho \in(0,1)$ let $H^{-\rho }_2(\CO)$ be the  real interpolation  space given by $H^{-\rho }_2(\CO):=(L^2(\CO),H^{-1}_2(\CO))_{1-\rho ,2}$.
Note, by Theorem 3.7.1 \cite{bergh}, $H^{-\rho }_2(\CO)$ is dual to $H^\rho _2(\CO)$, $\rho \in(0,1)$. Furthermore, we have  $(H^{-\rho }_2(\CO),H^{\rho }_2(\CO))_{\frac 12,2}=L^2(\CO)$ and $(H^\alpha_2(\CO),H^\beta_2(\CO))_{\rho,2}=H^\theta_2(\CO)$ for $\theta=\alpha(1-\rho)+\beta\rho$, $\rho\in(0,1)$ and $\vert\alpha\vert,\vert\beta\vert\le 1$.

Since we need it to tackle the compactness, let us introduce the following space.
Given $p\in (1,\infty)$, $\alpha\in(0,1)$, let $\WW ^ {\alpha}_p (I;B)$ be the Sobolev space
of all $u\in L^p(0,\infty;B)$ such that
$$
\int_I  \int_{I\cap [t,t+1]}   \;\frac{\vert u(t)-u(s)\vert_B ^ p}{ \vert t-s\vert ^ {1+\alpha p}}\,ds\,dt<\infty;
$$
equipped with the norm
$$
\left\Vert u\right\Vert_{ \WW^ {\alpha}_p(I ;B)}:=\left( \int_I  \int_{I\cap [t,t+1]}  \;\frac{\vert u(t)-u(s)\vert_B ^ p}{ \vert t-s\vert ^ {1+\alpha p}}\,ds\,dt\right) ^ \frac 1p.
$$

\begin{theorem}\label{th-gutman}
Let $B_0\subset B\subset B_1$ be Banach spaces, $B_0$ and $B_1$ reflexive, with compact embedding of $B_0$ to $B$. Let $p\in(1,\infty)$ and $\alpha\in(0,1)$ be given. Let $X$ be the space
$$
X=L^p(0,T;B_0)\cap \WW ^{\alpha}_p (0,T;B_1).
$$
Then the embedding of $X$ to $L^p(0,T;B)$
is compact.
\end{theorem}
\begin{proof}
See \cite[p. 86, Corollary 5]{Simon1986}.
\end{proof}

\section*{Acknowledgements} 

\noindent E.H. gratefully acknowledges the supported by the Austrian Science Foundation, Project number: P34681. M.A.H. thanks the Austrian Academy of Sciences for the funding of the academic visit at Montanuniversit\"at Leoben in June / July 2022 in the framework of a JESH project 2019. 
The research of M.A.H. was supported by the research project INV-2023-162-2850 of Facultad de Ciencias at Universidad de los Andes.

\renewcommand{\emph}{\textsl}

\end{document}